\documentclass[10pt]{amsart}

\usepackage{amsmath,amssymb, amsfonts}
\usepackage[all]{xy}
\usepackage{color}
\usepackage{verbatim}
\usepackage{enumerate}
\usepackage{tikz-cd}
\usepackage{MnSymbol}

\DeclareMathAlphabet{\mathpzc}{OT1}{pzc}{m}{it}
\newcommand{\adj}[4]{#1\negmedspace: #2\rightleftarrows #3:\negmedspace #4}
\newtheorem{thm}{Theorem}[section]
\newtheorem{cor}[thm]{Corollary}
\newtheorem{example}[thm]{Example}
\newtheorem{lem}[thm]{Lemma}

\newtheorem{prop}[thm]{Proposition}
\theoremstyle{definition}
\newtheorem{defn}[thm]{Definition}
\newtheorem{rem}[thm]{Remark}
\numberwithin{equation}{section}
\newtheorem{notation}[thm]{Notation}

\DeclareFontFamily{U}{rsf}{} \DeclareFontShape{U}{rsf}{m}{n}{
  <5> <6> rsfs5 <7> <8> <9> rsfs7 <10->  rsfs10}{}
\DeclareMathAlphabet{\mathscr}{U}{rsf}{m}{n}

\newcommand*{\defeq}{\mathrel{\vcenter{\baselineskip0.5ex \lineskiplimit0pt
                     \hbox{\scriptsize.}\hbox{\scriptsize.}}}%
                     =}

\renewcommand{\imath}{\sqrt{-1}}

\DeclareMathOperator{\colim}{colim}

\def\llp{\mathrel{\ooalign{\hss$\square$\hss\cr$\diagup$}}}

\numberwithin{equation}{section}

\begin{document}

\title[]{Flat model structures for accessible exact categories}
\author{Jack Kelly}\thanks{}%
\address{Jack Kelly,
Lincoln College,
Turl Street,
Oxford,
OX1 3DR}

\dedicatory{}
\subjclass{}%
\thanks{}
\keywords{}%

\begin{abstract}
We develop techniques for constructing model structures on chain complexes valued in accessible exact categories, and apply this to show that for a closed symmetric monoidal, locally presentable exact category $\mathpzc{E}$ with exact filtered colimits and enough flat objects, the flat cotorsion pair on $\mathpzc{E}$ induces an exact model structure on $\mathrm{Ch}(\mathpzc{E})$. Further we show that when enriched over $\mathbb{Q}$ such categories furnish convenient settings for homotopical algebra - in particular that they are Homotopical Algebra Contexts, and admit powerful Koszul duality theorems. As an example, we consider categories of sheaves valued in monoidal locally presentable exact categories. 
\noindent 
\end{abstract}
\maketitle 
\tableofcontents

\section{Introduction}

Let $\mathpzc{E}$ be a monoidal elementary exact category - that is $\mathpzc{E}$ has a generating set consisting of compact projective objects, and is equipped with a closed symmetric monoidal structure such that projectives are flat, the monoidal unit is projective, and the tensor product of two projectives is projective. In \cite{kelly2016homotopy} we showed that the categories $\mathrm{Ch}_{\ge0}(\mathpzc{E})$ and $\mathrm{Ch}(\mathpzc{E})$ have projective model structures, and admit a very rich theory of homotopical algebra. Precisely, when $\mathpzc{E}$ is enriched over $\mathbb{Q}$ they are homotopical algebra contexts in the sense of To\"{e}n-Vezzosi. This was later expanded upon in \cite{kelly2019koszul}, where it is shown that deep Koszul duality theorems hold in such categories. In particular, the fact that the category $\mathrm{Ind(Ban_{k})}$ for $k$ a Banach field, or more generally a Banach ring, is a monoidal elementary exact category (in fact it is quasi-abelian) means that it as a convenient setting for derived analytic geometry. In applications to analytic geometry, for example \cite{reconstruction}, one is led to consider the homological algebra of categories of sheaves on a space $X$ valued in $\mathrm{Ind(Ban_{k})}$. Although the category $\mathrm{Shv}(X,\mathrm{Ind(Ban_{k})})$ of sheaves on $X$ is quasi-abelian, it will not in general have enough projectives. However as we will show in this work it has enough flat objects, and this will be enough to endow $\mathrm{Ch}(\mathrm{Shv}(X,\mathrm{Ind(Ban_{k})}))$ with flat model structures, so that they become homotopical algebra contexts. More generally we will prove the following.

\begin{thm}
Let $\mathpzc{E}$ be a purely $\lambda$-accessible closed symmetric monoidal exact category with a generator such that 
\begin{enumerate}
\item
$\lambda$-pure monomorphisms are admissible
\item
colimits of transfinite sequences of acyclic complexes are acyclic (i.e. $\mathpzc{E}$ is weakly elementary in the terminology of the present paper)
\item
$\mathpzc{E}$ has enough flat objects.
\end{enumerate}
Then there is a model structure on both $\mathrm{Ch}(\mathpzc{E})$ and $\mathrm{Ch}_{\ge0}(\mathpzc{E})$ where
\begin{enumerate}
\item
the weak equivalences are the quasi-isomorphisms.
\item
the cofibrations are the degree-wise admissible monomorphisms whose cokernels are $dg$-flat.
\end{enumerate}
If $\mathpzc{E}$ is $\lambda$-presentable and $\mathpzc{E}$ is enriched over $\mathbb{Q}$, then with these model structures $\mathrm{Ch}(\mathpzc{E})$ and $\mathrm{Ch}_{\ge0}(\mathpzc{E})$ are homotopical algebra contexts.
\end{thm}

The model structure is an exact model structure in the sense of \cite{hovey}, i.e. it arises from a Hovey triple $(\mathfrak{C},\mathfrak{W},\mathfrak{F})$ where $\mathfrak{C}$, $\mathfrak{W}$, and $\mathfrak{F}$ are the classes of cofibrant, acyclic, and fibrant objects respectively. This Hovey triple is cooked up from the flat cotorsion pair $(\mathrm{Flat},\mathrm{Flat}^{\perp})$ on $\mathpzc{E}$ using the recipe of Gillespie \cite{gillespie2006flat}. Although Ding and Yang \cite{yang2014question} showed that Gillespie's method always produces Hovey triple on $\mathrm{Ch}(\mathpzc{E})$ from a complete cotorsion pair on $(\mathfrak{L},\mathfrak{R})$ in the case that $\mathpzc{E}$ is abelian, there is no guarantee that this will work for arbitrary exact $\mathpzc{E}$. 

We also prove a version of this theorem for $\mathrm{Ch}_{\le0}(\mathpzc{E})$. In Subsection \ref{subsubsec:DoldKan} we consider the simplicial objects $\mathrm{s}\mathpzc{E}$ and cosimplicial objects $\mathrm{cs}\mathpzc{E}$, and prove Dold-Kan equivalences.

To connect with very recent work of \cite{estrada2023k}, we show that their methods generalise easily to prove the existence of the $K$-flat model structure, and the corresponding recollement.

\begin{thm}[Theorem \ref{thm:Kflat}]
Let $\mathpzc{E}$ be a purely locally $\lambda$-presentable closed symmetric monoidal exact category which is weakly elementary, has enough flat objects, and has a flat tensor unit. Let $K\mathcal{F}$ denote the class of $K$-flat objects. Then 
\begin{enumerate}
\item
Let $\mathfrak{W}$ denote the class of quasi-isomorphisms in $\mathrm{Ch}(\mathpzc{E})$. Then $(K\mathcal{F},\mathfrak{W},(K\mathcal{F}\cap\mathfrak{W})^{\perp_{\otimes}})$ is a Hovey triple on $\mathrm{Ch}(\mathpzc{E}_{\otimes})$, where $\mathpzc{E}_{\otimes}$ denotes the $\otimes$-pure exact structure. The induced model structure on $\mathrm{Ch}(\mathpzc{E}_{\otimes})$ is monoidal and satisfies the monoid axiom.
\item
$(\widetilde{dg\mathrm{Flat}},\mathcal{W},\widetilde{\mathrm{Flat}}^{\perp})$ is a Hovey triple on $\mathrm{Ch}(\mathpzc{E})$. The induced model structure on $\mathrm{Ch}(\mathpzc{E})$ is monoidal and satisfies the monoid axiom. Moreover it is left Quillen equivalent to the one from Part i) through the identity functor.
\item
$X$ is acyclic and $K$-flat if and only if $X$ is $\otimes$-pure acyclic. In particular there is a recollement.
\begin{displaymath}
\begin{tikzcd}
\mathbf{K}(\mathpzc{E})\big\slash K\mathcal{F}\arrow[r,bend right=25,shift right=0.2ex]\arrow[r,bend left=25,shift left=0.2ex]  &\mathbf{Ch}(\mathpzc{E}_{\otimes})\arrow[l,"\perp" {inner sep=0.3ex,rotate=180},
    "\perp"' {inner sep=0.3ex,rotate=180}] 
\arrow[l,"\perp" {inner sep=0.3ex,rotate=180},
    "\perp"' {inner sep=0.3ex,rotate=180}]\arrow[r,bend right=25,shift right=0.2ex]\arrow[r,bend left=25,shift left=0.2ex]  & 
    \mathbf{Ch}(\mathpzc{E}) \arrow[l,"\perp" {inner sep=0.3ex,rotate=180},
    "\perp"' {inner sep=0.3ex,rotate=180}]& \\
    \end{tikzcd}
    \end{displaymath}
\end{enumerate}
\end{thm}

The techniques we use in the present paper were pioneered in particular  by Estrada, Gillespie, Saorin, \v{S}t'ov\'{\i}\v{c}ek, and others (\cite{Gillespie2}, \cite{gillespie2006flat},\cite{estrada2014derived}, \cite{saorin2011exact}, \cite{vst2012exact}). The central concept here is that of \textit{deconstructibility}, namely whether given a class of objects $\mathpzc{A}$, there is a \textit{set} of objects $\mathpzc{S}$ such that $\mathpzc{A}$ is precisely the class of transfinite extensions of objects in $\mathpzc{S}$. After slightly refining the definition of deconstructibility from \cite{vst2012exact}, we prove the following

\begin{lem}[Lemma \ref{lem:accdec}]
Let $\mathpzc{E}$ be a purely $\lambda$-accessible exact category with a generator $G$. Let $\mathpzc{A}$ a class of objects in $\mathpzc{E}$ such that transfinite extensions by $\lambda$-pure monomorphisms of objects in $\mathpzc{A}$ exist and are in $\mathpzc{A}$.  Suppose further that $\mathpzc{A}$ is strongly $\lambda$-pure subobject stable, and that transfinite compositions of admissible monomorphisms with cokernel in $\mathpzc{A}$ are admissible. Then $\mathpzc{A}$ is presentably deconstructible in itself relative to $\mathbf{AdMon}$.
\end{lem}
The proof can be seen as a significant generalisation of \cite{Gillespie2} Proposition 4.9. As in \cite{vst2012exact}, the main utility of this theorem is to show that $(\mathpzc{A},\mathpzc{A}^{\perp})$ is a complete cotorsion pair when $\mathpzc{A}$ also contains a generator.

We also analyse homotopical algebra using the flat model structure. In particular we prove the following.

 \begin{thm}[Corollary \ref{cor:modelstruturealg}]
 Let $\mathpzc{E}$ be a closed symmetric monoidal purely locally presetanble exact category with enough flat object which is weakly elementary. Let $\mathpzc{M}$ be any of the model categories $\mathrm{Ch}(\mathpzc{E}),\mathrm{Ch}_{\ge0}(\mathpzc{E}),\mathrm{Ch}_{\le0}(\mathpzc{E}),\mathrm{s}\mathpzc{E}$, or $\mathrm{cs}\mathpzc{E}$ equipped with the flat model structure. Let $\mathpzc{P}$ be any operad in $\mathpzc{M}$. Then the transferred model structure exists on $\mathrm{Alg}_{\mathfrak{P}}(\mathpzc{M})$. If $\mathpzc{E}$ is enriched over $\mathbb{Q}$ then for any symmetric operad $\mathfrak{P}$ the transferred model structure exists on $\mathrm{Alg}_{\mathfrak{P}}(\mathpzc{M})$.
\end{thm}

Specialising to the commutative operad, we prove the following.

\begin{thm}
 Let $\mathpzc{E}$ be a closed symmetric monoidal purely locally presentable weakly elementary exact category with enough flat objects, enriched over $\mathbb{Q}$. Let $\mathpzc{M}$ be any of the model categories $\mathrm{Ch}(\mathpzc{E})$ or $\mathrm{Ch}_{\ge0}(\mathpzc{E})$. Then $\mathpzc{M}$ is a HA context in the sense of \cite{toen2004homotopical}.
\end{thm}

As we explain, very strong Koszul duality results also hold in this setting. 

\subsection{Layout of the Paper}

The strucure of the paper is as follows. In Section \ref{sec:gencon} we give our notation and conventions for this paper. In particular we recall some terminology and strandard results from the theories of monoidal model categories and exact categories.

In Section \ref{sec:accexact} our work begins in earnerst. We recall some general facts about locally presentable categories, as well as the $\lambda$-pure exact structure on a $\lambda$-accessible category. We provide some refined concepts of deconstructiblity of classes of objects, and explain how one can use these to apply the small object argument in the construction of pre-covering classes and cotorsion pairs. We define purely $\lambda$-accessible exact categories, and prove deconstructibility results therein. 

In Section \ref{sec:exactmodelstructures} we begin by recalling the correspondence between exact weak factorisation systems/ exact model structures and cotorsion pairs/ Hovey triples. We use some of our deconstructibility results from Section \ref{sec:accexact} to construct and modify model structures on exact categories. We also establish some basic properties of exact model structures, and explain how to construct model structures on chain complexes. We also consider monoidal model structures on chain complexes, and their relation to so-called monoidally compatible cotorsion pairs. As an application, we show that often one can extend monoidal structures on an exact category $\mathpzc{E}$ to monoidal structures on the left heart $\mathrm{LH}(\mathpzc{E})$. We conclude this section with discussion of the Dold-Kan correspondences.

In Section \ref{ref:modelstructuresonacc} we again specialise to accessible exact categories. We prove general results establishing when Gillespie's method for producing model structures on $\mathrm{Ch}(\mathpzc{E})$ from cotorsion pairs on $\mathpzc{E}$ works. As an application we show that for $\mathpzc{E}$ a purely $\lambda$-presentable exact category with strongly exact filtered colimits, its left heart is Grothendieck abelian. We also establish the existence of certain injective cotorsion pairs in the sense of \cite{MR3459032}, and use these to give an $(\infty,1)$-categorical formulation of results of \cite{MR3459032} concerning recollements. 

In Section \ref{sec:flatKflat} we finally establish the existence of the flat and $K$-flat model structures, and generalise the recollement of \cite{estrada2023k}.

In Section \ref{sec:homotopicalalgebra} we study homotopical algebra in exact categories. We give minor generalisation of existence theorems for model structures on algebras over operads from \cite{kelly2016homotopy}. We also estlabish when model structures on chain complexes arising from cotorsion pairs give rise to HA contexts. Finally, for such model structures we also explain how \cite{kelly2019koszul} can be used to prove srong Koszul duality results. 

Finally in Section \ref{sec:sheaves} we analyse in detail our main motivating example: the category of sheaves valued in an exact category. We prove the existence and functoriality of the flat model structure on complexes of sheaves, and explain how to construct three- and six-functor formalisms. As a consequence, we explain how to generalise results of \cite{spaltenstein} concerning the six operations.

\subsection{Acknoweldgements}

The author would like to thank James Gillespie for discussions related to the work in this paper. 

\section{Generalities and Conventions}\label{sec:gencon}

In this first section we mainly introduce our conventions for the main platers of this article - namely monoidal model categories and exact categories. 

\subsection{Notation and Conventions}

Throughout this work we will use the following notation.
\begin{itemize}
\item
$1$-categories will be denoted using the mathpzc font $\mathpzc{C},\mathpzc{D},\mathpzc{E}$, etc. In particular we denote by $\mathpzc{Ab}$ the category of abelian groups and ${}_{\mathbb{Q}}\mathpzc{Vect}$ the category of $\mathbb{Q}$-vector spaces. If $\mathpzc{M}$ is a model category, or a category with weak equivalences, its associated $(\infty,1)$-category will be denoted $\textbf{M}$.
\item 
Operads will be denoted using capital fractal letters $\mathfrak{C},\mathfrak{P}$, etc. Algebras over an operad will generally be denoted using  capital letters $X,Y$, etc. The category of algebras over an operad will be denoted $\mathrm{Alg}_{\mathfrak{P}}$.
\item
We denote the operads for unital associative algebras, unital commutative algebras, non-unital commutative algebras, and Lie algebras by $\mathfrak{Ass},\mathfrak{Comm},\mathfrak{Comm}^{nu}$, and $\mathfrak{Lie}$ respectively. 
\item
 For the operad $\mathfrak{Ass},\mathfrak{Comm},\mathfrak{Lie}$ we will denote the corresponding free algebras by $T(V),S(V)$, and $L(V)$ respectively. We also denote by $U(L)$ the universal enveloping algebra of a Lie algebra $L$.
\item
Unless stated otherwise, the unit in a monoidal category will be denoted by $k$, the tensor functor by $\otimes$, and for a closed monoidal category the internal hom functor will be denoted by $\underline{\mathrm{Hom}}$. Monoidal categories will always be assumed to be symmetric, with symmetric braiding $\sigma$.
\item
Filtered colimits will be denoted by $\lim_{\rightarrow}$. Projective limits will be denoted $\lim_{\leftarrow}$.
\item
The first infinite ordinal will be denoted $\aleph_{0}$.
\end{itemize}

Let us now introduce some conventions for chain complexes. We will use homological grading.

\begin{defn}\label{chaincat}
A \textit{chain complex} in a pre-additive category $\mathpzc{E}$ is a sequence

\begin{displaymath}
\xymatrix{K_{\bullet}=\ldots\ar[r] & K_{n}\ar[r]^{d_{n}} & K_{n-1}\ar[r]^{d_{n-1}} & K_{n-2}\ar[r] &\ldots}
\end{displaymath}

where the $K_{i}$ are objects and the $d_{i}$ are morphisms such that $d_{n-1}\circ d_{n}=0$. The morphisms are called \textit{differentials}. A \textit{morphism of chain complexes} $f_{\bullet}:K_{\bullet}\rightarrow L_{\bullet}$ is a collection of morphisms $f_{n}:K_{n}\rightarrow L_{n}$ such that the following diagram commutes for each $n$:

\begin{displaymath}
\xymatrix{\ldots\ar[r] & K_{n+1}\ar[d]_{f_{n+1}}\ar[r]^{d^{K}_{n+1}} & K_{n}\ar[d]_{f^{n}}\ar[r]^{d^{K}_{n}} & K_{n-1}\ar[d]_{f^{n-1}}\ar[r] &\ldots\\
\ldots\ar[r] & L_{n+1}\ar[r]^{d^{L}_{n+1}} & L_{n}\ar[r]^{d^{L}_{n}} & L_{n-1}\ar[r] &\ldots}
\end{displaymath}
\end{defn}

The category whose objects are chain complexes and whose morphisms are as described above is called the \textit{category of chain complexes in} $\mathpzc{E}$, denoted $\mathrm{Ch}(\mathpzc{E})$. We also define $\mathrm{Ch}_{\ge0}(\mathpzc{E})$ to be the full subcategory of $\mathrm{Ch}(\mathpzc{E})$ on complexes $A_{\bullet}$ such that $A_{n}=0$ for $n<0$, $\mathrm{Ch}_{\le0}(\mathpzc{E})$ to be the full subcategory of $\mathrm{Ch}(\mathpzc{E})$ on complexes $A_{\bullet}$ such that $A_{n}=0$ for $n>0$, $\mathrm{Ch}_{+}(\mathpzc{E})$, the full subcategory of chain complexes $A_{\bullet}$ such that $A_{n}=0$ for $n<<0$, $\mathrm{Ch}_{-}(\mathpzc{E})$, the full subcategory of chain complexes $A_{\bullet}$ such that $A_{n}=0$ for $n>>0$ and $\mathrm{Ch}_{b}(\mathpzc{E})$  to be the full subcategory of $\mathrm{Ch}(\mathpzc{E})$ on complexes $A_{\bullet}$ such that $A_{n}\neq 0$ for only finitely many $n$. A lot of the statements in the rest of this document apply to several of these categories at once. In such cases we will write $\mathrm{Ch}_{*}(\mathpzc{E})$, and specify that $*$ can be any element of some subset of $\{\ge0,\le0,+,-,b,\emptyset\}$, where by definition $Ch_{\emptyset}(\mathpzc{E})=\mathrm{Ch}(\mathpzc{E})$.\newline
\\
We will  frequently use the following special chain complexes.

\begin{defn}
If $E$ is an object of a pointed category $\mathpzc{E}$ we let $S^{n}(E)\in \mathrm{Ch}(\mathpzc{E})$ be the complex whose $n$th entry is $E$,  with all other entries being $0$. We also denote by $D^{n}(E)\in \mathrm{Ch}(\mathpzc{E})$ the complex whose $n$th and $(n-1)$st entries are $E$, with all other entries being $0$, and the differential $d_{n}$ being the identity.
\end{defn}

Let us also introduce some notation for truncation functors.

\begin{defn}
Let $\mathpzc{E}$ be an additive category which has kernels. For a complex $X_{\bullet}$ we denote by $\tau_{\ge n}X$ the complex such that $(\tau_{\ge n}X)_{m}=0$ if $m<n$, $(\tau_{\ge n}X)_{m}= X_{m}$ if $m>n$ and $(\tau_{\ge n}X)_{n}=\textrm{Ker}(d_{n})$. The differentials are the obvious ones. The construction is clearly functorial. Dually we define the truncation functor $\tau_{\le k}$. 
\end{defn}

All of the above categories are naturally enriched over $Ch(\mathpzc{Ab})$. We denote the enriched hom by $\textbf{Hom}(-,-)$. For notational clarity we recall its definition here.

\begin{defn}
Let $X_{\bullet},Y_{\bullet}\in \mathrm{Ch}(\mathpzc{E})$. We define $\textbf{Hom}(X_{\bullet},Y_{\bullet})\in Ch(\mathpzc{Ab})$ to be the complex with
$$\textbf{Hom}(X_{\bullet},Y_{\bullet})_{n}=\prod_{i\in\mathbb{Z}}\textrm{Hom}_{\mathpzc{E}}(X_{i},Y_{i+n})$$
and differential $d_{n}$ defined on $\textrm{Hom}_{\mathpzc{E}}(X_{i},Y_{i+n})$ by
$$df=d^{Y}_{i+n}\circ f-(-1)^{n}f\circ d^{X}_{i}$$
\end{defn}

Let $(\mathpzc{E},\otimes,k)$ be a monoidal additive category, i.e. $\otimes$ is an additive bifunctor. There is an induced monoidal structure on $\mathrm{Ch}_{*}(\mathpzc{E})$ for $*\in\{\ge0,\le0,+,-,b,\emptyset\}$. The unit is $S^{0}(k)$. If $X_{\bullet}$ and $Y_{\bullet}$ are chain complexes then we set

$$(X_{\bullet}\otimes Y_{\bullet})_{n}=\bigoplus_{i+j=n}X_{i}\otimes Y_{j}$$
If $i+j=n$, then we define the differential on the summand $X_{i}\otimes Y_{j}$ of $(X_{\bullet}\otimes Y_{\bullet})_{n}$ by
$$d^{X_{\bullet}\otimes Y_{\bullet}}_{n}|_{X_{i}\otimes Y_{j}}=d^{X_{\bullet}}_{i}\otimes id_{Y_{\bullet}}+(-1)^{i}id_{X_{\bullet}}\otimes d_{j}^{Y_{\bullet}}$$
If $*\in\{\ge0,\le0,+,-,b,\emptyset\}$ then $(\mathrm{Ch}_{*}(\mathpzc{E}),\otimes,S^{0}(k))$ is a monoidal additive category. 

If $(\mathpzc{E},\otimes,k,\underline{\textrm{Hom}})$ is a closed monoidal additive category then we define a functor
$$\underline{\textrm{Hom}}(-,-):\mathrm{Ch}(\mathpzc{E})^{op}\times \mathrm{Ch}(\mathpzc{E})\rightarrow \mathrm{Ch}(\mathpzc{E})$$

$$\underline{\textrm{Hom}}(X_{\bullet},Y_{\bullet})_{n}=\prod_{i\in\mathbb{Z}}\underline{\textrm{Hom}}_{\mathpzc{E}}(X_{i},Y_{i+n})$$
and differential $d_{n}$ defined on $\textrm{Hom}_{\mathpzc{E}}(X_{i},Y_{i+n})$ by

$$d=\underline{\textrm{Hom}}(d_{i}^{X_{\bullet}},id)+(-1)^{i}\underline{\textrm{Hom}}(id,d_{i+n}^{Y_{\bullet}})$$
This does define an internal hom on the monoidal category 

$$(\mathrm{Ch}(\mathpzc{E}),\otimes,S^{0}(k))$$
The internal hom on chain complexes also restricts to a bifunctor

$$\underline{\textrm{Hom}}(-,-):\mathrm{Ch}_{b}(\mathpzc{E})^{op}\times \mathrm{Ch}_{b}(\mathpzc{E})\rightarrow \mathrm{Ch}_{b}(\mathpzc{E})$$
Then

$$(\mathrm{Ch}_{b}(\mathpzc{E}),\otimes,S^{0}(k),\underline{\textrm{Hom}})$$
is a closed monoidal additive category. In fact, in both of these categories there are natural isomorphisms of chain complexes of abelian groups.
$$\textbf{Hom}(X_{\bullet},\underline{\textrm{Hom}}(Y_{\bullet},Z_{\bullet}))\cong\textbf{Hom}(X_{\bullet}\otimes Y_{\bullet},Z_{\bullet})$$

The categories $\mathrm{Ch}_{*}(\mathpzc{E})$ for $*\in\{+,-,b,\emptyset\}$ also come equipped with a shift functor. It is given on objects by $(A_{\bullet}[1])_{i}=A_{i+1}$ with differential $d_{i}^{A[1]}=-d^{A}_{i+1}$. The shift of a morphism $f^{\bullet}$ is given by $(f_{\bullet}[1])_{i}=f_{i+1}$. $[1]$ is an auto-equivalence with inverse $[-1]$. We set $[0]=\textrm{Id}$ and $[n]=[1]^{n}$ for any integer $n$.\newline
\\
 Finally, we define the mapping cone as follows.

\begin{defn}
Let $X_{\bullet}$ and $Y_{\bullet}$ be chain complexes in an additive category $\mathpzc{E}$ and $f_{\bullet}:X_{\bullet}\rightarrow Y_{\bullet}$. The \textit{mapping cone of} $f_{\bullet}$, denoted $\textrm{cone}(f_{\bullet})$ is the complex whose components are
$$\textrm{cone}(f_{\bullet})_{n}=X_{n-1}\oplus Y_{n}$$
and whose differential is
 \begin{displaymath}
d^{\textrm{cone}(f)}_{n}  = \left(
     \begin{array}{lr}
       -d_{n-1}^{X} &0\\
       -f_{n-1} & d^{Y}_{n}

   \end{array}
            \right)
\end{displaymath} 
\end{defn}
There are natural morphisms $\tau:Y_{\bullet}\rightarrow \textrm{cone}(f)$ induced by the injections $Y_{i}\rightarrow X_{i-1}\oplus Y_{i}$, and $\pi:\textrm{cone}(f)\rightarrow X_{\bullet}[-1]$ induced by the projections $X_{i-1}\oplus Y_{i}\rightarrow X_{i-1}$. The sequence
$$Y_{\bullet}\rightarrow\textrm{cone}(f)\rightarrow X_{\bullet}[-1]$$
is split exact in each degree.

\subsection{Conventions for Monoidal Model Categories}

Now we recall some of our conventions for monoidal model categories, mostly from \cite{kelly2019koszul}. The point of the weaker conditions below is that a lot of the axioms for monoidal model categories, as well as the monoid axiom, are overdetermined and in many cases too strict.

\begin{defn}\label{defn:lprop}
\begin{enumerate}
\item
Let $\mathpzc{M}$ be a model category. A map $f:X\rightarrow Y$ in $\mathpzc{M}$ is said to be \textit{left proper } if any pushout diagram
\begin{displaymath}
\xymatrix{
X\ar[d]^{f}\ar[r] & A\ar[d]\\
Y\ar[r] & P
}
\end{displaymath}
 is a homotopy pushout. One defines relative right properness dually.
\item
Let $\mathpzc{M}$ be a model category equipped with a symmetric monoidal structure. For $\mathcal{C}$ a class of objects in $\mathpzc{M}$, a map $f:X\rightarrow Y$ in $\mathpzc{M}$ is said to be $\mathcal{C}$-\textit{monoidally left proper } if $C\otimes f$ is left proper for any $C\in\mathcal{C}$. 
\end{enumerate}
\end{defn}

\begin{defn}
A weak monoidal model category $\mathpzc{M}$ is said to be an \textit{almost monoidal model category} if any pushout-product of cofibrations is left proper. 
\end{defn}

\begin{defn}
Let $\mathpzc{M}$ be a model category equipped with a monoidal structure, and let $\mathcal{C}\subset\mathpzc{M}$ An object $X$ of $\mathpzc{M}$ is said to be $K$-\textit{transverse to} $\mathcal{C}$ if for any weak equivalence $f:A\rightarrow B$ in $\mathcal{C}$, the map $X\otimes A\rightarrow X\otimes B$ is a weak equivalence. $X$ is said to be $K$-\textit{flat} if it is $K$-transverse to $\mathpzc{M}$.
\end{defn}

Clearly $K$-flatness is a model category-theoretic version of flatness in exact categories.

\begin{defn}
Let $\mathpzc{M}$ be a model category which is equipped with a symmetric monoidal structure. Let $\mathcal{S}$ be a class of maps in $\mathpzc{M}$.
\begin{enumerate}
\item
$\mathcal{S}$ is said to \textit{satisfy the pushout-product axiom} if it is closed under arbitrary pushout-products.
\item
$\mathcal{S}$ is said to \textit{satisfy the weak pushout-product axiom} if whenever $s_{1}\Box\ldots\Box s_{n}$ is an iterated pushout-product of maps in $\mathcal{S}$, and one of the $s_{i}$ is a weak equivalence, then $s_{1}\Box\ldots\Box s_{n}$ is a weak equivalence. 
\item
$\mathpzc{M}$ is said to be a \textit{weak monoidal model category} if cofibrations satisfy the weak pushout-product axiom.
\item
$\mathpzc{M}$ is said to be a \textit{weakly unital monoidal model category} if for any acyclic finration $C\rightarrow k$, with $C$ cofibrant, andy any $X$, the map $C\otimes X\rightarrow X$ is an equivalence.
\item
$\mathpzc{M}$ is said to be $C$-\textit{monoidal} if cofibrations satisfy the pushout-product axiom and the weak pushout-product axiom.
\item
$\mathpzc{M}$ is said to be a \textit{monoidal model category} if it is $C$-monoidal and weakly unital.
\item
$\mathpzc{M}$ is said to be \textit{K-monoidal} if coifbrant objects are $K$-flat.
\item
$\mathpzc{M}$ is said to be $KC$-\textit{monoidal} if it is both $K$-monoidal and $C$-monoidal.
\end{enumerate}
\end{defn}

\begin{rem}
If $\mathpzc{M}$ is a weak monoidal model category then the tensor product functor
$$\otimes:\mathpzc{M}\times\mathpzc{M}\rightarrow\mathpzc{M}$$
is left derivable in the sense of homotopical categories (see. e.g \cite{Riehl}). Thus we can make sense of $\otimes^{\mathbb{L}}$.
\end{rem}

\begin{defn}
\begin{enumerate}
\item
A class $\mathcal{S}$ of in $\mathpzc{C}$ is said to \textit{satisfy the monoid axiom} if any transfinite composition of pushouts of maps of the form $X\otimes f$ for $X\in\mathpzc{C}$ and $f\in\mathcal{S}$ is a weak equivalence.
\item
$\mathpzc{C}$ is said to satisfy the \textit{pp-monoid axiom} if the class of maps consisting of as iterated pushout-products of acyclic cofibrations satisfies the monoid axiom.
\end{enumerate}
\end{defn}

\begin{defn}[\cite{white2017model} Definition 3.4]
A monoidal model category $\mathpzc{M}$ is said to satisfy the \textit{strong commutative monoid axiom} if whwnever $f$ is a (trivialy) cofibration in $\mathpzc{M}$, then $h^{\boxtimes n}\big\slash_{\Sigma_{n}}$ is a (trivial) cofibration in $\mathpzc{M}$ for all $n>0$.
\end{defn}

Note that if $\mathcal{S}$ satisfies the monoid axiom then it must consist of weak equivalences. Moreover if $\mathpzc{C}$ satisfies the pp-monoid axiom then it is an almost monoidal model category. Finally if $\mathpzc{C}$ is a monoidal model category then the validity of the pp-monoid axiom is equivalent to the validity of the usual monoid axiom.

\subsection{Exact Category Generalities}

Next we recall some general theory of exact categories. Recall that if $\mathpzc{E}$ is an additive category then a \textit{kernel-cokernel pair} in $\mathpzc{E}$ is a sequence
\begin{displaymath}
\xymatrix{
0\ar[r] & X\ar[r]^{f} & Y\ar[r]^{g} & Z\ar[r] &0
}
\end{displaymath}
where $g$ is a cokernel of $f$ and $f$ is a kernel of $g$. A \textit{Quillen exact category} is a pair $(\mathpzc{E},\mathcal{Q})$ where $\mathpzc{E}$ is an additive category, and $\mathcal{Q}$ is a class of kernel-cokernel pairs in $\mathpzc{E}$ satisfying some axioms which make many of the constructions of homological algebra possible. For details one can consult \cite{Buehler}. In particular, there is a sensible definition of what it means for a map $f:X_{\bullet}\rightarrow Y_{\bullet}$ of complexes in $\mathrm{Ch}(\mathpzc{E})$ to be a quasi-isomorphism. The class $\mathcal{W}_{\mathcal{Q}}$ of quasi-isomorphisms satisfy the $2$-out-of-$6$ property, and therefore $(\mathrm{Ch}(\mathpzc{E}),\mathcal{W}_{\mathcal{Q}})$ is a homotopical category in the sense of \cite{Riehl} Definition 2.1.1.

\begin{notation}
Let $(\mathpzc{E},\mathcal{Q})$ be a category with a distinguished class of exact seqiences.
\begin{enumerate}
\item
If $g$ appears as a cokernel in a short exact sequence in $\mathcal{Q}$, then $g$ is said to be an \textit{admissible epimorphism}. The class of all admissible epimorphisms is denoted $\mathbf{AdEpi}$.
\item
If $f$ appears as a kernel in a short exact sequence in $\mathcal{Q}$, then $f$ is said to be an \textit{admissible monomorphism}. The class of all admissible monomorphisms is denoted $\mathbf{AdMon}$.
\item
$\mathbf{SplitMon}$ is the class of split monomorphisms, i.e. morphisms $i$ for which there existence a map $p$ such that $p\circ i$ is the identity. 
\end{enumerate}
\end{notation}

\begin{defn}[Definition 3.1, Definition 3.2 \cite{bazzoni2013one}]
A \textit{left exact category} is a pair $(\mathpzc{E},\mathcal{Q})$ where $\mathpzc{E}$ is an additive category, and $\mathcal{Q}$ a class of kernel-cokernel pairs in $\mathpzc{E}$ such that 
\begin{enumerate}
\item
The identity map $1_{0}:0\rightarrow 0$ is an admissible epimorphism.
\item
The composition of two admissible epimorphisms is an epimorphism.
\item
The pullback of an admissible epimorphism along an arbitrary morphisms exists and is an admissible epimorphism.
\end{enumerate}
A left exact category $(\mathpzc{E},\mathcal{Q})$ is said to be \textit{strongly left exact} if whenever $i:A\rightarrow B$ and $p:B\rightarrow C$ are morphisms in $\mathpzc{E}$ such that $p$ has a kernel and $p\circ i$ is an admissible epimorphism, then $p$ is an admissible epimorphism.

\textit{(Strongly) right exact categories} are defined dually. $(\mathpzc{E},\mathpzc{Q})$ is said to be an \textit{exact category} if it is both left and right exact.
\end{defn}

\begin{example}
Let $(\mathpzc{E},\mathcal{Q})$ be a Quillen exact category. Then by \cite{Buehler} Proposition 2.16 it is both a strongly right exact category and a strongly left exact category. 
\end{example}

When it is clear, when referring to an exact category we will often suppress the $\mathcal{Q}$, and just refer to an exact category $\mathpzc{E}$.

\begin{defn}
A subcategory $\mathpzc{D}\subseteq\mathpzc{E}$ of an exact category $\mathpzc{E}$ is said to be a \textit{generating subcategory of }$\mathpzc{E}$, or to \textit{generate} $\mathpzc{E}$, if for any $E\in\mathpzc{E}$ there is a $D\in\mathpzc{D}$ and an admissible epimorphism $D\rightarrow\mathpzc{E}$.
\end{defn}

\subsubsection{The Left Heart}

 In \cite{henrard2021left} Henrard, Kvamme, Van Roosmalen, and Wegner construct the \textit{left heart} $\mathrm{LH}(\mathpzc{E})$ of a left exact category $\mathpzc{E}$. This is essentially a left exact abelian envelope of $\mathpzc{E}$. First they define the so-called \textit{left} $t$-\textit{structure} on the derived category $\mathrm{D}(\mathpzc{E})$ for which the truncation functor $\tau_{\ge0}$, which sends a complex $X_{\bullet}$ to the complex
 $$\ldots\rightarrow X_{n}\rightarrow\ldots\rightarrow X_{1}\rightarrow\mathrm{Ker}(d^{X}_{0})\rightarrow0\rightarrow 0\rightarrow\ldots$$
 
 $\mathrm{LH}(\mathpzc{E})$ is then defined to be the heart of this $t$-structure. By \cite{henrard2021left} Corollary 3.10, an object of $\mathrm{LH}(\mathpzc{E})$ can be described as a complex of the form
 \begin{displaymath}
 \xymatrix{
 \ldots\ar[r] & 0\ar[r] & \mathrm{Ker}(f)\ar[r] & X\ar[r]^{f} & Y\ar[r] 0\ar[r] & \ldots
 }
 \end{displaymath}
 with $Y$ in degree $0$.  There is a natural exact functor $\phi:\mathpzc{E}\rightarrow\mathrm{LH}(\mathpzc{E})$ sending an object $E$ to the complex consisting of $E$ concentrated in degree $0$. By \cite{henrard2021left} Theorem 3.12 the induced functor
 $$\phi:\mathrm{D}(\mathpzc{E})\rightarrow\mathrm{D}(\mathrm{LH}(\mathpzc{E}))$$
 is an equivalence for $\mathpzc{E}$ strongly left exact.
 
We will denote the homology functors for the left $t$-structure by $\mathrm{LH}_{n}$.

\subsubsection{Monoidal Exact Categories and the $\otimes$-Pure Exact Structure}

\begin{defn}
A \textit{symmetric monoidal exact category} is an exact category $\mathpzc{E}$ equipped with a unital syemmtric monoidal structure $(\otimes ,k)$ such that $\otimes$ is an additive functor. A \textit{closed symmetric monoidal exact category} is an exact category $\mathpzc{E}$ equipped with a unital closed syemmtric monoidal structure $(\otimes ,k,\underline{\mathrm{Hom}})$.
\end{defn}

\begin{defn}
Recall that an object $F$ in a closed symmetric monoidal exact category $\mathpzc{E}$ is said to be \textit{flat}, if whenever 
$$0\rightarrow X\rightarrow Y\rightarrow Z\rightarrow 0$$
is an exact sequence in $\mathpzc{E}$,
$$0\rightarrow F\otimes X\rightarrow F\otimes Y\rightarrow F\otimes Z\rightarrow0$$
is an exact sequence in $\mathpzc{E}$. $F$ is said to be \textit{strongly flat} if it is flat and $F\otimes(-)$ commutes with kernels.
\end{defn}

\begin{defn}[\cite{kelly2016homotopy} Definition 2.4.75]
Let $\mathcal{S}\subset\mathpzc{E}$ be a full subcategory. A short exact sequence
$$0\rightarrow A\rightarrow B\rightarrow C\rightarrow 0$$
is said to be $\mathcal{S}$-\textit{pure} if 
$$0\rightarrow S\otimes A\rightarrow S\otimes B\rightarrow S\otimes C\rightarrow 0$$
is a short exact sequence for any $S\in\mathcal{S}$. When $\mathcal{S}=\mathpzc{E}$, we say that a $\mathcal{S}$-pure exact sequence is $\otimes$-\textit{pure}.
\end{defn}

The class of all $\mathcal{S}$-pure monomorphisms is denoted $\mathbf{PureMon}_{\mathcal{S}}$. The class of all $\otimes$-pure admissible monomorphisms is denoted $\mathbf{PureMon}_{\otimes}$.

\begin{example}[\cite{kelly2016homotopy} Lemma 2.4.76]\label{example:flatpure}
If $\mathpzc{E}$ is weakly idempotent complete and $Z$ is flat then any short exact sequence
$$0\rightarrow X\rightarrow Y\rightarrow Z\rightarrow 0$$
is $\otimes$-pure.
\end{example}

\begin{prop}[\cite{kelly2016homotopy} Proposition 2.4.77]
When $\mathpzc{E}$ is weakly idempotent complete the class of $\mathcal{S}$-pure exact sequences defines an exact structure on $\mathpzc{E}$. 
\end{prop}

We denote by $\mathpzc{E}_{\otimes}$ the category $\mathpzc{E}$ equipped with the exact structure given by $\otimes$-pure exact sequences, and call it the $\otimes$-pure exact structure.  

\begin{prop}
Let $\mathpzc{E}$ be a weakly idempotent complete symmetric monoidal exact category, and let
$$0\rightarrow X\rightarrow Y\rightarrow Z\rightarrow 0$$
be a short exact sequence with $Y$ and $Z$ flat. Then $X$ is flat.
\end{prop}

\begin{proof}
Let $0\rightarrow A\rightarrow B\rightarrow C\rightarrow 0$ be an exact sequence. Consider the diagram
\begin{displaymath}
\xymatrix{
 & 0\ar[d] & 0\ar[d] & 0\ar[d] &\\
 0\ar[r] & A\otimes X\ar[d]\ar[r] & A\otimes Y\ar[d]\ar[r] & A\otimes Z\ar[d]\ar[r] & 0\\
 0\ar[r] & B\otimes X\ar[d]\ar[r] & B\otimes Y\ar[d]\ar[r] & B\otimes Z\ar[d]\ar[r] & 0\\
 0\ar[r] & C\otimes X\ar[d]\ar[r] & C\otimes Y\ar[d]\ar[r] & C\otimes Z\ar[d]\ar[r] & 0\\
 & 0 & 0 & 0 & 
}
\end{displaymath}
Since $Y$ and $Z$ are flat the second and third columns are exact. Since $Z$ is flat the rows are short exact sequences by Example \ref{example:flatpure}. Thus by the $3\times 3$ lemma the first column is short exact. 
\end{proof}

\begin{prop}\label{prop:pushoutprodpure}
Let  $f:U\rightarrow V$ and $g:X\rightarrow Y$ be $\otimes$-pure monomorphisms with respective cokernels $C$ and $D$. Then $U\otimes Y\oplus_{U\otimes X}V\otimes X\rightarrow V\otimes Y$ is an $\otimes$-pure monomorphism with cokernel $C\otimes D$.
\end{prop}

\begin{proof}
This statement is essentially contained in the proof of \cite{vst2012exact} Theorem 8.11,
\end{proof}

 \subsubsection{Smallness Conditions and Exactness of Colimits}
 
 Here we introduce some notation and conventions for smallness of objects in categories. We also recall some terminology for smallness of objects exactness of certain colimits in exact categories from \cite{kelly2016homotopy}.

 \begin{defn}
Suppose that $\mathpzc{E}$ has $(\mathcal{I};\mathpzc{S})$-colimits. We say that $(\mathcal{I};\mathpzc{S})$-colimits are exact in $\mathpzc{E}$ if for any functor
$F\in \mathpzc{Fun}_{\mathpzc{S}}(\mathcal{I};\mathrm{Ch}(\mathpzc{E}))$
such that $F(i)$ acyclic for any object $i$ in $\mathcal{I}$, the colimit $\textrm{lim}_{\rightarrow_{\mathcal{I}}}F(i)$
is acyclic. Similarly one defines exactness of $(\mathcal{I};\mathpzc{S})^{cocont}$-colimits, $(\mathcal{I};\mathpzc{S})$-limits, and $(\mathcal{I};\mathpzc{S})^{cont}$-limits.
\end{defn}

\begin{defn}[\cite{kelly2016homotopy} Definiton 2.6.96]\label{defsmallnesscond}
Let $\mathpzc{E}$ be a category, $\mathcal{S}$ a class of morphisms in $\mathpzc{E}$, and $\kappa$ a cardinal. An object $E$ of $\mathpzc{E}$ is said to be
\begin{enumerate}
\item
$(\kappa,\mathcal{S})$-\textit{small} if the canonical map 
$$\varinjlim_{\beta\in\lambda}\textrm{Hom}(E,F_{\beta})\rightarrow\textrm{Hom}(E,\varinjlim_{\beta\in\mathcal{\lambda}}F_{\beta})$$ is an isomorphism for any cardinal $\lambda$ with $cofin(\lambda)\ge\kappa$ and any $\lambda$-indexed transfinite sequence where $F_{i}\rightarrow F_{i+1}$ is in $\mathcal{S}$.
\item
$\mathcal{S}$-\textit{small} if it is $(\kappa,\mathcal{S})$-\textit{small} for some cardinal $\kappa$
\item
$(\kappa,\mathcal{S})$-\textit{compact} if the canonical map 
$$\varinjlim_{\beta\in\lambda}\textrm{Hom}(E,F_{\beta})\rightarrow\textrm{Hom}(E,\varinjlim_{\beta\in\mathcal{\lambda}}F_{\beta})$$ 
 is an isomorphism for any regular cardinal $\lambda\ge\kappa$ and any $\lambda$-indexed transfinite sequence where $F_{i}\rightarrow F_{i+1}$ is in $\mathcal{S}$.
\item
$\mathcal{S}$-\textit{compact} if it is $(\kappa,\mathcal{S})$-\textit{compact} for some cardinal $\kappa$
\item
$(\kappa,\mathcal{S})$-\textit{presented} if the natural map 
$$\varinjlim_{i\in\mathcal{I}}\textrm{Hom}(E,F_{i})\rightarrow\textrm{Hom}(E,\varinjlim_{i\in\mathcal{I}}F_{i})$$
 is an isomorphism for any $\lambda$-filtered inductive system $F:\mathcal{I}\rightarrow\mathpzc{E}$ whose colimit exists where $\lambda\ge\kappa$ is regular, and such that $F(\alpha)\in\mathcal{S}$ for any morphism $\alpha$ in $\mathcal{I}$.
 \item
 $\mathcal{S}$-\textit{presented} if it is $(\kappa,\mathcal{S})$-presented for some cardinal $\kappa$.
 \item
 $\mathcal{S}$-\textit{tiny} if it is $(0,\mathcal{S})$-presented, where $0$ is the first ordinal.
 \item
 \textbf{tiny} if it is $\mathcal{S}$-\textit{tiny} for $\mathcal{S}=\mathpzc{Mor}(\mathpzc{E})$.
\end{enumerate}
\end{defn}

\begin{defn}
Given a set of maps $I$ in a category $\mathpzc{E}$, we denoted by $\mathrm{cell}(I)$ the class of morphisms obtained as transfinite compositions of pushouts of elements of $I$.
\end{defn}

\begin{defn}[\cite{kelly2016homotopy} Definition 2.6.96]\label{weaklyelementary}
Let $\mathpzc{E}$ be an exact category and $\mathpzc{S}$ a collection of morphisms in $\mathpzc{E}$. $\mathpzc{E}$ is said to be
\begin{enumerate}
\item
\textit{weakly} $(\lambda;\mathcal{S})$-\textit{elementary} for an ordinal $\lambda$ if $\mathpzc{E}$ has $(\lambda;\mathpzc{S})^{cocont}$-colimits and $(\lambda;\mathpzc{S})^{cocont}$-colimits are exact. 
\item
\textit{weakly }$\mathpzc{S}$-\textit{elementary} if for any ordinal $\lambda$ $\mathpzc{E}$ is weakly $(\lambda;\mathcal{S})$-elementary. 
\item
\textit{weakly }$\mathbf{AdMon}$-\textit{elementary} if it is weakly $\mathpzc{S}$-elementary for the class $\mathpzc{S}=\mathbf{AdMon}$ of admissible monomorphisms.
\item
\textit{weakly elementary} if it is weakly $\mathpzc{S}$-elementary for $\mathpzc{S}=\mathpzc{Mor}(\mathpzc{E})$.
\end{enumerate}
\end{defn}

We introduce one more definition. For a class of objects $\mathpzc{A}$, let $\mathbf{AdMon}_{\mathpzc{A}}$ denote the class of admissible monomorphisms with cokernel in $\mathpzc{A}$. In particular, the condition of being weakly $\mathbf{AdMon}_{\mathpzc{A}}$-elementary will appear frequently. The same condition appears in forthcoming work of Gillespie \cite{Gillespiebook} wherein a class $\mathpzc{A}$ such that $\mathpzc{E}$ is weakly $\mathbf{AdMon}_{\mathpzc{A}}$-elementary in the terminology of the present paper, is called an \textit{efficient} class in loc. cit. This in turn is a riff on \cite{saorin2011exact}, where an exact category $\mathpzc{E}$ is called \textit{efficient} if it is $\mathbf{AdMon}$-elementary in the terminology of the presnet paper.

%

The important relationship between smallness of objects and exactness of colimits occurs when there are enough projectives of a certain size.

\begin{prop}[\cite{kelly2016homotopy} Proposition 2.6.101]\label{elementinduct} 
Let $\mathpzc{E}$ be a complete and cocomplete exact category, $\mathcal{I}$ a  filtered category, and 
$$0\rightarrow F\rightarrow G\rightarrow H\rightarrow 0$$
a null sequence of functors $\mathcal{I}\rightarrow\mathpzc{E}$ such that for each $i\in\mathcal{I}$,
$$0\rightarrow F(i)\rightarrow G(i)\rightarrow H(i)\rightarrow 0$$
is exact. Suppose there is a class $\mathcal{P}$ of projective generators of $\mathpzc{E}$ such that the maps
$$\varinjlim_{i\in\mathcal{I}}\textrm{Hom}(P,F(i))\rightarrow\textrm{Hom}(P,\varinjlim_{i\in\mathcal{I}}F(i))$$
$$\varinjlim_{i\in\mathcal{I}}\textrm{Hom}(P,G(i))\rightarrow\textrm{Hom}(P,\varinjlim_{i\in\mathcal{I}}G(i))$$
$$\varinjlim_{i\in\mathcal{I}}\textrm{Hom}(P,H(i))\rightarrow\textrm{Hom}(P,\varinjlim_{i\in\mathcal{I}}H(i))$$
are isomorphisms for any $P\in\mathcal{P}$. Then the sequence 
$$0\rightarrow \varinjlim_{i\in\mathcal{I}}F(i)\rightarrow \varinjlim_{i\in\mathcal{I}}G(i)\rightarrow \varinjlim_{i\in\mathcal{I}}H(i)\rightarrow 0$$
is exact.
\end{prop}
 
 \section{Accessible Exact Categories}\label{sec:accexact}
 
In this chapter we consider exact categories whose underlying category is accessible. Recall that for a cardinal $\lambda$, a category $\mathpzc{C}$ is $\lambda$-\textit{accessible} if there is a set $\mathcal{G}$ of $\lambda$-presented objects such that every object of $\mathpzc{C}$ can be written as a $\lambda$-filtered colimit of objects of $\mathcal{G}$. If in addition $\mathpzc{C}$ is cocomplete then it is said to be \textit{locally} $\lambda$-\textit{presentable}. Importantly, we will use accessibility to prove that certain classes of objects deconstructible in themselves in the sense of \cite{vst2012exact} Definition 3.9. This will be crucial for constructing cotorsion pairs. Gillespie also studies accessible exact categories in the forthcoming work \cite{Gillespiebook}, in which it is shown that exact model structures on such categories have well-generated homotopy categories.
 
 \subsection{The General Theory}
 
  We begin by recalling some basic facts about locally presentable categories from \cite{adamek1994locally}. Let $\lambda$ be a cardinal.

 \subsubsection{$\lambda$-Pure Morphisms}

 \begin{defn}[\cite{adamek1994locally} Definiton 2.27 (for Part (1))]
 Let $\mathpzc{E}$ be any category
 \begin{enumerate}
 \item
 A morphism $f:A\rightarrow B$ is said to be $\lambda$-\textit{pure} if each commutative square
 \begin{displaymath}
 \xymatrix{
 A'\ar[d]^{f'}\ar[r]^{U} & A\ar[d]^{f}\\
 B'\ar[r]^{v} & B
 }
 \end{displaymath}
 with $A'$ and $B'$ $\lambda$-presented, there is a morphism $\overline{u}:B'\rightarrow A$ such that $u=\overline{u}\circ f'$. 
 \item
  A morphism $f:X\rightarrow Y$ in $\mathpzc{C}$ is said to be a $\lambda$-\textit{pure epimorphism} if for all $\lambda$-presented objects $E$, $\mathrm{Hom}(E,f):\mathrm{Hom}(E,X)\rightarrow\mathrm{Hom}(E,Y)$ is an epimorphism of sets. 
 \end{enumerate}
 \end{defn}
 
   \begin{prop}[\cite{adamek1994locally} Proposition 2.30 (2)]\label{prop:lambdadirected}
Let $\mathpzc{C}$ be a locally $\lambda$-presentable category. A morphism $f:X\rightarrow Y$ in $\mathpzc{C}$ is a $\lambda$-pure monomorphism if it is a $\lambda$-directed colimit in $\mathrm{Mor}(\mathpzc{C})$ of split monomorphisms. 
\end{prop}

\begin{thm}[\cite{adamek1994locally} Theorem 2.33]\label{thm:deconstructpres}
Let $\mathcal{K}$ be a $\lambda$-accessible category. There exist arbitrary large regular cardinals $\gamma\ge\lambda$ such that every map $A\rightarrow B$ in $\mathcal{K}$ with $A$ $\gamma$-presentable factors through a $\lambda$-pure monomorphism $\overline{f}:\overline{A}\rightarrow B$ with $\overline{A}$ $\gamma$-presentable.
\end{thm}

\begin{cor}
Let $\mathpzc{E}$ be a locally $\lambda$-presentable category. There is a cardinal $\gamma\ge\lambda$ such that any object $E$ of $\mathpzc{E}$ can be written as a $\gamma$-filtered colimit $E\cong\colim_{\mathcal{I}}\overline{E}_{i}$ where 
\begin{enumerate}
\item
each $\overline{E}_{i}$ is $\gamma$-presentable,
\item
each map $\overline{E}_{i}\rightarrow E$ is a $\lambda$-pure monomorphism. 
\end{enumerate}
\end{cor}

\begin{proof}
$\mathpzc{E}$ is also $\kappa$-presentable folr any $\kappa\ge\lambda$. By Theorem \ref{thm:deconstructpres} there are arbitrarlily large regular cardinals $\gamma$ such that whenever $D\rightarrow E$ is a map with $D$ $\gamma$-presentable, then $D\rightarrow E$ factors as $D\rightarrow\overline{D}\rightarrow E$ where $\overline{D}$ is $\gamma$-presentable and $\overline{D}\rightarrow E $ is pure. Now $\mathpzc{E}$ is also $\gamma$-presentable. Thus we may write $E\cong\colim_{\mathcal{I}}E_{i}$ as a $\gamma$-filtered colimit with each $E_{i}$ $\gamma$-presentable. Each $E_{i}\rightarrow E$ factors through a $\lambda$-pure monomorphism $\overline{E}_{i}\rightarrow E$ with $\overline{E}_{i}$ $\gamma$-presentable.
\end{proof}

\subsection{The $\lambda$-Pure Exact Structure}

Let $\mathpzc{E}$ be a locally $\lambda$-accessible additive category for some regular cardinal $\lambda$. Following \cite{krause2012approximations}, say that a null sequence
 $$0\rightarrow X\rightarrow Y\rightarrow Z\rightarrow 0$$
 in $\mathpzc{E}$ is $\lambda$-\textit{pure exact} if for any $\lambda$-presentable object $H$, the sequence of abelian groups
  $$0\rightarrow Hom(H,X)\rightarrow Hom(H,Y)\rightarrow Hom(H,Z)\rightarrow 0$$
  is exact. 

In \cite{positselski2023locally} Proposition 4.4 the following is shown, which improves upon Proposition \ref{prop:lambdadirected} in the case of additive categories.

\begin{prop}[\cite{positselski2023locally} Proposition 4.4]
Let $\mathpzc{E}$ be a $\lambda$-accessible additive category. A morphism $f:X\rightarrow Y$ in $\mathpzc{E}$ is  a $\lambda$-pure monomorphism if and only if it is a $\lambda$-directed colimit in $\mathrm{Mor}(\mathpzc{E})$ of split monomorphisms. 
\end{prop}

As a consequence it is explained in \cite{positselski2023locally} Section 4 that, in much the same way as for locally presentable exact categories (\cite{krause2012approximations}, \cite{gillespie2016derived}, \cite{kelly2016homotopy} Proposition 3.3.53), the $\lambda$-pure exact structure does define an exact structure on $\mathpzc{E}$. We also have the following.
%
%

\begin{prop}[\cite{estrada2017pure}, Proposition 2.9]\label{prop:colimitslambdapure}
 Let $P:\mathcal{I}\rightarrow\mathpzc{E}$ be a $\lambda$-directed system in a $\lambda$-accessible additive category $\mathpzc{E}$. Then the map $\bigoplus_{i\in\mathcal{I}}P_{i}\rightarrow\textrm{colim}_{\mathcal{I}}P_{i}$ is a $\lambda$-pure epimorphism.
  \end{prop}

In particular we have the following.

  \begin{cor}
  Let $\mathpzc{E}$ be a $\lambda$-accessible additive category. The collection of all $\lambda$-pure exact sequences defines an exact structure on $\mathpzc{E}$. If $\mathpzc{E}$ has aribtrary direct sums then the $\lambda$-pure exact structure has enough projectives.
  \end{cor}
  
  The class of $\lambda$-pure monomorphisms is denoted $\mathbf{PureMon}_{\lambda}$.
     
  \subsection{Deconstructibility and Structure Theorems} 
 
 $\lambda$-accessibility will allow us to prove deconstructiblity results for objects in exact categories, in the sense of \cite{vst2012exact}, thus allowing us to apply the results therein to construct cotorsion pairs. We begin by giving a slightly generalised definition of deconstructibility. 
  
\begin{defn}
Let $\mathpzc{C}$ be a category and $\mathcal{M}$ a set of morphisms in $\mathpzc{C}$. A class of objects $\mathpzc{A}$ in $\mathpzc{C}$ is said to be $\mathcal{M}$-\textit{pre-deconstructible} if 
\begin{enumerate}
\item
it can be written as a transfinite composition
$$A\cong\colim_{\alpha<\gamma}A_{\alpha}$$
where for each $\alpha<\gamma$ the map $A_{\alpha}\rightarrow A_{\alpha+1}$ is in $\mathcal{M}$ and each $A_{\alpha}\in\mathpzc{A}$. 
\item
every object in the domain of a map in $\mathcal{M}$ is small relative to $\mathcal{M}$.
\end{enumerate}
\end{defn}

For a left exact category $\mathpzc{E}$, a class of objects $\mathpzc{A}\subset\mathpzc{E}$, and a class of admisstible monomorphisms $\mathbf{I}$, we denote by $\mathbf{I}_{\mathpzc{A}}$ the class of maps in $\mathbf{I}$ with cokernel in $\mathpzc{A}$.
  
 \begin{defn}[c.f. \cite{vst2012exact} Definition 3.9]
Let $\mathbf{I}$ be a class of admissible monomorphisms. 
\begin{enumerate}
\item
A class of objets $\mathpzc{A}$ in $\mathpzc{E}$ is said to be \textit{pre-deconstructible in itself relative to }$\mathbf{I}$ if it is $\mathbf{J}_{\mathpzc{A}}$-pre-deconstructible for some set of admissible monomorphisms $\mathbf{J}\subset\mathbf{I}$ whose domains and codomains are in $\mathpzc{A}$. 
\item
A class of objets $\mathpzc{A}$ in $\mathpzc{E}$ is said to be \textit{deconstructible in itself relative to }$\mathbf{I}$. if it is $\mathbf{J}_{\mathpzc{A}}$-pre-deconstructible for some set of admissible monomorphisms $\mathbf{J}\subset\mathbf{I}$ whose domains and codomains are in $\mathpzc{A}$, and transfinite compositions of pushouts of maps in $\mathbf{J}_{\mathpzc{A}}$ are in $\mathbf{I}_{\mathpzc{A}}$. 
\end{enumerate}
\end{defn}

In applications we will typically require the following.

\begin{defn}\label{defn:deconstrexact}
Let $\mathpzc{E}$ be an exact category and $\mathpzc{A}$ a class of objects in $\mathpzc{E}$. $\mathpzc{A}$ is said to be \textit{presentably (pre-)deconstructible in itself} relative to a class of monomorphisms $\mathbf{I}$, if it is $\mathbf{J}_{\mathpzc{A}}$-(pre-)deconstructible where $\mathbf{J}\subset\mathbf{I}$, and the domain and codomian of objects in $\mathbf{J}$ are small relative to transfinite compositions of pushouts of maps in $\mathbf{J}$.
\end{defn}

Note that if $\mathpzc{E}$ is accessible then all classes which are (pre-)deconstructible in themselves relative to some class $\mathbf{I}$ are in fact presentably (pre-)deconstructible in themselves relative to $\mathbf{I}$.

Many examples of deconstructible classes of objects which are of interest to us consist of strongly $\lambda$-pure subobject stable classes in accessible categories.

\begin{defn}
\begin{enumerate}
\item
A \textit{purely} $\lambda$-\textit{accessible exact category} is a complete exact category $(\mathpzc{E},\mathpzc{Q})$ such that $\mathpzc{E}$ is $\lambda$-accessible, and such that $\lambda$-pure monomorphisms are admissible.
\item
A \textit{purely locally} $\lambda$-\textit{presentable exact category}  is an exact category $(\mathpzc{E},\mathpzc{Q})$ such that $\mathpzc{E}$ is $\lambda$-presentable, and such that $\lambda$-pure monomorphisms are admissible. 
\end{enumerate}
\end{defn}

\begin{rem}
Let $(\mathpzc{E},\mathpzc{Q})$ be an exact category whose underlying category is $\lambda$-accessible/ locally $\lambda$-presentable. Suppose further that it has a generating set consisting of $\lambda$-presentable projectives. Then $(\mathpzc{E},\mathpzc{Q})$ is purely $\lambda$-accessible/ lcoally purely $\lambda$-presentable.
\end{rem}

\begin{rem}
As pointed out in \cite{positselski2023locally}, \cite{adamek1994locally} Observation 2.4 and Proposition 1.16 imply that $\lambda$-accessible additive categories are weakly idempotent complete.
\end{rem}

\begin{defn}
Let $\mathpzc{E}$ be a purely $\lambda$-accessible exact category  and $\mathpzc{A}$ a class of objects in $\mathpzc{E}$. 
\begin{enumerate}
\item
A map $f:N\rightarrow M$ is said to be an \textit{almost-}$(\mathpzc{A},\lambda)$-\textit{pure monomorphism} if there is a fibre product diagram
\begin{displaymath}
\xymatrix{
N\ar[d]^{f}\ar[r]^{g} & A\ar[d]^{i}\\
M\ar[r]^{h} & B
}
\end{displaymath}
where $A,B\in\mathpzc{A}$, $i$ is a $\lambda$-pure monomorphism with cokernel in $\mathpzc{A}$, and $h$ is an admissible epimorphism.
\item
A map $f:N\rightarrow M$ is said to be a \textit{pseudo-}$(\mathpzc{A},\lambda)$-\textit{pure monomorphism} if it is an admissible monomorphism which can be written as follows. There is a $\Gamma$-indexed transfinite sequence
\begin{displaymath}
\xymatrix{
M_{0}\ar[r]\ar[d] & M_{\alpha}\ar[d]\ar[r] & M_{\alpha'}\ar[r]\ar[d] & \ldots\\
M\ar[d]\ar[r] & M\ar[r]\ar[d] & M\ar[r]\ar[d] & \ldots\\
M\big\slash M_{0}\ar[r] &M\big\slash M_{\alpha}\ar[r] & M\big\slash M_{\alpha'}\ar[r] &\ldots
}
\end{displaymath}
where for each successor ordinal $\alpha+1\in \Gamma$, $M_{\alpha+1}\rightarrow M$ is an almost $(\mathpzc{A},\lambda)$-pure monomorphism, and $f$ is the colimit of the maps $M_{\alpha}\rightarrow M$ in $\mathrm{Mor}(\mathpzc{E})$. The class of all pseudo-$(\mathpzc{A},\lambda)$-pure monomorphisms will be denoted $\mathbf{pureMon}_{\mathpzc{A},\lambda}$. 
\end{enumerate}
\end{defn}

\begin{defn}
\begin{enumerate}
\item
Let $\mathcal{S}$ be a class of admissible monomorphisms in an exact category $\mathpzc{E}$. We say that a class $\mathpzc{A}$ in $\mathpzc{E}$ is $\mathcal{S}$-\textit{subobject stable} if whenever $M\rightarrow N$ is a map in $\mathcal{S}$ with $N\in\mathpzc{A}$, then $M$ and $N\big\slash M$ are in $\mathpzc{A}$.
\item
Let $\mathpzc{E}$ be a purely $\lambda$-accessible exact category. A class of objects $\mathpzc{A}$ is said to be $\lambda$-\textit{pure subobject stable} if it is $\mathbf{PureMon}_{\lambda}$-subobject stable.
\item
Let $\mathpzc{E}$ be a purely $\lambda$-accessible exact category. A class of objects $\mathpzc{A}$ is said to be \textit{strongly} $\lambda$-\textit{pure subobject stable} if
\begin{enumerate}
\item
 it is $\mathbf{pureMon}_{\mathpzc{A},\lambda}$-subobject stable.
 \item
for any $\Gamma$-indexed transfinite sequence
\begin{displaymath}
\xymatrix{
M_{0}\ar[r]\ar[d] & M_{\alpha}\ar[d]\ar[r] & M_{\alpha'}\ar[r]\ar[d] & \ldots\\
M\ar[d]\ar[r] & M\ar[r]\ar[d] & M\ar[r]\ar[d] & \ldots\\
M\big\slash M_{0}\ar[r] &M\big\slash M_{\alpha}\ar[r] & M\big\slash M_{\alpha'}\ar[r] &\ldots
}
\end{displaymath}
where for each successor ordinal $\alpha+1\in \Gamma$, $M_{\alpha+1}\rightarrow M$ is an almost $(\mathpzc{A},\lambda)$-pure monomorphism, the colimit sequence
$$0\rightarrow\colim M_{\alpha}\rightarrow M\rightarrow\colim M\big\slash M_{\alpha}\rightarrow0$$
is exact and $M\big\slash M_{\alpha}$ is in $\mathpzc{A}$. In particular $\colim M_{\alpha}\rightarrow M$ is in $\mathbf{pureMon}_{\mathpzc{A},\lambda}$.
 \end{enumerate}
\end{enumerate}
\end{defn}

\begin{rem}\label{rem:pushpull}
\begin{enumerate}
\item
Let 
\begin{displaymath}
\xymatrix{
M\ar[d]^{f}\ar[r]^{g} & A\ar[d]^{i}\\
N\ar[r]^{h} & B
}
\end{displaymath}
be a fibre-product diagram with $i$ an admissilble monomorphism and $h$ an admissible epimorphism. Then $f$ is an admissible monomorphism and $g$ an admissible epimorphism by \cite{Buehler} Proposition 2.15. Moreover by \cite{Buehler} Proposition 2.12 the right-hand square is also a pushout, so $N\big\slash M\cong B\big\slash A$. Thus almost $(\mathpzc{A},\lambda)$-pure monomorphisms are automatically admissible monomorphisms with cokernel in $\mathpzc{A}$.
\item
Suppose that $\mathpzc{E}$ is weakly elementary. Then pseudo $(\mathpzc{A},\lambda)$-pure monomorphisms are admissible monomorphisms. Thus in this instance, if $\mathpzc{A}$ is $\lambda$-pure subobject stable and is closed under colimits of transfinite sequences of objects in $\mathpzc{A}$, then $\mathpzc{A}$ is strongly $\lambda$-pure subobject stable 
\end{enumerate}
 \end{rem}

The following result can be seen as a generalisation of \cite{Gillespie2} Proposition 4.9.

\begin{lem}\label{lem:accdec}
Let $\mathpzc{E}$ be a purely $\lambda$-accessible exact category with a generator $G$. Let $\mathpzc{A}$ be a class of objects in $\mathpzc{E}$ which is strongly $\lambda$-pure subobject stable and closed under transfinite extensions. Suppose further that  $\mathpzc{E}$ is weakly $\mathbf{AdMon}_{\mathpzc{A}}$-elementary. Then there exist arbitrarily large regular cardinals $\gamma$ such that for map $N\rightarrow M\in\mathbf{AdMon}_{\mathpzc{A}}$, there is a transfinite sequence
$$M_{\alpha\in A}\rightarrow M$$ 
such that
\begin{enumerate}
\item
$M_{0}=N$
\item
for each successor ordinal $\alpha+1$, $M_{\alpha+1}\rightarrow M$ is an almost $(\mathpzc{A},\lambda)$-pure monomorphism.
\item
for any $\alpha\in A$ $M_{\alpha+1}\big\slash M_{\alpha}$ is $\gamma$-presentable and is in $\mathpzc{A}$.
\item
for any limit ordinal $\alpha$, $M_{\alpha}\rightarrow M$ is in $\mathbf{pureMon}_{\mathpzc{A},\lambda}$.
\item
for any $\alpha<\beta$, $M_{\alpha}\rightarrow M_{\beta}\in\mathbf{AdMon}_{\mathpzc{A}}$.
\item
$\colim_{\alpha}M_{\alpha}\rightarrow M$
is an isomorphism.
\end{enumerate}
In particular taking $N=0$, $\mathpzc{A}$  is presentably deconstructible in itself relative to $\mathbf{AdMon}$.
\end{lem}

\begin{proof}
To begin, let $G$ be a generator for $\mathpzc{E}$, and consider the set $\mathrm{Hom}(G,M)$. Fix a well-order $ A$ on $\mathrm{Hom}(G,M)$. Then we may write $\mathrm{Hom}(G,M)=\bigcup_{ \alpha\in A}X_{\alpha}$ as an increasing union of subsets of $\mathrm{Hom}(G,M)$, such that $X_{\alpha+1}\setminus X_{\alpha}$ is a singleton. Now $G$ is $\overline{\gamma}$-compact for some $\overline{\gamma}$. Let $\gamma\ge\overline{\gamma}$ be such that every map $A\rightarrow B$ in $\mathpzc{E}$ with $A$ $\gamma$-presentable factors through a $\lambda$-pure monomorphism $\overline{A}\rightarrow B$ with $\overline{A}$ $\gamma$-presentable. 

We construct each $M_{\alpha}$ by transfinite induction. Define $M_{0}=N$.  Let $\alpha\in A$ and suppose that for all $\alpha'<\alpha$, $M_{\alpha'}\rightarrow M$ has been constructed.

Suppose $\alpha+1$ is a successor ordinal. Consider $M\big\slash M_{\alpha}$ and write $\{f_{\alpha}\}=X_{\alpha+1}\setminus X_{ \alpha}$. The map $f_{\alpha}:G\rightarrow M\rightarrow M\big\slash M_{\alpha}$ factors through some $\overline{M_{\alpha+1}}\rightarrow M\big\slash M_{\alpha}$ with $\overline{M_{\alpha+1}}\rightarrow M$ being a $\lambda$-pure monomorphism and $\overline{M_{\alpha+1}}$ being $\gamma$-presentable. Now $M\big\slash M_{\alpha}$ is in $\mathpzc{A}$, and since $\overline{M}_{\alpha+1}\rightarrow M\big\slash M_{\alpha}$ is $\lambda$-pure, $\overline{M}_{\alpha+1}$ and $(M\big\slash M_{\alpha})\big\slash\overline{M}_{\alpha+1}$ are in $\mathpzc{A}$.

We have a diagram
\begin{displaymath}
\xymatrix{
M_{\alpha}\ar[d]\ar[r] & M_{\alpha+1}\ar[d]\ar[r] & \overline{M_{\alpha+1}}\ar[d]\\
M_{\alpha}\ar[r] & M\ar[r] & M\big\slash M_{\alpha}
}
\end{displaymath}
where the right-hand square is a pullback, and both rows are exact. Then $M_{\alpha+1}\rightarrow M$ is an almost $(\mathpzc{A},\lambda)$-pure monomorphism, and $M\big\slash M_{\alpha}\cong(M\big\slash M_{\alpha})\big\slash\overline{M}_{\alpha+1}\in\mathpzc{A}$ by Remark \ref{rem:pushpull}.

If $\alpha$ is a limit ordinal, define $M_{\alpha}\defeq\colim_{\alpha'<\alpha}M_{\alpha'}$. Since $M_{\alpha'}\rightarrow M_{\alpha'+1}$ is almost $(\mathpzc{A},\lambda)$-pure for any $\alpha'<\alpha$, and $\mathpzc{A}$ is strongly $\lambda$-pure subobject stable by assmption, the map $M_{\alpha}\rightarrow M$ is in $\mathbf{AdMon}_{\mathpzc{A}}$, and in fact is in $\mathbf{pureMon}_{\mathpzc{A},\lambda}$.


Now by construction
$$\mathrm{Hom}(G,\colim M_{\alpha})\rightarrow \mathrm{Hom}(G,M)$$
is an epimorphism. Therefore $\colim M_{\alpha}\rightarrow M$ is an admissible epimorphism. Moreover the map $\colim M_{\alpha}\rightarrow M$ is an admissible monomorphism, so it must be an isomorphism. 

Note that each $M_{\alpha}\rightarrow M_{\beta}$ is a transfinite composition of almost maps in $\mathbf{AdMon}_{\mathpzc{A}}$, and thus is in $\mathbf{AdMon}_{\mathpzc{A}}$. In particular if $N=0$ then each $M_{\alpha}$ is in $\mathpzc{A}$. 

%
%
\end{proof}


\begin{cor}
Let $\mathpzc{E}$ be a purely $\lambda$-accessible exact category with a generator for the $\lambda$-pure exact structure which is weakly $\mathbf{AdMon}$-elementary. Then $\mathpzc{E}$ is of Grothendieck type. In particular it has enough injectives.
\end{cor}

\begin{rem}
Let $\mathpzc{E}$ be finitely accessible and additive. With the $\aleph_{0}$-pure exact structure $\mathpzc{E}$ is in fact an elementary exact category. In particular this exact structure has a generator given as the sum of compact objects.  Suppose $(\mathpzc{E},\mathpzc{Q})$ is a purely $\aleph_{0}$-exact structure which is weakly elementary (for example the $\aleph_{0}$-pure exact structure itself). Again the sum of compact objects is a generator. Any $\aleph_{0}$-pure stable class of objects which is closed under transfinite colimits is presentably deconstructible in itself relative to $\mathbf{AdMon}$ by Remark \ref{rem:pushpull}. In particular, in this case $(\mathpzc{E},\mathpzc{Q})$ is of Grothendieck type. This recovers \cite{positselski2023locally} Theorem 5.3.
\end{rem}

\begin{lem}[c.f. \cite{gillespie2006flat} Lemma 4.9.]\label{lem:subepi}
Let $\mathpzc{E}$ be a purely locally $\lambda$-presentable exact category. Let $\gamma>\lambda$ be such that any map $A\rightarrow B$ with $A$ $\gamma$-presentable factors through a pure monomorphism $\overline{A}\rightarrow B$ with $\overline{A}$ $\gamma$-presentable. Let $f:X\rightarrow Y$ be an admissible epimorphism, and $j:T\rightarrow Y$ an admissible monomorphism with $T$ $\gamma$-presentable. Then there is an admissible monomorphism $i:S\rightarrow X$ with $S$ $\gamma$-presentable, and a commutative diagram
\begin{displaymath}
\xymatrix{
S\ar[d]\ar[r]^{i} & X\ar[d]^{f}\\
T\ar[r]^{j} & Y
}
\end{displaymath}
with $S\rightarrow T$ an admissible epimorphism.
\end{lem}

\begin{proof}
By pulling back to $T$ we may assume that $Y=T$ and that $Y$ is $\gamma$-presentable. Let $K=\mathrm{Ker}(f)$. Write $K\cong\colim_{i\in\mathcal{I}}K_{i}$ as a $\gamma$-filtered colimit with each $K_{i}$ being $\gamma$-presented and $K_{i}\rightarrow K$ being a $\lambda$-pure monomorphism, and $X\cong\colim_{j\in\mathcal{J}}X_{j}$ with each $X_{j}$ being $\gamma$-presented and $X_{j}\rightarrow X$ being a $\lambda$-pure monomorphism. By passing to a cofinal diagram we may assume that $\mathcal{I}=\mathcal{J}$, and that $K\rightarrow X$ is a colimit of maps $K_{i}\rightarrow X_{i}$ which by the obscure axiom are admisisble monomorphisms. Let $C_{i}= X_{i}\big\slash K_{i}$. We have $C\defeq\colim_{i\in\mathcal{I}}C_{i}\cong Y$ with each $C_{i}$ being $\gamma$-presentable. Thus $C$ is $\gamma$-presentable, and the map $\mathrm{Id}:C\rightarrow C$ factors through some $C_{i}$. In particular $C_{i}\rightarrow C$ is an admissible epimorphism. Thus $X_{i}\rightarrow C_{i}\rightarrow C\rightarrow Y$ is an admissible epimorphism.
\end{proof}

\begin{lem}[c.f. \cite{MR2139915} Theorem 4.1]
Let $\mathpzc{E}$ be a purely locally $\lambda$-presentable exact category, and let $\mathpzc{A}$ be a strongly $\lambda$-pure subobject stable class closed under transfinite extensions such that $\mathpzc{E}$ is weakly $\mathbf{AdMon}_{\mathpzc{A}}$-elementary. Let $\kappa\ge\lambda$ be such that any morphism $X\rightarrow Y$ with $X$ $\kappa$-presentable factors through a $\lambda$-pure monomorphism $\overline{X}\rightarrow Y$ with $\overline{X}$ $\kappa$-presentable. Let $\gamma\ge\kappa$ be a regular cardinal such that $\gamma$-presentable objects are closed under finite limits. Denote by $\mathbf{I}^{\gamma}_{\mathpzc{A}}$ the class of admissible monomorphisms $A\rightarrow B$ with cokernel in $\mathpzc{A}$, such that $B$ is $\gamma$-presented. Then
 $\mathbf{AdMon}_{\mathpzc{A}}\subseteq\mathbf{I}_{\mathpzc{A}}^{\gamma}$-cof.
\end{lem}

\begin{proof}
Since $\mathpzc{E}$ is weakly $\mathbf{AdMon}_{\mathpzc{A}}$-elementary and $\mathpzc{A}$ is closed under transfinite extensions, we clearly have $\mathbf{I}_{\mathpzc{A}}^{\gamma}$, and hence $\mathbf{I}_{\mathpzc{A}}^{\gamma}$-cell, is contained in $\mathbf{AdMon}_{\mathpzc{A}}$. 

Let $A\rightarrow B$ be an admissible monomorphism with $B\big\slash A$ in $\mathpzc{A}$, and $|B\big\slash A|\le\kappa$. By Lemma \ref{lem:subepi} there is an admissible $B'\rightarrow B$ with $|B'|\le\kappa$ and such that $B'\rightarrow B\rightarrow B\big\slash A$ is an admissible epimorphism. Let $A'=\mathrm{Ker}(B'\rightarrow B\big\slash A)$. Then $|A'|\le\gamma$, so that $A'\rightarrow B'$ is in $\mathbf{I}^{\gamma}_{\mathpzc{A}}$. Let
\begin{displaymath}
\xymatrix{
A\ar[d]\ar[r] & M\ar[d]^{p}\\
B\ar[r] & N
}
\end{displaymath}
be a commutative diagram with $p\in\mathbf{I}^{\gamma}_{\mathpzc{A}}$-inj. Since $A'\rightarrow B'\in\mathbf{I}^{\gamma}_{\mathpzc{A}}$ we find a lift in the diagram

\begin{displaymath}
\xymatrix{
A'\ar[d]\ar[r] & M\ar[d]^{p}\\
B'\ar[ur]\ar[r] & N
}
\end{displaymath}
Now by \cite{Buehler} Proposition 2.12 there is an exact sequence
$$0\rightarrow A'\rightarrow B'\oplus A\rightarrow B\rightarrow 0$$
This induces a lift in the first diagram.

 Write $B\cong\colim_{\alpha<\Gamma}B_{\alpha}$ with $B_{0}=A$, where $B_{\alpha}\rightarrow B$ is an admissiblbe monomorphism, and each $B_{\alpha+1}\big\slash B_{\alpha}$ is $\kappa$-presentable and in $\mathpzc{A}$. Thus $A\rightarrow B$ is a transfinite extension of maps in $\mathbf{I}_{\mathpzc{A}}^{\gamma}$-cof, and hence is in  $\mathbf{I}_{\mathpzc{A}}^{\gamma}$-cof.
 
\end{proof}
%

\subsubsection{Deconstructiblity in Chain Complexes}

Let $\mathpzc{E}$ be a purely $\lambda$-accessible exact category in which transfinite extensions of $\lambda$-pure monomorphisms are admissible monomorphisms. Then  when equipped with the degreewise exact structure, $\mathrm{Ch}(\mathpzc{E})$ is also a purely $\lambda$-accessible exact category in which transfinite extensions of $\lambda$-pure monomorphisms are admissible monomorphisms. A particularly useful observation here is that acyclic complexes are strongly $\lambda$-pure subobject stable. Let $\mathfrak{W}$ denote the class of acyclic sequences in $\mathrm{Ch}(\mathpzc{E})$.

\begin{lem}\label{lem:Znsubexact}
Let $\mathpzc{E}$ be a weakly idempotent complete exact category, and let
$$0\rightarrow X\rightarrow Y\rightarrow Z\rightarrow 0$$
be a degreewise exact sequence in $\mathrm{Ch}(\mathpzc{E})$ such that for each $n\in\mathbb{Z}$, 
$$0\rightarrow Z_{n}X\rightarrow Z_{n}Y\rightarrow Z_{n}Z\rightarrow 0$$
is an exact sequence. If $Y$ is acyclic then both $X$ and $Z$ are acyclic.
\end{lem}

\begin{proof}
Pick a right abelianisation $I:\mathpzc{E}\rightarrow\mathpzc{A}$, and use the fact that this is true for abelian categories. 
\end{proof}

\begin{cor}[c.f. \cite{estrada2023k} Lemma 3.2]
Let $\mathpzc{E}$ be a purely $\lambda$-accessible exact category in which transfinite compositions of $\lambda$-pure monomorphisms are admissible monomorphisms. Let 
$$0\rightarrow X\rightarrow Y\rightarrow Z\rightarrow 0$$
be a $\lambda$-pure exact sequence in $\mathrm{Ch}(\mathpzc{E})$. If $Y$ is acyclic then both $X$ and $Z$ are acyclic.
\end{cor}

\begin{proof}
As in \cite{estrada2023k} Lemma 3.2, let $E$ be a $\lambda$-presentable object of $\mathpzc{E}$. Then for each $n$ $S^{n}(E)$ is $\lambda$-presentable objects of $\mathrm{Ch}(\mathpzc{E})$. Thus 
$$0\rightarrow X\rightarrow Y\rightarrow Z\rightarrow 0$$
is a pure exact sequence in $\mathrm{Ch}(\mathpzc{E})$ then 
$$0\rightarrow\mathrm{Hom}(E,Z_{n}X)\rightarrow\mathrm{Hom}(E,Z_{n}Y)\rightarrow\mathrm{Hom}(E,Z_{n}Z)\rightarrow0$$
is an exact sequence, so
$$0\rightarrow Z_{n}X\rightarrow Z_{n}Y\rightarrow Z_{n}Z\rightarrow 0$$
is a pure exact sequence. The claim now follows from Lemma \ref{lem:Znsubexact}.  
\end{proof}

\begin{prop}\label{prop:acyclicintersctstronglypure}
Let $\mathpzc{E}$ be a purely $\lambda$-accessible exact category $\mathpzc{A}$ be a strongly $\lambda$-pure subobject stable class in $\mathrm{Ch}(\mathpzc{E})$. Suppose that the class $\mathpzc{A}\cap\mathfrak{W}$ of acyclic complexes in $\mathpzc{A}$ is closed under transfinite extensions. Then $\mathpzc{A}\cap\mathfrak{W}$ is also a strongly $\lambda$-pure subobject stable class.
\end{prop}

\begin{proof}
Since $\mathrm{Ch}(\mathpzc{E})$ is weakly $\mathbf{AdMon}_{\mathpzc{A}}$-elementary, it is certainly weakly $\mathpzc{A}\cap\mathfrak{W}$-elementary. Let $M\in\mathpzc{A}\cap\mathfrak{W}$, and $M_{\alpha\in A}\rightarrow M$ be a transfinite sequence of objects in $\mathpzc{E}_{\big\slash M}$ such that 
\begin{enumerate}
\item
for each successor ordinal $\alpha+1$, $M_{\alpha+1}\rightarrow M$ is an almost $(\mathpzc{A}\cap\mathfrak{W},\lambda)$-pure monomorphism.
\item
each $M_{\alpha}\rightarrow M_{\alpha+1}$ is a $\lambda$-pure monomorphism with cokernel in $\mathpzc{A}\cap\mathfrak{W}$.
\end{enumerate}
Now by the assumption that $\mathpzc{A}$ is strongly $\lambda$-pure subobject stable, we get an exact sequence
$$0\rightarrow\colim M_{\alpha}\rightarrow M\rightarrow\colim M\big\slash M_{\alpha}\rightarrow 0$$
with $M\big\slash M_{\alpha}\in\mathpzc{A}$. Now as a transfinite extension of objects in $\mathpzc{A}\cap\mathfrak{W}$, $\colim M_{\alpha}$ is in $\mathpzc{A}\cap\mathfrak{W}$. $M$ is also acyclic. By \cite{kelly2016homotopy} Lemma 4.2.36 $\colim M\big\slash M_{\alpha}$ is also acyclic. 
\end{proof}

\subsection{Cotorsion Pairs and Pre-Covering Classes}

In this Section we give the main application of our work on deconstructibility above, namely the construction of functorially complete cotorsion pairs. 

\subsubsection{Cotorsion Pairs}
We begin by recalling what cotorsion pairs are. As in \cite{kelly2016homotopy} we will largely follow the terminology and notation of \cite{vst2012exact}. Let $\mathcal{S}$ be a class of objects in an exact category $(\mathpzc{E},\mathpzc{Q})$. We denote by ${}^{\perp_{\mathpzc{Q}}}\mathcal{S}$ the class of all objects $X$ such that $\mathrm{Ext}^{1}(X,S)=0$ for all $S\in\mathcal{S}$, and by $\mathcal{S}^{\perp_{\mathpzc{Q}}}$ the class of all objects $X$ such that $\mathrm{Ext}^{1}(S,X)=0$ for all $S\in\mathcal{S}$. When it does not cause confusion we will just write these as ${}^{\perp}\mathcal{S}$ and $\mathcal{S}^{\perp}$ respectively. We also denote by $\mathbf{AdMon}_{\mathcal{S}}$ the class of admissible monomorphisms with cokernel in $\mathcal{S}$, and by $\mathbf{AdEpi}_{\mathcal{S}}$ the class of admissible epimorphisms with kernel in $\mathcal{S}$.

\begin{defn}
Let $\mathpzc{E}$ be an exact category. A \textit{cotorsion pair} on $\mathpzc{E}$ is a pair of families of objects $(\mathfrak{L},\mathfrak{R})$ of $\mathpzc{E}$ such that $\mathfrak{L}=\;^{\perp}\mathfrak{R}$ and $\mathfrak{R}=\mathfrak{L}^{\perp}$.
\end{defn}

\begin{defn}
A cotorsion pair $(\mathfrak{L},\mathfrak{R})$ is said to have \textit{enough (functorial) projectives} if for every $X\in\mathpzc{E}$ there is an admissible epic $p:Y\rightarrow X$, (functorial in $X$), such that $Y\in\mathfrak{L}$ and $\textrm{Ker}(p)\in\mathfrak{R}$. It is said to have \textit{enough (functorial) injectives} if, for every $X$, there is an admissible monic $i:X\rightarrow Z$, (functorial in $X$), such that  $Z\in\mathfrak{R}$ and $\textrm{Coker}(i)\in\mathfrak{L}$. A cotorsion pair is said to be \textit{(functorially) complete} if it has enough (functorial) projectives and enough (functorial) injectives.
\end{defn}

Constructing functorially complete cotorsion pairs is in general a difficult task. Fortunately \cite{saorin2011exact} and \cite{vst2012exact} have provided useful technqiues for doing so.

\subsubsection{Cotorsion Pairs Determined by a Set}

Let $\mathcal{S}$ be a set of objects in an exact category $\mathpzc{E}$. Suppose that $\mathpzc{E}$ is weakly $\mathbf{AdMon}_{\mathcal{S}}$ elementary. Let $\mathrm{Filt}(\mathcal{S})$ denote the class of objects obtained as transfinite extensions of objects in $\mathcal{S}$. 

For a given class $\mathcal{I}$ of admissible monomorphisms, denote by $\mathrm{Coker}(\mathcal{I})$ the class of objects obtained as cokernels of maps in $\mathcal{I}$. 

\begin{defn}[\cite{saorin2011exact} Definition 2.3]
Let $\mathpzc{E}$ be an exact category and $\mathcal{I}$ a class of admissible monomorphisms. We say that $\mathcal{I}$ is
\begin{enumerate}
\item
is \textit{homological} if the following conditions are equivalent for any object $T\in\mathpzc{E}$:
\begin{enumerate}
\item
$\mathrm{Ext}^{1}(S,T)=0$ for all $S\in\mathrm{coker}(\mathcal{I})$
\item
The map $T\rightarrow0$ belongs to $\mathcal{I}$-inj.
\end{enumerate}
\item
\textit{strongly homological} if given any $j:A\rightarrow B\in\mathbf{AdMon}_{\mathrm{coker}(\mathcal{I})}$, there is a morphism $i:A'\rightarrow B'$ in $\mathcal{I}$ giving rise to a commutative diagram whose rows are exact sequences
\begin{displaymath}
\xymatrix{
0\ar[r] & A'\ar[r]\ar[d] & B'\ar[r]\ar[d] & S\ar[d]\ar[r] & 0\\
0\ar[r] & A\ar[r] & B\ar[r] & S\ar[r] & 0
}
\end{displaymath}
\end{enumerate}
\end{defn}

The following is shown in \cite{saorin2011exact}.

\begin{prop}[\cite{saorin2011exact} Proposition 2.7 (2)]\label{prop:gens}
Let $\mathpzc{E}$ be an exact category with a set of generators aribtrary coproducts of which exist, and which is weakly idempotent complete. Then for each set of objects $\mathpzc{A}$ there is a strongly homological set of inflations $\mathcal{I}$ such that $\mathrm{Coker}(\mathcal{I})$ is the class of objects isomorphism to objects of $\mathpzc{A}$. Moreover if $\mathpzc{A}$ is closed under kernels of admissible epimorphisms between objects in $\mathpzc{A}$, and $\mathpzc{A}$ contains a generator, then $\mathcal{I}$ can be chosen such that the domains and codomains of maps in $\mathcal{I}$ are in $\mathpzc{A}$. Moreover, any admissible monomorphism with cokernel in $\mathpzc{A}$ is a pushout of a map in $\mathcal{I}$.
\end{prop}

\begin{rem}
The second claim of the above proposition concerning the domain and codomain being in $\mathpzc{A}$ is not stated explicitly in \cite{saorin2011exact} Proposition 2.7 (2), but may be immediately deduced from the constructive proof. 
\end{rem}

Strongly homological sets of morphisms are closely related to smallness of cotorsion pairs.
The following definitions are from \cite{kelly2016homotopy}, but as mentioned in loc. cit, they already appeared in much the same form in \cite{hovey} Section 6 for abelian categories, and \cite{gillespie2016derived} Section 5.2 for the $G$-exact structure on a Grothendieck abelian category.

\begin{defn}[\cite{kelly2016homotopy} Definition 4.1.13]
Let $\mathpzc{E}$ be an exact category. A cotorsion pair $(\mathfrak{L},\mathfrak{R})$ on $\mathpzc{E}$ is said to be \textit{cogenerated by a set} if there is a set of objects $\mathcal{G}$ in $\mathfrak{L}$ such that $X\in\mathfrak{R}$ if and only if $\textrm{Ext}^{1}(G,X)=0$ for all $G\in\mathcal{G}$.
\end{defn}

\begin{thm}[\cite{vst2012exact} Theorem 5.16]\label{thm:maincotorsexist}
Let $\mathcal{S}$ be a set of objects in an exact category $\mathpzc{E}$. Suppose that $\mathpzc{E}$ is weakly $\mathbf{AdMon}_{\mathcal{S}}$ elementary, and that there exists a strongly homological set of morphisms $\mathcal{I}$ such that
\begin{enumerate}
\item
$\mathrm{Filt}(\mathcal{S})$ contains a generator for $\mathpzc{E}$.
\item
$\mathrm{Coker}(\mathcal{I})=\mathcal{S}$
\item
domains and codomains of maps in $\mathcal{I}$ are small relative to $\mathcal{I}$-cell.
\end{enumerate}
Then $(\mathrm{Filt}(\mathcal{S}),\mathcal{S}^{\perp})$ is a functorially complete small cotorsion pair
\end{thm}

\begin{proof}
The fact that it is a functorially complete cotorsion pair is proved identically to \cite{vst2012exact} Theorem 5.16. The fact that it is small is tautological.
\end{proof}

Since in an accessible category all objects are small with repsect to all classes of morphisms, we have the following immediate corollary using Proposition \ref{prop:gens}.

\begin{cor}
Let $\mathcal{S}$ be a set of objects in a purely $\lambda$-accessible exact category $\mathpzc{E}$. Suppose that $\mathpzc{E}$ is weakly $\mathbf{AdMon}_{\mathcal{S}}$ elementary and that $\mathrm{Filt}(\mathcal{S})$ contains a generator. Then $(\mathrm{Filt}(\mathcal{S}),\mathcal{S}^{\perp})$ is a functorially complete small cotorsion pair.
\end{cor}

\begin{cor}\label{thm:cotorsexist}
Suppose that a class $\mathpzc{A}$ is presentably deconstructible in itself relative to $\mathbf{AdMon}$, contains a generator, and is closed under transfinite extensions by admissible monomorphisms. Then $(\mathpzc{A},\mathpzc{A}^{\perp})$ is a small functorially complete cotorsion pair.
\end{cor}

\begin{proof}
Let $\mathbf{J}_{\mathpzc{A}}$ be the class of maps from Definition \ref{defn:deconstrexact}. Apply Theorem \ref{thm:maincotorsexist} to $\mathcal{S}=\mathrm{Coker}(\mathbf{J}_{\mathpzc{A}})$. 
\end{proof}

Then we immediately get the corollary which we will primarily use.


\begin{cor}\label{cor:cotorsaccess}
Let $\mathpzc{E}$ be a purely $\lambda$-accessible exact category. Let $\mathpzc{A}$ be a strongly $\lambda$-pure subobject stable class of objects in $\mathpzc{E}$ which is closed under transfinite extensions by admissible monomorphisms, contains a generator, and is thick. Suppose that $\mathpzc{E}$ is weakly $\mathbf{AdMon}_{\mathpzc{A}}$-elementary. Then $(\mathpzc{A},\mathpzc{A}^{\perp})$ is a complete cotorsion pair.
\end{cor}

\subsubsection{Pre-covering and Pre-enveloping Classes}

The proof of Theorem \ref{thm:maincotorsexist} is at its core just an application of the small object argument - the assumptions just allow us to run such an argument. Closely related to this is establishing whether classes of objects are (pre-)covering. Let $\mathpzc{A}$ be a class of objects in an category $\mathpzc{E}$. Recall that an $\mathpzc{A}$-\textit{precover} of an object $E$ is a map $A\rightarrow E$ with $A\in\mathpzc{A}$ such that $\mathrm{Hom}(A',A)\rightarrow\mathrm{Hom}(A',E)$ is an epimorphism for any $A'\in\mathpzc{A}$, and an $\mathpzc{A}$-\textit{preenvelope} of an object $E$ is a map $E\rightarrow A$ such that $A\in\mathpzc{A}$ and $\mathrm{Hom}(A,A')\rightarrow\mathrm{Hom}(E,A')$ is an epimorphism for all $A'\in\mathpzc{A}$. If $\mathpzc{E}$ has an exact structure, then an $\mathpzc{A}$-special precover is an admissible epimorphism $\pi:A\rightarrow\mathpzc{E}$ with $A\in\mathpzc{A}$ and $\mathrm{Ker}(\pi)\in\mathpzc{A}^{\perp}$, and an $\mathpzc{A}$-special preenvelope is an admissible monomorphism $i:E\rightarrow A$ such that $A\in\mathpzc{A}$ and $\mathrm{coker}(i)\in{}^{\perp}\mathpzc{A}$. Note that by the long exact sequence on $\mathrm{Ext}$, $\mathpzc{A}$-special precovers (resp. $\mathpzc{A}$-special preenvelopes) are $\mathpzc{A}$-precovers (resp. $\mathpzc{A}$-preenvelopes). 

\begin{prop}
Suppose that a class $\mathpzc{A}$ is presentably deconstructible in itself relative to a class of admissible monomorphisms $\mathbf{I}$
\begin{enumerate}
\item
Any object $M$ of $\mathpzc{E}$ admits an $\mathpzc{A}$-precover.
\item
If $\mathpzc{A}$ generates $\mathpzc{E}$ and $\mathpzc{E}$ is weakly $\mathbf{I}_{\mathpzc{A}}$-elementary then every object $M$ of $\mathpzc{E}$ admits a special $\mathpzc{A}$-precover and a special $\mathpzc{A}^{\perp}$-envleope. 
\end{enumerate}
\end{prop}

\begin{proof}
\begin{enumerate}
\item
By the small object argument we may factor any map $0\rightarrow M$ as $0\rightarrow A\rightarrow M$ with $0\rightarrow A\in\mathbf{I}_{\mathpzc{A}}$-cell and $A\rightarrow M$ having the right-lifting property with respect to all maps in $\mathbf{I}_{\mathpzc{A}}$. Note that $A\in\mathpzc{A}$. Moreover any map of the form $0\rightarrow B$ with $B\in\mathpzc{A}$ is in $\mathbf{I}_{\mathpzc{A}}$. Thus the map $\mathrm{Hom}(B,A)\rightarrow\mathrm{Hom}(B,M)$ is an epimorphism. 
\item
In this case $\mathbf{I}_{\mathpzc{A}}$-cell consists precisely of admissible monomorphisms whose cokernel is in $\mathpzc{A}$. The proof proceeds identically to \cite{saorin2011exact} Theorem 2.13.
\end{enumerate}
\end{proof}

\begin{defn}[\cite{kelly2016homotopy} Definition 4.1.14]
Suppose $\mathpzc{E}$ is an exact category. A cotorsion pair $(\mathfrak{L},\mathfrak{R})$ is said to be \textit{small} if the following conditions hold
\begin{enumerate}
\item
$\mathfrak{L}$ contains a set of admissible generators of $\mathpzc{E}$.
\item
$(\mathfrak{L},\mathfrak{R})$ is cogenerated by a set $\mathcal{G}$.
\item
For each $G\in\mathcal{G}$ there is an admissible monic $i_{G}$ with cokernel $G$ such that, if $\textrm{Hom}_{\mathpzc{E}}(i_{G},X)$ is surjective for all $G\in\mathcal{G}$, then $X\in\mathfrak{R}$.
\end{enumerate}
The set of $i_{G}$ together with the maps $0\rightarrow U_{i}$ for some generating set $\{U_{i}\}$ contained in $\mathfrak{L}$ is called a set of \textit{generating morphisms} of $(\mathfrak{L},\mathfrak{R})$.
\end{defn}

\section{Exact Model Structures}\label{sec:exactmodelstructures}

In this chapter we will apply the results of Chapter 1 to construct model structures on exact categories.

\subsection{Weak Factorisation Systems and Cotorsion Pairs}

We begin by recalling the general theory of exact model structures and cotorsion pairs, developed in \cite{kelly2016homotopy}, \cite{vst2012exact}, \cite{gillespie2016exact}, \cite{gillespie2016derived}, \cite{gillespie}, \cite{Gillespie2} following the seminal work of \cite{hovey} in the abelian case. Let $\mathcal{S}$ be a class of maps in a category $\mathpzc{E}$. Denote by $\mathcal{S}^{\llp}$ the class of maps which have the right lifting property with respect to $\mathcal{S}$, and ${}^{\llp}\mathcal{S}$ the class of maps which have the left lifting proerty with respect to $\mathcal{S}$.

\begin{defn}
A \textit{(functorial) weak factorisation system} on a category $\mathpzc{E}$  is a pair $(\mathcal{L},\mathcal{R})\subset\mathrm{Mor}(\mathpzc{E})\times\mathrm{Mor}(\mathpzc{E})$ of classes of maps in $\mathpzc{E}$ such that 
\begin{enumerate}
\item
any map $g:X\rightarrow Y$ in $\mathpzc{E}$ (functorially) factors as $f\circ c$, where $c:X\rightarrow Z\in\mathcal{L}$, and $f:Z\rightarrow Y\in\mathcal{R}$.
\item
$\mathcal{L}={}^{\llp}\mathcal{R}$, $\mathcal{R}=\mathcal{L}^{\llp}$.
\end{enumerate}
\end{defn}

Now let $\mathpzc{E}$ be an exact category. 

\begin{defn}
Let $\mathpzc{E}$ be an exact category. A  weak factorisation system $(\mathcal{L},\mathcal{R})$ on $\mathpzc{E}$ is said to be \textit{compatible} if
\begin{enumerate}
\item
$f\in\mathcal{L}$ if and only if $f$ is an admissible monic and $0\rightarrow\textrm{Coker}(f)$ belongs to $\mathcal{L}$.
\item
$f\in\mathcal{R}$ if and only if $f$ is an admissible epic and $\textrm{Ker}(f)\rightarrow 0$ belongs to $\mathcal{R}$.
\end{enumerate}
\end{defn}

There is a correspondence between compatible weak factorsiation systems and cotorsion pairs.

\begin{thm}[\cite{vst2012exact} Theorem 5.13]
Let $\mathpzc{E}$ be an exact category. Then
$$(\mathcal{L},\mathcal{R})\mapsto(\textrm{Coker}\mathcal{L},\textrm{Ker}\mathpzc{R})\textrm{ and }(\mathfrak{A},\mathfrak{B})\mapsto(\mathbf{AdMon}_{\mathfrak{A}},\mathbf{AdEpi}_{\mathfrak{B}})$$
define mutually inverse bijective mappings between compatible weak factorisation systems and complete cotorsion pairs. The bijections restrict to mutually inverse mappings between compatible functorial weak factorisation systems and functorially complete cotorsion pairs.
\end{thm}

\subsubsection{Hovey Triples and Model Structures}

Let $(\mathcal{C},\mathcal{F},\mathcal{W})$ be a model structure on a category $\mathpzc{E}$. This precisely means that $(\mathcal{C}\cap\mathcal{W},\mathcal{F})$ and $(\mathcal{C},\mathcal{W}\cap\mathcal{F})$ are weak factorisation systems.

For abelian categories, Hovey \cite{hovey} made precise the relationship between model category structures and homological data - Hovey triples.

\begin{defn}
Let $\mathpzc{E}$ be an exact category. Let $(\mathcal{C},\mathcal{F},\mathcal{W})$ be a model structure on $\mathpzc{E}$. The model structure is said to be \textit{compatible} if both $(\mathcal{C}\cap\mathcal{W},\mathcal{F})$ and $(\mathcal{C},\mathcal{F}\cap\mathcal{W})$ are compatible weak factorisation systems.
\end{defn}

As for weak factorisation systems, there is corresponding homological data. We will call a subcategory $\mathpzc{D}$ of an exact category $\mathpzc{E}$ \textit{thick} if whenever
$$0\rightarrow A\rightarrow B\rightarrow C\rightarrow 0$$
is a short exact sequence and two of the objects are in $\mathpzc{D}$, then so is the third.

\begin{defn}
A \textit{Hovey triple} on an exact category $\mathpzc{E}$ is a triple $(\mathfrak{C},\mathfrak{W},\mathfrak{F})$ of collections of objects of $\mathpzc{E}$ such that the full subcategory on $\mathfrak{W}$ is closed under retracts and thick, and that both $(\mathfrak{C},\mathfrak{F}\cap\mathfrak{W})$ and $(\mathfrak{C}\cap\mathfrak{W},\mathfrak{F})$ are complete cotorsion pairs. 
\end{defn}

We then have the following theorem (Theorem 6.9 in \cite{vst2012exact}). It is originally due to \cite{hovey} in the abelian case and \cite{gillespie} Theorem 3.3 in the more general exact case.

\begin{thm}[ \cite{vst2012exact} Thoerem 6.9]\label{weakmon}
Let $\mathpzc{E}$ be a weakly idempotent complete exact category. Then there is a bijection between Hovey triples and compatible model structures. The correspondence assigns to a Hovey triple $(\mathfrak{C},\mathfrak{W},\mathfrak{F})$ the model structure $(\mathcal{C},\mathcal{W},\mathcal{F})$ such that
\begin{enumerate}
\item
$\mathcal{C}=\mathbf{AdMon}_{\mathfrak{C}}$
\item
$\mathcal{F}=\mathbf{AdEpi}_{\mathfrak{F}}$
\item
$\mathcal{W}$ consists of morphisms of the form $p\circ i$ where $i\in\mathbf{AdMon}_{\mathfrak{W}}$ and $p\in\mathbf{AdEpi}_{\mathfrak{W}}$.
\end{enumerate}
\end{thm}

When dealing with model structures on non-negatively graded complexes, we will also need the following definition from \cite{kelly2016homotopy} concerning model structures on exact categories which are only `half-compatible' in a precise sense.

\begin{defn}[\cite{kelly2016homotopy}, Definition 4.1.12]
Let $\mathpzc{E}$ be an exact category. A model structure $(\mathcal{C},\mathcal{F},\mathcal{W})$ on $\mathpzc{E}$ is said to be \textit{left pseudo-compatible} if there are classes of objects $\mathfrak{C}$ and $\mathfrak{W}$  such that
\begin{enumerate}
\item
The full subcategory on $\mathfrak{W}$ is thick.
\item
 A map $f$ is in $\mathcal{C}$ (resp. $\mathcal{C}\cap\mathcal{W}$) if and only if it is an admissible monic with cokernel in $\mathfrak{C}$ (resp. $\mathfrak{C}\cap\mathfrak{W}$). 
 \item
 An admissible monic is in $\mathcal{W}$ if and only if its cokernel is in $\mathfrak{W}$.
 \end{enumerate}
 As before $\mathfrak{C}$/ $\mathfrak{W}$ /$\mathfrak{C}\cap\mathfrak{W}$ are called the \textit{cofibrant }/\textit{trivial}/  \textit{trivially cofibrant} objects. The pair $(\mathfrak{C},\mathfrak{W})$ will be called the \textit{left homological Waldhausen pair} of the model structure. Dually one defines \textit{right pesudo-compatible} model structures and  \textit{right homological Waldhausen pairs}
 \end{defn}

 \begin{prop}\label{prop:rlpgenepi}
  Let $\mathpzc{E}$ be an exact category, and let $\mathfrak{C}$ be a generating subcategory of $\mathpzc{E}$ closed under direct summands. Let $\tilde{\mathcal{F}}=\mathbf{AdMon}_{\mathfrak{C}}^{\llp}$. Then $\tilde{\mathcal{F}}$ consists of admissible epimorphisms. 
 \end{prop}
 
 \begin{proof}
 $\mathbf{AdMon}_{\mathfrak{C}}^{\llp}$ in particular contains maps of the form $0\rightarrow C$ for $C\in\mathfrak{C}$. If $f:X\rightarrow Y$ is a map in $\mathbf{AdMon}_{\mathfrak{C}}^{\llp}$ then $\mathrm{Hom}(C,X)\rightarrow\mathrm{Hom}(C,Y)$ is an epimorphism. By \cite{kelly2016homotopy} $f$ is an admissible epimorphism.
 \end{proof}
%
%
%
 \begin{prop}\label{prop:condad}
 Let $\mathpzc{E}$ be a weakly idempotent complete exact category, and let $(\mathcal{L},\mathcal{R})$ be a weak factorisation system on $\mathpzc{E}$. If $\mathcal{L}=\mathbf{AdMon}_{\mathfrak{L}}$ for some generating subcategory $\mathfrak{L}$, then the weak factorisation system is compatible.
 \end{prop}
 
 \begin{proof}
 By Proposition \ref{prop:rlpgenepi} $\mathcal{R}$ consists of admissible epimorphisms. Let $\mathfrak{R}=\mathrm{Ker}(\mathcal{R})$. We claim that $\mathcal{R}=\mathbf{AdEpi}_{\mathfrak{R}}$. This follows immediately from \cite{vst2012exact} Lemma 5.14 and Lemma 5.15.
 \end{proof}

\begin{thm}[c.f. \cite{lurie2006higher} Proposition A.2.6.8]\label{thm:amendingmodel}
Let $\mathpzc{E}$ be an exact category, $\mathfrak{C}$ a class of objects in $\mathpzc{E}$, and $\mathcal{W}$ a class of maps in $\mathpzc{E}$ satisfying the $2$-out-of-$3$ property. Denote by $\mathfrak{W}$ the class of objects $X$ such that $0\rightarrow X$ is in $\mathcal{W}$.  Suppose that
\begin{enumerate}
\item
$\mathfrak{C}$ and $\mathfrak{C}\cap\mathfrak{W}$ are closed under transfinite extensions.
\item
$\mathpzc{E}$ is weakly $\mathbf{AdMon}_{\mathfrak{C}}$-elementary.
\item
$\mathfrak{C}$ and $\mathfrak{C}\cap\mathfrak{W}$ are presentably deconstructible in themselves
\item
$\mathcal{W}$ satisfies the two-out-of-three property and is closed under retracts.
\item
$\mathbf{AdMon}_{\mathfrak{C}}^{\llp}\subset\mathcal{W}$
\item
$\mathbf{AdMon}_{\mathfrak{C}}\cap\mathcal{W}=\mathbf{AdMon}_{\mathfrak{C}\cap\mathfrak{W}}$.
\end{enumerate}
Then there exists a model structure on $\mathpzc{E}$ with cofibrations given by $\mathbf{AdMon}_{\mathfrak{C}}$ and acyclic cofibrations given by $\mathbf{AdMon}_{\mathfrak{C}\cap\mathfrak{W}}$. It is left pseudo-compatible if $\mathfrak{C}$ contains a generator, and it is compatible if $\mathfrak{C}\cap\mathfrak{W}$ contains a generator.
\end{thm}

\begin{proof}
Clearly $\mathcal{C}$ and $\mathcal{C}\cap\mathcal{W}$ are closed under pushouts. Since by assumption $\mathfrak{C}$ and $\mathfrak{C}\cap\mathfrak{W}$ are closed under summands and under transfinite extensions, and $\mathpzc{E}$ is weakly $\mathbf{AdMon}_{\mathfrak{C}}$-elementary, $\mathcal{C}$ and $\mathcal{C}\cap\mathcal{W}$ are weakly saturated. By the small object argument we may factor any map $f:X\rightarrow Y$ as $f\circ c$ and $\tilde{f}\circ \tilde{c}$ where $c\in\mathbf{AdMon}_{\mathfrak{C}}$, $f\in\mathbf{AdMon}_{\mathfrak{C}}^{\llp}$, $\tilde{c}\in\mathbf{AdMon}_{\mathfrak{C}\cap\mathfrak{W}}$, and $\tilde{f}\in\mathbf{AdMon}_{\mathfrak{C}\cap\mathfrak{W}}^{\llp}$. Next we show that trivial fibrations have the right lifting property with respect to cofibrations. Let $c$ be an acyclic fibration. We factor it as $p\circ i$ with $i\in\mathbf{AdMon}_{\mathfrak{C}}$, and $p\in\mathbf{AdMon}_{\mathfrak{C}}^{\llp}$. Then $p$ is in $\mathcal{W}$ by assumption, and by the $2$-out-of-$3$ assumption on $\mathcal{W}$, $i$ is also in $\mathcal{W}$. $c$ must then be a retract of $p$, so $c$ also has the right lifting property with respect to all cofibrations.

It remains to show that the model structure is left pseudo-compatible if $\mathfrak{C}$ contains a generator, and compatible if  $\mathfrak{C}\cap\mathfrak{W}$ contains a generator.  This follows immediately from Proposition \ref{prop:condad}.
\end{proof}

Let us give some immediate applications of this result.

\begin{cor}[Changing the Cofibrant Objects] \label{cor:Quillenequivalentmodel}
Let $\mathpzc{E}$ be a weakly idempotent complete exact category equipped with a left pseudo-compatible model structure defined by $(\mathfrak{C},\mathfrak{W})$. Suppose that $\mathfrak{C}'\subset\mathpzc{E}$ is closed under transfinite extensions by admissible monomorphisms, that $\mathpzc{E}$ is weakly $\mathbf{AdMon}_{\mathfrak{C}'}$-elementary, and that $\mathfrak{C}\subseteq\mathfrak{C}'$ and both $\mathfrak{C}'$ and $\mathfrak{C}'\cap\mathfrak{W}$ are presentably deconstructible in themselves. Then there exists a left pseudo-compatible model structure on $\mathpzc{E}$ determined by $(\mathfrak{C}',\mathfrak{W})$ which is Quillen equivalent to the one determined by $(\mathfrak{C},\mathfrak{W})$. 
\end{cor}

\begin{proof}
By assumption $\mathbf{AdMon}_{\mathfrak{C}'}$ and  $\mathbf{AdMon}_{\mathfrak{C}'\cap\mathfrak{W}}$ are closed under transfinite compositions, and both $\mathfrak{C}'$ and $\mathfrak{C}'\cap\mathfrak{W}$ are presentably deconstructible in themselves. The class $\mathcal{W}$ satisfies the two-out-of-three property and is closed under retracts by the assumption that they form the weak equivalences of the model structure determined by $(\mathfrak{C},\mathfrak{W})$. Since $\mathfrak{C}\subset\mathfrak{C}'$, we have $\mathbf{AdMon}_{\mathfrak{C}'}^{\llp}\subseteq \mathbf{AdMon}_{\mathfrak{C}}^{\llp}\subseteq\mathcal{W}$. Finally, since $\mathbf{AdMon}\cap\mathcal{W}=\mathbf{AdMon}_{\mathfrak{W}}$, we have $\mathbf{AdMon}_{\mathfrak{C}}\cap\mathcal{W}=                \mathbf{AdMon}_{\mathfrak{C}'\cap\mathfrak{W}}$
\end{proof}

\begin{cor}[Going Up]\label{cor:goingup}
Let $\mathpzc{E}$ be a cocomplete exact category and $\mathpzc{D}$ a thick reflective subcategory, such that the inclusion $i:\mathpzc{D}\rightarrow\mathpzc{E}$ commutes with transfinite compositions of admissible monomorphisms. Let $\mathcal{W}$ be a class of weak equivalences on $\mathpzc{E}$ satisfying the two-out-of-three property. Let $\mathfrak{W}$ be the class of objets such that $0\rightarrow X$ is in $\mathfrak{W}$, and suppose that an admissible monomorphism $i:A\rightarrow B$ is in $\mathcal{W}$ if and only if $\mathrm{coker}(i)\in\mathfrak{W}$.

Let $(\mathfrak{C},\mathfrak{W}_{\mathpzc{D}})$ be a Waldhausen pair defining a left peudo-compatible model structure on $\mathpzc{D}$ with corresponding weak equivalences given by $\mathcal{W}\cap\mathfrak{D}$, and such that $\mathfrak{C}$ and $\mathfrak{W}_{\mathpzc{D}}$ are presentably deconstructible in themselves. 

 Suppose that $\mathpzc{E}$ is weakly $\mathbf{AdMon}_{\mathfrak{C}}$-elementary. Finally suppose that $\mathpzc{D}$ generates $\mathpzc{E}$ ,i.e. that for every $E$ in $\mathpzc{E}$ there is a $D\in\mathpzc{D}$ and an admissible epimorphism $D\rightarrow\mathpzc{E}$. Then $(\mathfrak{C},\mathfrak{W}_{\mathpzc{E}})$ is a Waldhausen pair on $\mathpzc{E}$, and the corresponding weak equivalences are precisely $\mathcal{W}$.  Moreover when $\mathpzc{E}$ is equipped with the correspoding model structure induced by $(\mathfrak{C},\mathfrak{W}_{\mathpzc{E}})$, and $\mathpzc{D}$ equipped with the model structure induced by the Waldhausen pair $(\mathfrak{C},\mathfrak{W}_{\mathpzc{D}})$, then 
$$\adj{L}{\mathpzc{E}}{\mathpzc{D}}{i}$$
is a Quillen equivalence.


\end{cor}

\begin{proof}
$\mathfrak{C}\cap\mathfrak{W}_{\mathpzc{E}}$ and $\mathfrak{C}$ are presentably deconstructible in themselves in $\mathpzc{D}$. Since $i$ commutes with transfinite compositions of admissible monomorphisms, they are also presentably deconstructible in themselves in $\mathpzc{E}$. Moreover we have $\mathfrak{W}_{\mathpzc{E}}\cap\mathpzc{D}=\mathfrak{W}_{\mathpzc{D}}$. By assumption we have $\mathbf{AdMon}_{\mathfrak{C}\cap\mathfrak{W}_{\mathpzc{E}}}=\mathbf{AdMon}_{\mathfrak{C}}\cap\mathcal{W}$. The existence of the model structure follows immediately from Theorem \ref{thm:amendingmodel}.

The functor $L$ is left derivable, by taking resolutions by objects of $\mathpzc{D}$. In particular, objects of $\mathpzc{D}$ are $L$-acyclic, and if
$$0\rightarrow X\rightarrow Y\rightarrow D\rightarrow 0$$
is a short eact sequence in $\mathpzc{E}$ with $D\in\mathpzc{D}$, then
$$0\rightarrow L(X)\rightarrow L(Y)\rightarrow D\rightarrow 0$$
is exact in $\mathpzc{D}$. 
It follows immediately that the adjunction
$$\adj{L}{\mathpzc{E}}{\mathpzc{D}}{i}$$
is Quillen. It is clearly a Quillen equivalence.
\end{proof}

We can also use Theorem \ref{thm:amendingmodel} to amend the underlying exact structure, in the following precise sense.

 \begin{cor}[Changing the Underlying Exact Structure]\label{cor:changingtheunderlyingexactstructure}\label{prop:changingunderling}
Let $(\mathpzc{E},\mathpzc{Q})$ and $(\mathpzc{E},\mathpzc{Q}')$ be purely $\lambda$-accessible exact categories, with the same underlying category $\mathpzc{E}$, and with $\mathpzc{Q}\subset\mathpzc{Q}'$. Suppose that 
\begin{enumerate}
\item
$(\mathpzc{E},\mathpzc{Q}')$ is equipped with a left pseudo-compatible model structure determined by a Waldhausen pair $(\mathfrak{C}',\mathfrak{W}')$, with class of weak equivalences $\mathcal{W}'$
\item
 both $\mathfrak{C}'$ and $\mathfrak{C}'\cap\mathfrak{W}'$ are strongly $\lambda$-pure subobject stable in $(\mathpzc{E},\mathpzc{Q})$.
 \item
 $\mathfrak{C}'$ and $\mathfrak{C}'\cap\mathfrak{W}'$ are closed under transfinite extensions in $(\mathpzc{E},\mathpzc{Q})$
 \item
 $(\mathpzc{E},\mathpzc{Q})$ is weakly $\mathbf{AdMon}_{\mathfrak{C}'}$-elementary.
\item
$\mathbf{AdMon}_{\mathfrak{C}'}^{\llp}$ is in the class of weak equivalences $\mathcal{W}'$
\item
$\mathfrak{C}'$ contains a generator for $(\mathpzc{E},\mathpzc{Q})$.
\end{enumerate}
Then $(\mathfrak{C}',\mathfrak{W}')$ also determines a left pseudo-compatible model structure on $(\mathpzc{E},\mathpzc{Q})$. Moreover this model structure is Quillen equivalent to the original one on $(\mathpzc{E},\mathpzc{Q}')$. 
\end{cor}

\begin{proof}
By assumption $\mathbf{AdMon}_{\mathfrak{C}'}$ and $\mathbf{AdMon}_{\mathfrak{C}'\cap\mathfrak{W}'}$ are closed under transfinite compositions in $(\mathpzc{E},\mathpzc{Q})$. Moreover in $(\mathpzc{E},\mathpzc{Q}')$ we have $\mathbf{AdMon}_{\mathfrak{C}'}\cap\mathcal{W}=\mathbf{AdMon}_{\mathfrak{C}'\cap\mathfrak{W}'}$, this is clearly also the case in $(\mathpzc{E},\mathpzc{Q})$, since $\mathpzc{Q}\subseteq\mathpzc{Q}'$. 

Since $\mathfrak{C}'$ and $\mathfrak{W}'$ are strongly $\lambda$-pure subobject stable they are presentably deconstructible in themesleves in $(\mathpzc{E},\mathpzc{Q})$.  By assumption $\mathcal{W}'$ is part of a model structure so satisfies the $2$-out-of-$3$ property. By assumption $\mathbf{AdMon}_{\mathfrak{C}'}^{\llp}\subseteq\mathcal{W}'$. Finally  $(\mathpzc{E},\mathpzc{Q})$ is weakly $\mathbf{AdMon}_{\mathfrak{C}'}$-elementary by assumption. This proves the existence of the model structure.

The fact that $\mathfrak{C}'$ contains a generator for $(\mathpzc{E},\mathpzc{Q})$ means that the cofibration -acyclic fibration weak factorisation system is determined by the complete cotorsion pair $(\mathfrak{C}',(\mathfrak{C}')^{\perp_{\mathpzc{Q}}})$.
%
\end{proof}
\subsubsection{Injective Cotorsion Pairs}

Many of the results of this paper will involve the construction of injective cotorsion pairs and their associated model structures. We use the results from the previous section to prove the existenec of injective cotorsion pairs.

\begin{defn}[\cite{MR3459032} Definition 3.4]
Let $\mathpzc{E}$ be an exact category with enough injectives. A complete cotorsion pair $(\mathfrak{W},\mathfrak{F})$ on $\mathpzc{E}$ is said to be an \textit{injective cotorsion pair} if $\mathfrak{W}$ is thick and $\mathfrak{W}\cap\mathfrak{F}$ coincides with the class of injective objects in $\mathpzc{E}$.
\end{defn}

The following is essentially tautological from the assumption that $\mathpzc{E}$ has enough injectives, and $\mathfrak{W}\cap\mathfrak{F}$ coincides with the class of injectives.

\begin{prop}
Let $(\mathfrak{W},\mathfrak{F})$ be an injective cotorsion pair on $\mathpzc{E}$. Then $(\mathpzc{E},\mathfrak{W},\mathfrak{F})$ is a Hovey triple on $\mathpzc{E}$. 
\end{prop}
 
The following is an immediate consequence of Corollary \ref{cor:cotorsaccess}.

\begin{prop}
Let $\mathpzc{E}$ be a purely $\lambda$-accessible exact category. Let $\mathpzc{K}$ be a strongly $\lambda$-pure subobject stable class in $\mathpzc{E}$ which is thick, is closed under transfinite extensions by admissible monomorphisms, contains all injectives, and contains a generator. Suppose that $\mathpzc{E}$ is weakly $\mathbf{AdMon}_{\mathpzc{K}}$-elementary. Then $(\mathpzc{K},\mathpzc{K}^{\perp})$ is an injective cotorsion pair on $\mathpzc{E}$. 
\end{prop}

\subsection{Properties of Compatible Model Structures}

Using Hovey triples, many properties of compatible model structures can be understood and analysed through the lens of homological algebra.

\subsubsection{Cofibrant Generation}

The notion of cofibrant generation on the model category side corresponds on the homological side to the cotorsion pairs being small.

\begin{lem}[\cite{kelly2016homotopy} Lemma 4.1.16]\label{cofibgen}
Let $\mathpzc{E}$ be a weakly idempotent complete exact category together with a compatible weak factorisation system $(\mathcal{L},\mathcal{R})$ with corresponding cotorsion pair $(\mathfrak{L},\mathfrak{R})$. If the cotorsion pair is small with generating morphisms $I=\{0\rightarrow U_{i}\}\cup\{i_{G}\}$, then this weak factorisation system is cofibrantly small in the sense of \cite{kelly2016homotopy} Definition A.2.9. If in addition the generating morphisms have $cell(I)$-presented domain, the weak factorisation system is cofibrantly generated. If finally $\mathpzc{E}$ is locally presentable then the model structure is combinatorial.
\end{lem}

\subsubsection{Monoidal Model Structures}

Let us recall some material from \cite{kelly2016homotopy} concerning monoidal model structure on exact categories.


\begin{thm}\label{exactmonoidal}
Let $\mathpzc{E}$ be a closed symmetric monoidal exact category equipped with a left pseudo-compatible model structure in which cofibrations are $\otimes$-pure. Then 
\begin{enumerate}
\item
$\mathpzc{E}$ is an almost monoidal model category.
\item
If $X\otimes C$ is acyclic for any trivially cofibrant $X$ and any cofibrant $C$ then $\mathpzc{E}$ is a weak monoidal model category. If in addition $\mathpzc{E}$ is weakly $\mathbf{PureMon}_{\otimes}$-elementary then $\mathpzc{E}$ satisfies the pp-monoid axiom.
\item
If $C\otimes C'$ is cofibrant for any cofibrant objects $C,C'$, and is acyclic whenever $C$ or $C'$ is acyclic, then $\mathpzc{E}$ is $C$-monoidal. If in addition whenever $C\rightarrow k$ is an acyclic fibration with $C$ in $\mathfrak{C}$, then for any object $X$ of $\mathpzc{E}$, $C\otimes X \rightarrow X$ is a weak equivalence, $\mathpzc{E}$ is a monoidal model category.
\item
If $X\otimes C$ is acyclic for any trivially cofibrant $X$ and any object $C$, and $\otimes$-pure transfinite extensions of acyclic objects are acyclic then $\mathpzc{E}$ satisfies the monoid axiom.
\end{enumerate}
\end{thm}

\begin{proof}
\begin{enumerate}
\item
By Proposition \ref{prop:pushoutprodpure} any pushout product of $\otimes$-pure monomorphism is $\otimes$-pure, and hence admissible, monomorphism. Admissible monomorphisms are left proper by \cite{kelly2016homotopy} Proposition 4.2.45.
\item
This follows immediately from  Proposition \ref{prop:pushoutprodpure}.
\item
This follows immediately from  Proposition \ref{prop:pushoutprodpure}. Assuming the property of cofibrant resolutions of the unit, the claim is Theorem 4.1.17 in \cite{kelly2016homotopy}, which is essentially \cite{vst2012exact} Theorem 8.11
\item
This is Theorem 4.1.18 in \cite{kelly2016homotopy}
\end{enumerate}
\end{proof}

\subsubsection{Quillen Adjunctions}

The property of an adjunction between exact model categories being Quillen can also be translated into the language of cotorsion pairs.

\begin{lem}\label{lem:420}
Let $(\mathfrak{L},\mathfrak{R})$ be a complete cotorsion pair on an exact category $\mathpzc{E}$, and let 
$$0\rightarrow X\rightarrow Y\rightarrow Z\rightarrow 0$$
be a sequence in $\mathpzc{E}$ with $Z\in\mathfrak{L}$. Then the sequence is exact if and only if for any $Q\in\mathfrak{R}$, the sequence
$$0\rightarrow\mathrm{Hom}(Z,Q)\rightarrow\mathrm{Hom}(Y,Q)\rightarrow\mathrm{Hom}(X,Q)\rightarrow 0$$
is exact in $\mathrm{Ab}$.
\end{lem}

\begin{proof}
Suppose the sequence 
$$0\rightarrow X\rightarrow Y\rightarrow Z\rightarrow 0$$
is exact. We get a long exact sequence
$$0\rightarrow\mathrm{Hom}(Z,Q)\rightarrow\mathrm{Hom}(Y,Q)\rightarrow\mathrm{Hom}(X,Q)\rightarrow\mathrm{Ext}^{1}(Z,Q)\rightarrow\ldots$$
Since $\mathrm{Ext}^{1}(Z,Q)=0$ we get that 
$$0\rightarrow\mathrm{Hom}(Z,Q)\rightarrow\mathrm{Hom}(Y,Q)\rightarrow\mathrm{Hom}(X,Q)\rightarrow 0$$
is exact. Conversely, suppose that for each $Q\in\mathfrak{R}$ the sequence
$$0\rightarrow\mathrm{Hom}(Z,Q)\rightarrow\mathrm{Hom}(Y,Q)\rightarrow\mathrm{Hom}(X,Q)\rightarrow 0$$
is exact. Since the cotorsion pair is complete, there are enough $\mathfrak{R}$-objects. The dual of \cite{kelly2016homotopy} Proposition 2.6.93 implies that 
$$0\rightarrow X\rightarrow Y\rightarrow Z\rightarrow 0$$
is exact.
\end{proof}

\begin{cor}\label{cor:adjcotors}
Let 
$$\adj{L}{\mathpzc{D}}{\mathpzc{E}}{R}$$
be an adjunction of exact categories. Let $(\mathfrak{L}_{\mathpzc{D}},\mathfrak{R}_{\mathpzc{D}})$ and $(\mathfrak{L}_{\mathpzc{E}},\mathfrak{R}_{\mathpzc{E}})$ be complete cotorsion pairs on $\mathpzc{D}$ and $\mathpzc{E}$ respectively. The following are equivalent.
\begin{enumerate}
\item
$R$ sends objects in $\mathfrak{R}_{\mathpzc{E}}$ to objects in $\mathfrak{R}_{\mathpzc{D}}$.
\item
$L$ sends an exact sequence of the form
$$0\rightarrow X\rightarrow Y\rightarrow Z\rightarrow 0$$
with $Z\in\mathfrak{L}_{\mathpzc{D}}$ to an exact sequence
$$0\rightarrow L(X)\rightarrow L(Y)\rightarrow L(Z)\rightarrow 0$$
with $L(Z)\in\mathfrak{L}_{\mathpzc{E}}$. 
\item
$L$ sends objects in $\mathfrak{L}_{\mathpzc{D}}$ to objects in $\mathfrak{L}_{\mathpzc{E}}$
\item
$R$ sends an exact sequence of the form
$$0\rightarrow X\rightarrow Y\rightarrow Z\rightarrow 0$$
with $X\in\mathfrak{R}_{\mathpzc{E}}$ to an exact sequence
$$0\rightarrow R(X)\rightarrow R(Y)\rightarrow R(Z)\rightarrow 0$$
with $R(X)\in\mathfrak{R}_{\mathpzc{D}}$. 
\end{enumerate}
\end{cor}

\begin{proof}
$(1)\Rightarrow (2)$. Let 
$$0\rightarrow X\rightarrow Y\rightarrow Z\rightarrow 0$$
be an exact sequence in $\mathpzc{D}$ with $Z\in\mathfrak{L}_{\mathpzc{D}}$. Let $Q\in\mathfrak{R}_{\mathpzc{E}}$. Then the sequence
$$0\rightarrow\mathrm{Hom}(L(Z),Q)\rightarrow\mathrm{Hom}(L(Y),Q)\rightarrow\mathrm{Hom}(L(X),Q)\rightarrow 0$$
is isomorphic to the sequence
$$0\rightarrow\mathrm{Hom}(Z,R(Q))\rightarrow\mathrm{Hom}(Y,R(Q))\rightarrow\mathrm{Hom}(X,R(Q))\rightarrow 0$$
This is exact by Lemma \ref{lem:420}.

$(2)\Rightarrow (3)$ is clear, $(3)\Rightarrow(4)$ is dual to $(1)\Rightarrow (2)$, and $(4)\Rightarrow(1)$ is clear.
\end{proof}

In particular we get the following corollaries.

\begin{cor}
Let 
$$\adj{L}{\mathpzc{D}}{\mathpzc{E}}{R}$$
be an adjunction of exact categories.  Let $(\mathfrak{C}_{\mathpzc{D}},\mathfrak{W}_{\mathpzc{D}})$ and $(\mathfrak{C}_{\mathpzc{E}},\mathfrak{W}_{\mathpzc{E}})$ be Waldhausen pairs $\mathpzc{D}$ and $\mathpzc{E}$ respectively. Suppose that $L(\mathfrak{C}_{\mathpzc{D}})\subseteq\mathfrak{C}_{\mathpzc{E}}$, and $L(\mathfrak{C}_{\mathpzc{D}}\cap\mathfrak{W}_{\mathpzc{D}})\subseteq\mathfrak{W}_{\mathpzc{E}}$. Then the adjunction is Quillen.
\end{cor}

\begin{proof}
By Corollary \ref{cor:adjcotors} it suffices to observe that if
 $$0\rightarrow X\rightarrow Y\rightarrow C\rightarrow 0$$ is an exact sequence in $\mathpzc{D}$ with $C$ cofibrant, then 
$$0\rightarrow L(X)\rightarrow L(Y)\rightarrow L(Z)\rightarrow0$$
is an exact sequence in $\mathpzc{E}$. 
\end{proof}

\begin{cor}
Let 
$$\adj{L}{\mathpzc{D}}{\mathpzc{E}}{R}$$
be an adjunction of exact categories. Let $(\mathfrak{C}_{\mathpzc{D}},\mathfrak{W}_{\mathpzc{D}},\mathfrak{F}_{\mathpzc{D}})$ and $(\mathfrak{C}_{\mathpzc{E}},\mathfrak{W}_{\mathpzc{E}},\mathfrak{F}_{\mathpzc{E}})$ be Hovey triples on $\mathpzc{D}$ and $\mathpzc{E}$ respectively. Suppose that $R(\mathfrak{F}_{\mathpzc{E}})\subseteq\mathfrak{F}_{\mathpzc{D}}$ and $R(\mathfrak{F}_{\mathpzc{E}}\cap\mathfrak{W}_{\mathpzc{E}})\subseteq\mathfrak{F}_{\mathpzc{D}}\cap\mathfrak{W}_{\mathpzc{D}}$. Then the adjunction is Quillen.
\end{cor}

\subsection{Hovey Triples on Chain Complexes}

In \cite{kelly2016homotopy}, generalising results of \cite{Gillespie2} and \cite{gillespie2016derived}, we described a method for constructing compatible model structures on categories of chain complexes $ \mathrm{Ch}_{*}(\mathpzc{E})$ from cotorsion pairs on $\mathpzc{E}$. For complete and cocomplete abelian categories, in \cite{yang2014question} Ding and Yang showed that this method always produces a model structure. In private communication Timoth\'{e}e Moreau has shown that it holds for complete and cocomplete exact categories with exact products and exact coproducts, and in \cite{kelly2016homotopy} we showed it works for the projective cotorsion pair whenever $\mathpzc{E}$ satisfies the so-called axiom $AB4-k$ for some integer $k$, which essentially says that derived countable direct sums are sufficiently connective.

\begin{defn}\label{defn:dgK}
Let $(\mathfrak{L},\mathfrak{R})$ be a cotorsion pair on an exact category $\mathpzc{E}$. Let $X\in \mathrm{Ch}(\mathpzc{E})$ be a chain complex.
\begin{enumerate}
\item
$X$ is called an $\mathfrak{L}$ complex if it is acyclic and $Z_{n}X\in\mathfrak{L}$ for all $n$. The collection of all $\mathfrak{L}$ complexes is denoted $\widetilde{\mathfrak{L}}$.
\item
$X$ is called an $\mathfrak{R}$ complex if it is acyclic and $Z_{n}X\in\mathfrak{R}$ for all $n$. The collection of all $\mathfrak{R}$ complexes is denoted $\widetilde{\mathfrak{R}}$.
\item
$X$ is called a $K$-$\mathfrak{L}$ complex if $\mathrm{Hom}(X,B)$ is exact whenever $B$ is an $\mathfrak{R}$ complex. The class of all $K$-$\mathfrak{L}$ complexes is denoted $K\mathcal{L}$.
\item
$X$ is called a $K$-$\mathfrak{R}$ complex if $\mathrm{Hom}(A,X)$ is exact whenever $A$ is an $\mathfrak{L}$ complex. The class of all $K$-$\mathfrak{R}$ complexes is denoted $K\mathfrak{R}$.
\item
$X$ is called a $dg\mathfrak{L}$ complex if $X$ is a $K$-$\mathfrak{L}$ complex and $X_{n}\in\mathfrak{L}$ for each $n\in\mathbb{Z}$. The collection of all $dg\mathfrak{L}$ complexes is denoted $\widetilde{dg\mathfrak{L}}$.
\item
$X$ is called a $dg\mathfrak{R}$ complex if $X$ is a $K$-$\mathfrak{R}$ complex and $X_{n}\in\mathfrak{R}$ for each $n\in\mathbb{Z}$. The collection of all $dg\mathfrak{R}$ complexes is denoted $\widetilde{dg\mathfrak{R}}$.
\end{enumerate}
\end{defn}

We define the collections $\widetilde{\mathfrak{L}},\widetilde{\mathfrak{R}},K\mathfrak{L},K\mathfrak{R},\widetilde{dg\mathfrak{L}},\widetilde{dg\mathfrak{R}}$ similarly in the categories $ \mathrm{Ch}_{*}(\mathpzc{E})$ for  $*\in\{\ge0,\le0,+,-,b\}$. We will use the same notation for these collections irrespective of which category of chain complexes we are working in.

\subsubsection{Basic Properties}

Let us establish some basic properties of the classes defined in Definition \ref{defn:dgK}.

\begin{prop}\label{prop:correctprop}
 Le $(\mathfrak{L},\mathfrak{R})$ be a complete cotorsion pair on an exact category $\mathpzc{E}$ which has kernels. Let $Y$ be a complex with each $Y_{n}$ being in $\mathfrak{R}$. Then $Y$ is in $\tilde{\mathfrak{R}}$ if and only if 
 $$\mathrm{Hom}(L,Y)$$
 is acyclic for any $L\in\mathfrak{L}$.
 
 Dually let $X$ be a complex with each $X_{n}\in\mathfrak{L}$. Then $X\in\tilde{\mathfrak{L}}$ if and only if $\mathrm{Hom}(X,R)$ is acyclic for any $R\in\mathfrak{R}$.
\end{prop}

\begin{proof}
Since $\mathfrak{L}$ is a generating subcategory of $\mathpzc{E}$, if $\mathrm{Hom}(F,Y)$ is acyclic then $Y$ is acyclic. Now we have exact sequences $0\rightarrow Z_{n}Y\rightarrow Y_{n}\rightarrow Z_{n-1}Y\rightarrow0$, and therefore an exact sequence
$$0\rightarrow\mathrm{Hom}(F,Z_{n}Y)\rightarrow\mathrm{Hom}(F,Y_{n})\rightarrow\mathrm{Hom}(F,Z_{n-1}Y)\rightarrow\mathrm{Ext}^{1}(F,Z_{n}Y)\rightarrow \mathrm{Ext}^{1}(F,Y_{n})\cong 0$$
where in the isomorphism at the end we have used that $Y_{n}\in\mathfrak{R}$. However since $\mathrm{Hom}(F,Y)$ is acyclic we have that 
$$0\rightarrow\mathrm{Hom}(F,Z_{n}Y)\rightarrow\mathrm{Hom}(F,Y_{n})\rightarrow\mathrm{Hom}(F,Z_{n-1}Y)\rightarrow0$$
is short exact whence $\mathrm{Ext}^{1}(F,Z_{n}Y)\cong0$. 
\end{proof}

\begin{prop}[\cite{kelly2016homotopy} Lemma 4.2.23 (identical to \cite{Gillespie2} Lemma 3.4)]\label{prop:dgbound}
\;
\begin{enumerate}
\item
Bounded below complexes with entries in $\mathfrak{L}$ are $dg\mathcal{L}$-complexes.
\item
Bounded above complexes with entries in $\mathfrak{R}$ are $dg\mathfrak{R}$-complexes.
\end{enumerate}
\end{prop}

By \cite{kelly2016homotopy} Corollary 4.1.2 and Proposition 4.2.24 we have the following.

\begin{prop}\label{prop:dgltransf}
If $\mathfrak{L}$ is closed under transfinite extensions then $\widetilde{dg\mathfrak{L}}$ is closed under transfinite extensions. 
\end{prop}

\subsubsection{$dg$-Compatibility}

\begin{defn}\label{defn:dgcompat}
Let $\mathpzc{E}$ be a weakly idempotent complete exact category and $(\mathfrak{L},\mathfrak{R})$ a cotorsion pair on $\mathpzc{E}$. 
\begin{enumerate}
\item
We say that $(\mathfrak{L},\mathfrak{R})$ is $dg_{\ge0}$-compatible if $(\widetilde{dg\mathfrak{L}},\widetilde{\mathfrak{R}})$ is a functorially complete cotorsion pair on $\mathrm{Ch}_{\ge0}(\mathpzc{E})$, $\mathfrak{W}\cap \widetilde{dg\mathfrak{L}}=\widetilde{\mathfrak{L}}$ and the model structure whose cofibrations are $\mathbf{AdMon}_{\widetilde{dg\mathfrak{L}}}$, and whose acyclic cofibrations are $\mathbf{AdMon}_{\widetilde{\mathfrak{L}}}$ exists on $\mathrm{Ch}_{\ge0}(\mathpzc{E})$.
\item
We say that $(\mathfrak{L},\mathfrak{R})$ is $dg_{\le0}$-compatible if $(\widetilde{\mathfrak{L}},\widetilde{dg\mathfrak{R}})$ is a functorially complete cotorsion pair on $\mathrm{Ch}_{\le0}(\mathpzc{E})$, $\mathfrak{W}\cap \widetilde{dg\mathfrak{R}}=\widetilde{\mathfrak{R}}$ and the model structure whose fibrations are $\mathbf{AdEpi}_{\widetilde{dg\mathfrak{R}}}$, and whose acyclic fibrations are $\mathbf{AdEpi}_{\widetilde{\mathfrak{R}}}$ exists on $\mathrm{Ch}_{\le0}(\mathpzc{E})$.
\item
For $*\in\{b,+,-,\emptyset\}$ we say that $(\mathfrak{L},\mathfrak{R})$ is $dg_{*}$-compatible if $(\widetilde{\mathfrak{L}},\widetilde{dg\mathfrak{R}})$ and $(\widetilde{dg\mathfrak{L}},\widetilde{\mathfrak{R}})$ are (functorially) complete cotorsion pairs on $ \mathrm{Ch}_{*}(\mathpzc{E})$, $dg\mathfrak{L}\cap\mathfrak{W}=\widetilde{\mathfrak{L}}$, and $dg\mathfrak{R}\cap\mathfrak{W}=\widetilde{\mathfrak{R}}$
\end{enumerate}
\end{defn}

\begin{prop}[\cite{kelly2016homotopy} Corollary 4.4.92]\label{prop:simplicialmod}
Let $\mathpzc{E}$ be a complete and cocomplete exact category and $(\mathfrak{L},\mathfrak{R})$ a cotorsion pair on $\mathpzc{E}$. If $(\mathfrak{L},\mathfrak{R})$ is $dg_{\ge0}$-compatible (resp. $dg$-compatible) then the model structure on $\mathrm{Ch}_{\ge0}(\mathpzc{E})$ (resp. on $\mathrm{Ch}(\mathpzc{E}))$ is a Kan-complex enriched simplicial model category. 
\end{prop}

\begin{lem}\label{lem:dgdeconstruct}
If $(\mathfrak{L},\mathfrak{R})$ is a complete cotorsion pair on $\mathpzc{E}$ with $\mathpzc{E}$ being weakly $\mathbf{AdMon}_{\mathfrak{L}}$-elementary, and $\mathfrak{L}$ is presentably deconstructible in itself then $\widetilde{dg\mathfrak{L}}$ is presentably deconstructible in itself.
\end{lem}

\begin{proof}
There exists a set $\mathbf{J}_{\mathfrak{L}}$ of monomorphisms with small domains and codomains such that any admissible monomorphism $X\rightarrow Y$ with cokernel in $\mathfrak{L}$ is a transfinite composition of elements of $\mathbf{J}_{\mathfrak{L}}$. Let $\mathcal{S}=\mathrm{Coker}(\mathbf{J}_{\mathfrak{L}})$. Without loss of generality we may assume that $\mathcal{S}$ contains a generator. Note that $\mathrm{Filt}(\{S^{n}(A):n\in\mathbb{Z},A\in\mathcal{S}\})$ is a subcategory of $\widetilde{dg\mathfrak{L}}$ which generates $\mathpzc{E}$. Thus $(\mathrm{Filt}(\{S^{n}(A):n\in\mathbb{Z},A\in\mathcal{S}\}),\{S^{n}(A):n\in\mathbb{Z},A\in\mathcal{S}\}^{\perp})$ is a complete cotorsion pair on $\mathrm{Ch}(\mathpzc{E})$. We claim that $\{S^{n}(A):n\in\mathbb{Z},A\in\mathcal{S}\}^{\perp}=\tilde{\mathfrak{R}}$. The maps $S^{n}(G)\rightarrow D^{n+1}(G)$ are in $\mathbf{AdMon}_{\mathrm{Filt}(\{S^{n}(A):n\in\mathbb{Z},A\in\mathcal{S}\})}$, so by \cite{kelly2016homotopy} Corollary 2.6.110, any complex $X$ such that $X\rightarrow 0$ has the right lifting property with respect to such maps is acyclic. Also since $X$ is acyclic, we have $0\cong\mathrm{Ext}^{1}(S^{n}(A),X)\cong\mathrm{Ext}^{1}(A,Z_{n}X)$. Thus each $Z_{n}X\in\mathfrak{R}$, and we are done.
\end{proof}

\begin{cor}\label{cor:modelexistencedec}
Let $(\mathfrak{L},\mathfrak{R})$ be a complete cotorsion pair on a weakly idempotent complete exact category $\mathpzc{E}$ with $\mathfrak{L}$ presentably deconstructible in itself, and such that $\mathpzc{E}$ is weakly $\mathbf{AdMon}_{\mathfrak{L}}$-elementary. Then
\begin{enumerate}
\item
$(\mathfrak{L},\mathfrak{R})$ is $dg_{\ge0}$-compatible if and only if $\tilde{\mathfrak{L}}$ is deconstructible in itself in $\mathrm{Ch}_{\ge0}(\mathpzc{E})$.
\item
$(\mathfrak{L},\mathfrak{R})$ is $dg$-compatible if and only if $\mathfrak{W}\cap\widetilde{dg\mathfrak{L}}$ is deconstructible in itself in $\mathrm{Ch}(\mathpzc{E})$ and $\mathfrak{W}\cap\widetilde{dg\mathfrak{L}}=\tilde{\mathfrak{L}}$. 
\item
 if $\mathpzc{E}$ is finitely cocomplete $\mathfrak{W}\cap\widetilde{dg\mathfrak{L}}$ is deconstructible in itself in $\mathrm{Ch}(\mathpzc{E})$ and $\mathfrak{W}\cap\widetilde{dg\mathfrak{L}}=\tilde{\mathfrak{L}}$ then $(\mathfrak{L},\mathfrak{R})$ is  $dg_{\le0}$-compatible.
\end{enumerate}
\end{cor}

\begin{proof}
\begin{enumerate}
\item
Since the class of quasi-isomorphisms satisfies the $2$-out-of-$3$ property and is closed under retracts, the first claim follows immediately from Theorem \ref{thm:amendingmodel}, Lemma \ref{lem:dgdeconstruct}, and Proposition \ref{prop:dgltransf}.
\item
The second claim  follows similarly to the first.
\item
The final claim is a straightforward application of the second claim and the transfer theorem, specifically the statement Theorem A.5.29 in \cite{kelly2016homotopy}, where we transfer along the adjunction
$$\adj{\tau_{\le0}}{\mathrm{Ch}(\mathpzc{E})}{\mathrm{Ch}_{\le0}(\mathpzc{E})}{i}$$
\end{enumerate}
\end{proof}


If more generally $\mathfrak{L}$ is cogenerated by a set then so is $\widetilde{dg\mathfrak{L}}$.

\begin{prop}[\cite{kelly2016homotopy} Proposition 4.2.49, following \cite{gillespie2016derived} Section 4.4]\label{prop:cofibrantgenerators}
Let $(\mathfrak{L},\mathfrak{R})$ be a cotorsion pair in an exact category $\mathpzc{E}$ which has a set of admissible generators $\mathcal{G}$. Suppose that $(\mathfrak{L},\mathfrak{R})$ is cogenerated by a set $\{A_{i}\}_{i\in I}$. Then  
\begin{enumerate}
\item
$(\widetilde{dg\mathfrak{L}},\widetilde{\mathfrak{R}})$ is cogenerated by the set 
$$S=\{S^{n}(G):G\in\mathcal{G},n\in\mathbb{Z}\}\cup\{S^{n}(A_{i}):n\in\mathbb{Z},i\in I\}$$
for $*\in\{+\}$ (and. $*\in\{\emptyset\}$ if $\mathpzc{E}$ has kernels) and 
$$S=\{S^{n}(G):G\in\mathcal{G},n\ge0\}\cup\{S^{n}(A_{i}):n\ge0,i\in I\}$$
for $*\in\{\ge0\}$.

\item
Suppose that $(\mathfrak{L},\mathfrak{R})$ is small with generating morphisms the maps $\{0\rightarrow G:G\in\mathcal{G}\}$ together with monics $k_{i}$  (one for each $i\in I$).
\begin{displaymath}
\xymatrix{
0\ar[r] & Y_{i}\ar[r]^{k_{i}} & Z_{i}\ar[r] & A_{i}\ar[r] & 0
}
\end{displaymath}
Then $(\widetilde{dg\mathfrak{L}},\widetilde{\mathfrak{R}})$ is small with generating morphisms the set
$$\widetilde{I}=\{0\rightarrow D^{n}(G)\}\cup\{S^{n-1}(G)\rightarrow D^{n}(G)\}\cup\{S^{n}(k_{i}):S^{n}(Y_{i})\rightarrow S^{n}(Z_{i})\}$$
for $*\in\{+\}$ (and. $*\in\{\emptyset\}$ if $\mathpzc{E}$ has kernels)
and
$$\widetilde{I}=\{0\rightarrow S^{0}(G)\}\cup\{0\rightarrow D^{n}(G):n>0\}\cup\{S^{n-1}(G)\rightarrow D^{n}(G):n>0\}$$
$$\;\;\;\;\;\;\;\;\;\;\;\;\;\;\;\;\;\;\;\;\;\;\;\;\;\;\;\;\;\;\;\;\;\;\;\;\;\;\cup\{S^{n}(k_{i}):S^{n}(Y_{i})\rightarrow S^{n}(Z_{i}):n\ge0\}$$
for $*\in\{\ge0\}$.

\end{enumerate}
\end{prop}

%
%
\subsubsection{Monoidal Model Structures on Chain Complexes}

Next we examine monoidal properties of model structures determined by $dg_{*}$-compatible cotorsion pairs.

\begin{defn}[\cite{kelly2016homotopy} Definition 4.2.53]\label{tensormodel}
Let $(\mathpzc{E},\otimes)$ be a monoidal exact category. A cotorsion pair $(\mathfrak{L},\mathfrak{R})$ on $\mathpzc{E}$ is said to be \textit{monoidally }$dg_{*}$-\textit{compatible} for $*\in\{\ge0,+,\emptyset\}$ if
\begin{enumerate}
\item
 $(\mathfrak{L},\mathfrak{R})$ is $dg_{*}$ compatible.
\item
$\mathfrak{L}$ contains $k$ and is closed under $\otimes$.
\item
short exact sequences in $\mathfrak{L}$ are $\mathfrak{L}$-pure.
\end{enumerate}
Now let $*\in\{\ge0,\emptyset\}$. If in addition objects of $\mathfrak{L}$ are flat, $\mathpzc{E}$ is weakly $\mathbf{PureMon}_{\otimes}$-elementary and every (trivially) cofibrant complex is an $(\aleph_{0};\mathbf{PureMon}_{\otimes})$-extension of bounded below (trivially) cofibrant complexes, then the cotorsion pair is said to be \textit{ strongly monoidally }$dg_{*}$-\textit{compatible}.
\end{defn}

\begin{defn}
Let $(\mathfrak{L},\mathfrak{R})$ and $(\mathfrak{A},\mathfrak{B})$ be complete cotorsion pairs. $[(\mathfrak{L},\mathfrak{R}),(\mathfrak{A},\mathfrak{B})]$ are said to be \textit{monoidally compatible} if
\begin{enumerate}
\item
Whenever $A\in\mathfrak{A}$ and $L_{\bullet}\in\tilde{\mathfrak{L}}$, $A\otimes L_{\bullet}\in\tilde{\mathfrak{L}}$.
\item
Whenever $A_{\bullet}\in\tilde{\mathfrak{A}}$ and $L\in\mathfrak{L}$, $A_{\bullet}\otimes L\in\tilde{\mathfrak{L}}$.
\end{enumerate}
\end{defn}

\begin{rem}
The ordering in the above definition is important - this is why we use the notation $[(\mathfrak{L},\mathfrak{R}),(\mathfrak{A},\mathfrak{B})]$.
\end{rem}

\begin{example}
Let $(\mathfrak{L},\mathfrak{R})$ be monoidally $dg$-compatible. Then $[(\mathfrak{L},\mathfrak{R}),(\mathfrak{L},\mathfrak{R})]$ are monoidally compatible. 
\end{example}

\begin{example}
Let $(\mathpzc{E},\mathrm{Inj})$ be the injective cotorsion pair, and $(\mathrm{Flat},\mathrm{Flat}^{\perp})$ be the flat cotorsion pair. Then $[(\mathpzc{E},\mathrm{Inj}),(\mathrm{Flat},\mathrm{Flat}^{\perp})]$ are monoidally $dg$-compatible. In fact this is true if  $(\mathrm{Flat},\mathrm{Flat}^{\perp})$ is replaced by any cotorsion pair $(\mathfrak{A},\mathfrak{B})$ with $\mathfrak{A}\subseteq\mathrm{Flat}$. 
\end{example}

\begin{prop}\label{prop:basicompat}
If $[(\mathfrak{L},\mathfrak{R}),(\mathfrak{A},\mathfrak{B})]$ monoidally compatible then
\begin{enumerate}
\item
for all $A\in\mathfrak{A}$, $L\in\mathfrak{L}$, we have $A\otimes L\in\mathfrak{L}$.
\item
for all $A\in\mathfrak{A}$, $R\in\mathfrak{R}$, we have $\underline{\mathrm{Hom}}(A,R)\in\mathfrak{R}$.
\item
for all $L\in\mathfrak{L}$, $R\in\mathfrak{R}$ we have $\underline{\mathrm{Hom}}(L,R)\in\mathfrak{B}$
\end{enumerate}
\end{prop}

\begin{proof}
\begin{enumerate}
\item
Let $L\in\mathfrak{L}$. Then $D^{0}(L)\in\tilde{\mathfrak{L}}$. Thus $A\otimes D^{0}(L)\cong D^{0}(A\otimes L)\in\tilde{\mathfrak{L}}$, so $A\otimes L\in\mathfrak{L}$.
\item
Let $L_{\bullet}\in\tilde{\mathfrak{L}}$. We have
$$\mathbf{Hom}(L_{\bullet},\underline{\mathrm{Hom}}(A,R))\cong\mathbf{Hom}(L_{\bullet}\otimes A,R)$$
By assumption $L_{\bullet}\otimes A\in\tilde{\mathfrak{L}}$, so $\mathbf{Hom}(L_{\bullet}\otimes A,R)$ is acyclic. Since $\underline{\mathrm{Hom}}(A,R)$ is concentrated in degree $0$, in particular it is bounded, it is in $\mathfrak{R}$ by Proposition \ref{prop:dgbound}.
\item
This is entirely similar to the previous part.
\end{enumerate}
\end{proof}

\begin{prop}\label{prop:termwisemonoidg}
Let $\mathpzc{E}$ be weakly idempotent complete. If $[(\mathfrak{L},\mathfrak{R}),(\mathfrak{A},\mathfrak{B})]$ are monoidally $dg$-compatible, and there are enough $\mathfrak{L},\mathfrak{R},\mathfrak{A}$, and $\mathfrak{B}$ objects, then 
\begin{enumerate}
\item
If $A\in\mathfrak{A}$ and $R\in\widetilde{dg\mathfrak{R}}$ then $\underline{\mathrm{Hom}}(A,R)\in\widetilde{dg\mathfrak{R}}$
\item
If $L\in\mathfrak{L}$ and $R\in\widetilde{dg\mathfrak{R}}$ then $\underline{\mathrm{Hom}}(A,R)\in\widetilde{dg\mathfrak{B}}$
\item
If $A\in\mathfrak{A}$ and $R\in\tilde{\mathfrak{R}}$ then $\underline{\mathrm{Hom}}(A,R)\in\tilde{\mathfrak{R}}$.
\item
If $L\in\mathfrak{L}$ and $R\in\tilde{\mathfrak{R}}$ then  $\underline{\mathrm{Hom}}(L,R)\in\tilde{\mathfrak{B}}$.
\item
If $A_{\bullet}\in\widetilde{dg\mathfrak{A}}$ and  $L\in\mathfrak{L}$  then $A_{\bullet}\otimes L\in \widetilde{dg\mathfrak{L}}$.
\item
If $A\in\mathfrak{A}$ and  $L_{\bullet}\in\widetilde{dg\mathfrak{L}}$  then $A\otimes L_{\bullet}\in \widetilde{dg\mathfrak{L}}$.
\end{enumerate}
\end{prop}

\begin{proof}
\begin{enumerate}
\item
Let $L_{\bullet}\in\tilde{\mathfrak{L}}$. Then
$$\mathbf{Hom}(L_{\bullet},\underline{\mathrm{Hom}}(A,R))\cong\mathbf{Hom}(L_{\bullet}\otimes A,R)$$
By assumption $L_{\bullet}\otimes A\in\tilde{\mathfrak{L}}$. Thus $\mathbf{Hom}(L_{\bullet}\otimes A,R)$ is acyclic, so $\mathbf{Hom}(L_{\bullet},\underline{\mathrm{Hom}}(A,R))$ is acyclic. Moreover it is component-wise in $\mathfrak{R}$ by Proposition \ref{prop:basicompat}.
\item
This is entirely similar to the previous part. 
\item
Let us prove that it is acyclic. To establish this, it suffices to observe that for each $L\in\mathfrak{L}$, 
$$\mathbf{Hom}(L,\underline{\mathrm{Hom}}(A,R_{\bullet}))\cong\mathbf{Hom}(L\otimes A,R_{\bullet})$$
is acyclic, since $L\otimes A\in\mathfrak{L}$ again by Proposition \ref{prop:basicompat}, the right-hand side is acyclic. Now we have $\underline{\mathrm{Hom}}(A,Z_{n}R_{\bullet})\cong Z_{n}\underline{\mathrm{Hom}}(A,R_{\bullet})$. Since $Z_{n}R_{\bullet}\in\mathfrak{R}$, we have $\underline{\mathrm{Hom}}(A,Z_{n}R_{\bullet})\in\mathfrak{R}$, so $Z_{n}\underline{\mathrm{Hom}}(A,R_{\bullet})$. Thus $\underline{\mathrm{Hom}}(A,R_{\bullet})\in\tilde{\mathfrak{R}}$. 
\item
This is entirely similar to the previous part.
\item
$A_{\bullet}\otimes L$ is termwise in $\mathfrak{L}$. Let $R_{\bullet}$ be an object of $\tilde{\mathfrak{R}}$. We then have
$$\mathbf{Map}(A_{\bullet}\otimes L,R_{\bullet})\cong\mathbf{Map}(A_{\bullet},\underline{\mathrm{Hom}}(L,R_{\bullet}))$$
By part $3)$, $\underline{\mathrm{Hom}}(L,R_{\bullet})\in\tilde{\mathfrak{B}}$. Thus $\mathbf{Map}(A_{\bullet}\otimes L,R_{\bullet})$ is acyclic, as required.
\item
This is entirely similar to the previous part.
\end{enumerate}
\end{proof}
%

\begin{prop}\label{prop:relations}
Let $\mathpzc{E}$ be weakly idempotent complete. If $[(\mathfrak{L},\mathfrak{R}),(\mathfrak{A},\mathfrak{B})]$ are monoidally $dg$-compatible, and there are enough $\mathfrak{L},\mathfrak{R},\mathfrak{A}$, and $\mathfrak{B}$ objects, then 
\begin{enumerate}
\item
$\underline{\mathrm{Hom}}(L_{\bullet},R_{\bullet})\in\tilde{\mathfrak{B}}$ for all $L_{\bullet}\in\widetilde{dg\mathfrak{L}}$, and all $R_{\bullet}\in\tilde{\mathfrak{R}}$.
\item
$\underline{\mathrm{Hom}}(L_{\bullet},R_{\bullet})\in\tilde{\mathfrak{B}}$ for all $L_{\bullet}\in\tilde{\mathfrak{L}}$, and all $R_{\bullet}\in\widetilde{dg\mathfrak{R}}$.
\item
$\underline{\mathrm{Hom}}(A_{\bullet},R_{\bullet})\in\tilde{\mathfrak{R}}$ for all $A_{\bullet}\in\widetilde{dg\mathfrak{L}}$, and all $R_{\bullet}\in\tilde{\mathfrak{R}}$.
\item
$\underline{\mathrm{Hom}}(A_{\bullet},R_{\bullet})\in\tilde{\mathfrak{R}}$ for all $A_{\bullet}\in\tilde{\mathfrak{A}}$, and all $R_{\bullet}\in\widetilde{dg\mathfrak{R}}$.
\item
If $A_{\bullet}\in\widetilde{dg\mathfrak{A}}$ and $L_{\bullet}\in\widetilde{dg\mathfrak{L}}$ then $A_{\bullet}\otimes L_{\bullet}\in\widetilde{dg\mathfrak{L}}$.
\item
If $A_{\bullet}\in\tilde{\mathfrak{A}}$ and $L_{\bullet}\in\widetilde{dg\mathfrak{L}}$ then $A_{\bullet}\otimes L_{\bullet}\in\tilde{\mathfrak{L}}$. 
\item
If $A_{\bullet}\in\widetilde{dg\mathfrak{A}}$ and $L_{\bullet}\in\tilde{\mathfrak{L}}$ then $A_{\bullet}\otimes L_{\bullet}\in\tilde{\mathfrak{L}}$. 
\item
If $A_{\bullet}\in\widetilde{dg\mathfrak{A}}$ and $R_{\bullet}\in\widetilde{dg\mathfrak{R}}$ then $\underline{\mathrm{Hom}}(A_{\bullet},R_{\bullet})\in\widetilde{dg\mathfrak{R}}$.
\item
If $L_{\bullet}\in\widetilde{dg\mathfrak{L}}$ and $R_{\bullet}\in\widetilde{dg\mathfrak{R}}$ then $\underline{\mathrm{Hom}}(L_{\bullet},R_{\bullet})\in\widetilde{dg\mathfrak{B}}$.
\end{enumerate}
\end{prop}

\begin{proof}
\begin{enumerate}
\item
We have that each term of $\underline{\mathrm{Hom}}(L_{\bullet},R_{\bullet})$ is in $\mathfrak{R}$. Let $A\in\mathfrak{A}$. Then
$$\mathbf{Hom}(A,\underline{\mathrm{Hom}}(L_{\bullet},R_{\bullet}))\cong\mathbf{Hom}(A\otimes L_{\bullet},R_{\bullet})$$
By Proposition \ref{prop:termwisemonoidg} $A\otimes L_{\bullet}$ is in $\widetilde{dg\mathfrak{L}}$. Now we just use Proposition \ref{prop:correctprop}
\item
$\underline{\mathrm{Hom}}(L_{\bullet},R_{\bullet})$ is term-wise in $\mathfrak{B}$ by Proposition \ref{prop:termwisemonoidg}. Let $A\in\mathfrak{A}$. Then $A\otimes L_{\bullet}\in\tilde{\mathfrak{L}}$ by assumption. Thus $\mathbf{Hom}(A\otimes L_{\bullet},R_{\bullet})$ is acyclic. By Proposition \ref{prop:correctprop} we have that $\underline{\mathrm{Hom}}(L_{\bullet},R_{\bullet})$ is in $\tilde{\mathfrak{R}}$.
\item
This is entirely similart to Part (1).
\item
This is entirely similar to Part (2). 
\item
Let $A_{\bullet}\in\widetilde{dg\mathfrak{A}}$ and $L_{\bullet}\in\widetilde{dg\mathfrak{L}}$. Each term of $A_{\bullet}\otimes L_{\bullet}$ is in $\mathfrak{L}$. Let $R_{\bullet}$ be in $\tilde{\mathfrak{R}}$. Then
$$\mathbf{Hom}(A_{\bullet}\otimes L_{\bullet},R_{\bullet})\cong\mathbf{Hom}(A_{\bullet},\underline{\mathrm{Hom}}(L_{\bullet},R_{\bullet}))$$
Since $\underline{\mathrm{Hom}}(L_{\bullet},R_{\bullet})\in\tilde{\mathfrak{B}}$ is acyclic, the right-hand-side, and hence the left-hand-side is acyclic. 
\item
Let $R\in\mathfrak{R}$.
$$\mathbf{Hom}(A_{\bullet}\otimes L_{\bullet},R)\cong\mathbf{Hom}(L_{\bullet},\underline{\mathrm{Hom}}(A_{\bullet},R))$$
Now $\underline{\mathrm{Hom}}(A_{\bullet},R)\in\tilde{\mathfrak{R}}$, so $\mathbf{Hom}(L_{\bullet},\underline{\mathrm{Hom}}(A_{\bullet},R))$ is acyclic. Now $A_{\bullet}\otimes L_{\bullet}$ is term-wise in $L$, so the claim follows from Proposition \ref{prop:correctprop}.
\item
This is entirely similar to the previous part.
\item
Each term of $\underline{\mathrm{Hom}}(A_{\bullet},R_{\bullet})$ is in $\mathfrak{R}$. Let $L\in\tilde{\mathfrak{L}}$. Then
$$\mathbf{Hom}(L_{\bullet},\underline{\mathrm{Hom}}(A_{\bullet},R_{\bullet}))\cong\mathbf{Hom}(L_{\bullet}\otimes A_{\bullet},R_{\bullet})$$
Since $L_{\bullet}\otimes A_{\bullet}\in\tilde{\mathfrak{L}}$ this complex is acyclic as required.
\item
This is entirely similar to the previous part.
\end{enumerate}
\end{proof}

Let us briefly specialise to the monoidally compatible cotorsion pairs $[(\mathpzc{E},\mathrm{Inj}),(\mathrm{Flat},\mathrm{Flat}^{\perp})]$, and show how the work of \cite{gillespie2006flat} Section 5 generalises immediately from the abelian category of modules on a ringed space, to much more general exact categories.

\begin{prop}[c.f. \cite{gillespie2006flat} Lemma 5.3]
Let $\mathpzc{E}$ be a complete and cocomplete closed symmetric monoidal exact category. Suppose that $\mathpzc{E}$ has enough injectives. An object $X$ of $\mathpzc{E}$ is flat if and only if $\underline{\mathrm{Hom}}(X,I)$ is injective for any injective $I$.
\end{prop}

\begin{proof}
Suppose that $X$ is flat. Let $Y_{\bullet}$ be an acylic complex. Then $\mathbf{Hom}(Y_{\bullet},\underline{\mathrm{Hom}}(X,I))\cong\mathbf{Hom}(Y_{\bullet}\otimes X,I)$ is acyclic. This implies that $\underline{\mathrm{Hom}}(X,I)$ is flat. On the other hand, suppose that $\underline{\mathrm{Hom}}(X,I)$ is injective for any injective $I$. Then $\mathbf{Hom}(Y_{\bullet}\otimes X,I)\cong\mathbf{Hom}(Y_{\bullet},\underline{\mathrm{Hom}}(X,I))$ is acyclic. Since there are enough injectives, this implies that $Y_{\bullet}\otimes X$ is acyclic, so $X$ is flat. 
\end{proof}

\begin{lem}[c.f. \cite{gillespie2006flat} Lemma 5.4, Lemma 5.5, Proposition 5.6]
Let $X\in\mathrm{Ch}(\mathpzc{E})$. 
\begin{enumerate}
\item
The following are equivalent.
\begin{enumerate}
\item
$X$ is acyclic
\item
$\underline{\mathrm{Hom}}(X,R)$ is in $\widetilde{\mathrm{Flat}}^{\perp}$ for all $R\in\widetilde{dg\mathrm{Inj}}$.
\item
$\underline{\mathrm{Hom}}(X,R)\in\widetilde{\mathrm{Flat}}^{\perp}$ for all $R\in\mathrm{Inj}$.
\end{enumerate}
\item
The following are equivalent. 
\begin{enumerate}
\item
$X\in\widetilde{\mathrm{Flat}}$
\item
$\underline{\mathrm{Hom}}(X,R)\in\widetilde{\mathrm{Inj}}$ for all $R\in\widetilde{dg\mathrm{Inj}}$.
\item
$\underline{\mathrm{Hom}}(X,R)\in\widetilde{\mathrm{Inj}}$ for all $R\in\mathrm{Inj}$.
\end{enumerate}
\item
The following are equivalent.
\begin{enumerate}
\item
$X\in\widetilde{dg\mathrm{Flat}}$.
\item
each $X_{n}\in\mathrm{Flat}$ and $X\otimes L$ is acyclic for any acyclic $L$.
\end{enumerate}
\end{enumerate}
\end{lem}

\begin{proof}
\begin{enumerate}
\item
That $1)$ implies $2)$ follows from Proposition \ref{prop:relations}. $2)$ clearly implies $3)$. Now suppose $3)$ holds. Since the tensor unit is flat, we have
$$\mathbf{Hom}(X,R)\cong\mathbf{Hom}(k,\underline{\mathrm{Hom}}(X,R))$$
is acyclic for any $R\in\mathrm{Inj}$. Since there are enough injectives, this implies that $X$ is acyclic.
\item
Again that $1)$ implies $2)$ follows from Proposition \ref{prop:relations}, and $2)$ clearly implies $3)$. Now suppose $3)$ holds. We have
$$\mathbf{Hom}(X,R)\cong\mathbf{Hom}(k,\underline{\mathrm{Hom}}(X,R))$$
is acyclic. Again since there are enough injecitves, this implies that $X$ is acyclic. Finally, we have $\underline{\mathrm{Hom}}(Z_{n}X,R)\cong Z_{n}\underline{\mathrm{Hom}}(X,R)\in\mathrm{Inj}$ for all $R$, which implies that $Z_{n}X$ is flat.
\item
The proof of this part follows \cite{gillespie2006flat} Proposition 5.6 essentially identically.
$1)\Rightarrow 2)$ follows yet again from Proposition \ref{prop:relations}. For the converse, consider an exact sequence 
$$0\rightarrow X\rightarrow C\rightarrow F'\rightarrow 0$$
with $C\in\tilde{\mathrm{Flat}^{\perp}}$ and $F'\in\widetilde{dg\mathrm{Flat}}$. We claim that $C\in\widetilde{\mathrm{Flat}}$. It suffices to prove that for any $R\in\mathrm{Inj}$, we have $\underline{\mathrm{Hom}}(C,R)\in\widetilde{\mathrm{Inj}}$. . As an extension of flat objects, each $C_{n}$ is also flat. For any complex $E$, the sequence
$$0\rightarrow X\otimes E\rightarrow C\otimes E\rightarrow F'\otimes E\rightarrow 0$$
is short exact. Moreover by assumption $X\otimes E$ is acyclic. By $1)\Rightarrow 2)$ we have that $F'\otimes E$ is acyclic. This in turn implies that $C\otimes E$ is acyclic. Now For any acyclic complex $E$ and any injective $R$ we have
$$\mathbf{Hom}(E,\underline{\mathrm{Hom}}(C,R))\cong\underline{\mathrm{Hom}}(E\otimes C,R)$$
is acyclic. Thus $\underline{\mathrm{Hom}}(C,R)$ is in $\widetilde{dg\mathrm{Inj}}$. Since $C$ is acyclic so is $\underline{\mathrm{Hom}}(C,R)$ so it must in fact be in $\widetilde{\mathrm{Inj}}$, as required.
\end{enumerate}
\end{proof}

%

\begin{prop}[\cite{kelly2016homotopy} Proposition 4.2.55]\label{monoidaldgcompat}
Let $(\mathfrak{L},\mathfrak{R})$ be a hereditary monoidally $dg_{*}$-compatible cotorsion pair for $*\in\{\ge,\emptyset\}$ on weakly idempotent complete $\mathpzc{E}$. The model category structure induced by $(\mathfrak{L},\mathfrak{R})$ on $\mathrm{Ch}_{*}(\mathpzc{E})$ is monoidal.
\end{prop}

\begin{cor}[\cite{kelly2016homotopy} Proposition 4.2.56]\label{prop:tensacyclic}
Let $\mathfrak{L}$ be a class of objects in a monoidal weakly idempotent complete exact category $\mathpzc{E}$, and suppose that $\mathpzc{E}$ is weakly $\mathbf{PureMon}_{\otimes}$-elementary
\begin{enumerate}
\item
Suppose that any admissible monomorphism with cokernel in $\mathfrak{L}$ is pure. If $X$ is any complex then for any complex $L$ in $\widetilde{\mathfrak{L}}$, $L\otimes X$ is acyclic.
\item
Suppose that objects in $\mathfrak{L}$ are flat. If $L$ is an $(\aleph_{0};\mathbf{PureMon}_{\otimes})$-extension of bounded below complexes of objects in $\mathfrak{L}$ then for any acyclic complex $X$, $X\otimes L$ is acyclic.
\end{enumerate}
In particular if there is a  strongly monoidally $dg_{*}$-compatible cotorsion pair $(\mathfrak{L},\mathfrak{R})$ for $*\in\{\ge0,\emptyset\}$, then the induced model structure satisfies the monoid axiom. Moreover in this case if $C$ is cofibrant and $X$ is acyclic then $C\otimes X$ is acyclic. 
\end{cor}

It is often useful to consider a weaker version.

\begin{defn}
Let $(\mathfrak{L},\mathfrak{R})$ and $(\mathfrak{A},\mathfrak{B})$ be complete cotorsion pairs. $[(\mathfrak{L},\mathfrak{R}),(\mathfrak{A},\mathfrak{B})]$ are said to be 
\begin{enumerate}
\item
\textit{weakly monoidally compatible} if whenever $A\in\mathfrak{A}$ and $L_{\bullet}\in\tilde{\mathfrak{L}}$, $A\otimes L_{\bullet}$ is acyclic, and whenever $A_{\bullet}\in\tilde{\mathfrak{A}}$ $L\in\mathfrak{L}$. $A_{\bullet}\otimes L$ is acyclic. 
\item
\textit{closed weakly monoidally compatible} if it is weakly monoidally compatible, and whenever $R\in\mathfrak{R}$ and $L_{\bullet}\in\tilde{\mathfrak{L}}$, $\underline{\mathrm{Hom}}(L_{\bullet},R)$ is acyclic, and whenever $R_{\bullet}\in\tilde{\mathfrak{R}}$ and $L\in\mathfrak{L}$, $\underline{\mathrm{Hom}}(L,R_{\bullet})$ is acyclic 
\end{enumerate}
\end{defn}

\begin{prop}
Suppose that any $A\in\widetilde{dg\mathfrak{A}}$ may be written as a transfinite extension of objects of the form $S^{n}(A)$ with $A\in\mathfrak{A}$, and any $L\in\widetilde{dg\mathfrak{L}}$ may be written as a transfinite extensions of objects of the form $S^{n}(L)$ with $L\in\mathfrak{L}$. Let $\mathfrak{A}\otimes\mathfrak{L}$ denote the class of iterated tensor products of objects of $\mathfrak{A}$ and objects of $\mathfrak{L}$. If $\mathpzc{E}$ is weakly $\mathbf{AdMon}_{\mathfrak{A}\otimes\mathfrak{L}}$-elementary, and $[(\mathfrak{L},\mathfrak{R}),(\mathfrak{A},\mathfrak{B})]$ are weakly monoidally compatible then $A_{\bullet}\otimes L_{\bullet}$ is acyclic for any $A_{\bullet}\in\widetilde{dg\mathfrak{A}}$ and any $L_{\bullet}\in\tilde{\mathfrak{L}}$, and for any $A_{\bullet}\in\tilde{\mathfrak{A}}$ and $L_{\bullet}\in\widetilde{dg\mathfrak{L}}$. 
\end{prop}

\begin{proof}
We prove the case with $A\in\widetilde{dg\mathfrak{A}}$. $A_{\bullet}\otimes L_{\bullet}$ may be written as a transfinite extension of objects of the form $S^{n}(A_{\alpha})\otimes L_{\bullet}$ with $A_{\alpha}\in\mathfrak{A}$. $S^{n}(A_{\alpha})\otimes L_{\bullet}$ is acyclic by assumption. Now the claim follows since $\mathpzc{E}$ is assumed to be weakly $\mathbf{AdMon}_{\mathfrak{A}\otimes\mathfrak{L}}$-elementary.
\end{proof}

\subsubsection{A Monoidal Structure on the Left Heart}
In \cite{qacs} Section 1.5.2, Schneiders showed that given a monoidal elementary quasi-abelian category $\mathpzc{E}$, there is induced closed symmetric monoidal structure on its left heart $\mathrm{LH}(\mathpzc{E})$ such that the functor $\mathrm{LH}(\mathpzc{E})\rightarrow\mathpzc{E}$ is strong monoidal. Here we generalise this to exact categories.

Let $\mathpzc{E}$ be a closed symmetric monoidal exact category which, for simplicity, we assume to be finitely complete and cocomplete. Suppose that $(\mathfrak{L},\mathfrak{R})$ is a complete, monoidally compatible, cotorsion pair on $\mathpzc{E}$ with $k\in\mathfrak{L}$, so that the tensor product functor $\otimes$ is left derivable, and the internal hom functor $\underline{\mathrm{Hom}}(-,-)$ is right derivable. 

\begin{cor}[c.f. \cite{qacs} Proposition 1.5.3]
We have functorial isomorphisms
\begin{enumerate}
\item
$X\otimes^{\mathbb{L}}Y\cong Y\otimes^{\mathbb{L}}X$
\item
$X\otimes^{\mathbb{L}}k\cong X$
\item
$\mathbb{R}\mathrm{Hom}(X\otimes^{\mathbb{L}}Y,Z)\cong\mathbb{R}\mathrm{Hom}(X,\mathbb{R}\underline{\mathrm{Hom}}(Y,Z))$
\item
$\mathbb{R}\underline{\mathrm{Hom}}(k,Z)\cong Z$
\end{enumerate}
for any $X,Y\in\mathrm{D}_{+}(\mathpzc{E})$ and any $Z\in\mathrm{D}_{-}(\mathpzc{E})$.
\end{cor}

Denote by $\otimes_{LH}$ the functor given by the restriction of 
$$\mathrm{LH}_{0}\circ\otimes^{\mathbb{L}}:\mathbf{Ch}_{+}(\mathpzc{E})\times\mathbf{Ch}_{+}(\mathpzc{E})\rightarrow\mathrm{LH}(\mathpzc{E})$$ 
to $\mathrm{LH}(\mathpzc{E})\times\mathrm{LH}(\mathpzc{E})$. Denote also by $\underline{\mathrm{Hom}}_{LH}$ the restriction of 
$$\mathrm{LH_{0}}\circ\underline{\mathrm{Hom}}:(\mathbf{Ch}_{+}(\mathpzc{E})^{+})^{op}\times\mathbf{Ch}_{-}(\mathpzc{E})\rightarrow\mathrm{LH}(\mathpzc{E})$$
 to $\mathrm{LH}(\mathpzc{E})$.
 
 \begin{thm}
 $(\otimes_{\mathrm{LH}},i(k),\underline{\mathrm{Hom}}_{\mathrm{LH}})$ defines a closed symmetric monoidal structure on $\mathrm{LH}(\mathpzc{E})$. If $\mathpzc{E}$ has enough strongly flat objects then $\mathrm{LH}(\mathpzc{E})$ has enough flat objects. 
 \end{thm}
 
 \begin{proof}
 It is straightforward to prove that $\otimes_{LH}$ the associativity equivalences for $\otimes^{\mathbb{L}}$ induce associativity isomorphsisms for $\otimes_{LH}$, and similarly for commutativity, and the left and right unit coherences. The main content of this result is that $\underline{\mathrm{Hom}}_{LH}$ is an internal hom for this monoidal structure. We need to establish natural isomorphisms
 $$\mathrm{Hom}_{\mathrm{LH}(\mathpzc{E})}(X\otimes_{LH}Y,Z)\cong\mathrm{Hom}_{LH}(X,\underline{\mathrm{Hom}}(Y,Z))$$
 If we can show that this is true whenever $Y$ is in $\mathfrak{L}$, and that $\underline{\mathrm{Hom}}(-,Z)$ sends cokernels between maps in $\mathfrak{L}$ to kernels, then we are done. So assume first that $Y=L\in\mathfrak{L}$. Let $L'_{\bullet}\rightarrow X$ be a resolution by a complex in $\mathfrak{L}$, and $Z\rightarrow R_{\bullet}$ a resolution by a complex in $\mathfrak{R}$. We first claim that $\underline{\mathrm{Hom}}(L,R_{\bullet})$ is in $\mathrm{D}_{\le0}(\mathpzc{E})$. Now if 
 \begin{displaymath}
 \xymatrix{
 Z=\mathrm{Ker}(f)\ar[r] & A\ar[r]^{f} & B
 }
 \end{displaymath}
 then we have $\tau_{\ge 2}Z\cong 0$. Thus $\tau_{\ge 2}R_{\bullet}\cong\mathrm{Ker}(d_{2}^{R})\cong 0$, and 
 $$0\cong LH_{1}(R_{\bullet})\cong (0\rightarrow R_{2}\rightarrow\mathrm{Ker}(d_{1}^{R}))$$
 Now $\mathbf{Hom}(L,R_{\bullet})$ is concentrated in degrees $\ge 2$. By the above, we have $Z_{2}\mathbf{Hom}(L,R_{\bullet})\cong 0$, and 
 $$LH_{1}(\mathbf{Hom}(L,R_{\bullet}))\cong (0\rightarrow \mathbf{Hom}(L,R_{2})\rightarrow Z_{1}\mathbf{Hom}(L,R_{\bullet})\cong\mathbf{Hom}(L,\mathrm{Ker}(d_{1}^{R})))$$
 It follows that $LH_{1}(\mathbf{Hom}(L,R_{\bullet}))\cong 0$ and clearly $LH_{n}(\mathbf{Hom}(L,R_{\bullet}))\cong 0$ for $n\ge 2$ as well. We then have
 \begin{align*}
 \mathrm{Hom}_{\mathrm{LH}(\mathpzc{E})}(X\otimes_{LH}Y,Z)&\cong\mathrm{Hom}_{D(\mathpzc{E})}(L\otimes L_{\bullet}',R_{\bullet})\\
 &\cong\mathrm{Hom}_{D(\mathpzc{E})}(L'_{\bullet},\mathbf{Hom}(L,R_{\bullet}))\\
 &\cong\mathrm{Hom}_{D(\mathpzc{E})}(LH_{0}(L'_{\bullet}),LH_{0}(\mathbf{Hom}(L,R_{\bullet})))
 \end{align*}
 as required, where in the last isomorphism we have used that $L'_{\bullet}\in\mathrm{D}_{\ge0}(\mathpzc{E})$, and $\mathbf{Hom}(L,R_{\bullet})\in\mathrm{D}_{\le0}(\mathpzc{E})$.
 
 Now let $L_{1}\rightarrow L_{0}$ be a map between objects in $\mathfrak{L}$. We regard this as a complex, and extend it to a quasi-isomorphic one $L_{\bullet}\rightarrow (L_{1}\rightarrow L_{0})$, where the $0th$ and $1st$ terms of $L_{\bullet}$ are $L_{0}$ and $L_{1}$ respectively. $\underline{\mathrm{Hom}}(L_{0},Z)$ is isomorphic to the complex
 $$\underline{\mathrm{Hom}}(L_{0},R_{2})\rightarrow\underline{\mathrm{Hom}}(L_{0},R_{1})\rightarrow\underline{\mathrm{Hom}}(L_{0},\mathrm{Ker}(d_{0}^{R}))$$
 and $\underline{\mathrm{Hom}}(L_{1},Z)$ is isomorphic to the complex
  $$\underline{\mathrm{Hom}}(L_{1},R_{2})\rightarrow\underline{\mathrm{Hom}}(L_{1},R_{1})\rightarrow\underline{\mathrm{Hom}}(L_{0},\mathrm{Ker}(d_{1}^{R}))$$
  The kernel of $\underline{\mathrm{Hom}}(L_{0},R_{2})\rightarrow \underline{\mathrm{Hom}}(L_{1},R_{2})$ is then computed as $LH_{1}$ of the cone of the natural map between these complexes, i.e. $LH_{1}$ of the complex
  $$\underline{\mathrm{Hom}}(L_{0},R_{2})\rightarrow\underline{\mathrm{Hom}}(L_{0},R_{1})\oplus\underline{\mathrm{Hom}}(L_{1},R_{2})\rightarrow \underline{\mathrm{Hom}}(L_{0},\mathrm{Ker}(d_{0}^{R}))\oplus \underline{\mathrm{Hom}}(L_{1},R_{1})\rightarrow \underline{\mathrm{Hom}}(L_{1},\mathrm{Ker}(d_{0}^{R}))$$
  where $\underline{\mathrm{Hom}}(L_{1},\mathrm{Ker}(d_{0}^{R}))$ is in degree $0$. This is the complex
  \begin{displaymath}
  \xymatrix{
\mathrm{Ker}(d_{2})\ar[r] &\underline{\mathrm{Hom}(L_{0},R_{1})}\oplus\underline{\mathrm{Hom}}(L_{1},R_{2})\ar[r]^{\;\;\;\;d_{2}}&\\
\ar[r]&\underline{\mathrm{Hom}}(L_{0},\mathrm{Ker}(d_{0}^{R}))\times_{\underline{\mathrm{Hom}}(L_{1},\mathrm{Ker}(d_{0}^{R}))}\underline{\mathrm{Hom}}(L_{1},R_{1})
  }
  \end{displaymath}
  Now consider $\underline{\mathrm{Hom}}(L_{\bullet},R_{\bullet}))$. In low degrees this complex is 
  \begin{displaymath}
  \xymatrix{
  \underline{\mathrm{Hom}}(L_{0},R_{2})\ar[r] &   \underline{\mathrm{Hom}}(L_{1},R_{2})\oplus   \underline{\mathrm{Hom}}(L_{0},R_{1}) \\
 \ar[r] & \underline{\mathrm{Hom}}(L_{0},R_{0})\oplus\underline{\mathrm{Hom}}(L_{1},R_{1})\oplus\underline{\mathrm{Hom}}(L_{2},R_{2})\\
 \ar[r]^{d_{0}'\;\;\;\;\;\;\;\;\;\;\;\;\;\;\;\;\;\;\;\;\;\;\;\;\;\;\;\;\;\;\;\;\;\;\;\;\;\;\;\;\;\;\;\;\;\;\;\;\;\;\;\;\;\;\;\;\;\;\;\;\;\;\;\;} &  \underline{\mathrm{Hom}}(L_{0},R_{-1})\oplus\underline{\mathrm{Hom}}(L_{1},R_{0})\oplus\underline{\mathrm{Hom}}(L_{2},R_{1})\oplus\underline{\mathrm{Hom}}(L_{3},R_{2})
  }
  \end{displaymath}
  where $\underline{\mathrm{Hom}}(L_{0},R_{2})$ is in degree $2$.  We compute $\mathrm{LH}_{0}$ of this complex. Note that in degree $1$ this will be $\underline{\mathrm{Hom}}(L_{0},R_{1})\oplus\underline{\mathrm{Hom}}(L_{1},R_{2})$, so things are looking promising. Now there is a natural map $\underline{\mathrm{Hom}}(L_{0},R_{0})\oplus\underline{\mathrm{Hom}}(L_{1},R_{1})\oplus\underline{\mathrm{Hom}}(L_{2},R_{2})\rightarrow\underline{\mathrm{Hom}}(L_{0},R_{0})\oplus\underline{\mathrm{Hom}}(L_{1},R_{1})$ given by projection. We claim that this map restricts to an isomorphism
  $$\mathrm{Ker}(d_{0}')\cong\underline{\mathrm{Hom}}(L_{0},\mathrm{Ker}(d_{0}^{R}))\times_{\underline{\mathrm{Hom}}(L_{1},\mathrm{Ker}(d_{0}^{R}))}\underline{\mathrm{Hom}}(L_{1},R_{1})$$
  By applying applying $\mathrm{Hom}(L,-)$ for each $L\in\mathfrak{L}$ and appealing to Yoneda, we may work with elements. Let $(f,g,h)\in\mathrm{Ker}(d_{0}')$. This means precisely that 
  \begin{enumerate}
  \item
 $d_{0}^{R}\circ f=0$
  \item
  $f\circ d_{1}^{L}=d_{1}^{R}\circ g$
  \item
  $g\circ d_{2}^{L}=-d_{2}^{R}\circ h$
  \item
  $h\circ d_{3}^{L}=0$
  \end{enumerate}
  The first two equations imply that the image of the map $\mathrm{Ker}(d_{0}')\rightarrow \underline{\mathrm{Hom}}(L_{0},R_{0})\oplus\underline{\mathrm{Hom}}(L_{1},R_{1})$ lands in $\underline{\mathrm{Hom}}(L_{0},\mathrm{Ker}(d_{0}^{R}))\times_{\underline{\mathrm{Hom}}(L_{1},\mathrm{Ker}(d_{0}^{R}))}\underline{\mathrm{Hom}}(L_{1},R_{1})$.
 Since $d_{2}^{R}$ is an isomorphism onto its image, if $g=0$ then $h=0$. Thus the map is injective. Next we show it is a surjection onto $\underline{\mathrm{Hom}}(L_{0},\mathrm{Ker}(d_{0}^{R}))\times_{\underline{\mathrm{Hom}}(L_{1},\mathrm{Ker}(d_{0}^{R}))}\underline{\mathrm{Hom}}(L_{1},R_{1})$. The first two equations imply that its image is at least contained in  $\underline{\mathrm{Hom}}(L_{0},\mathrm{Ker}(d_{0}^{R}))\times_{\underline{\mathrm{Hom}}(L_{1},\mathrm{Ker}(d_{0}^{R}))}\underline{\mathrm{Hom}}(L_{1},R_{1})$. It is surjectiv. Let $(f,g)\in\underline{\mathrm{Hom}}(L_{0},\mathrm{Ker}(d_{0}^{R}))\times_{\underline{\mathrm{Hom}}(L_{1},\mathrm{Ker}(d_{0}^{R}))}\underline{\mathrm{Hom}}(L_{1},R_{1})$ be given. Since $d_{2}^{R}$ is an isomorphism onto $\mathrm{Ker}(d_{1}^{R})$ and $d_{1}^{R}\circ g\circ d_{2}^{L}\cong g\circ d_{1}^{L}\circ d_{2}^{L}\cong 0$ there is a unique $h$ such that $-d_{2}^{R}\circ h=g\circ d_{2}^{L}$. This also implies that $h\circ d_{3}^{L}=0$, so that $(f,g,h)\in \mathrm{Ker}(d_{0}')$. 
\end{proof}

\subsubsection{Quillen Adjunctions of Complexes}

Here we relate Hovey triples to Quillen adjunctions. 

\begin{lem}
Let 
$$\adj{L}{\mathpzc{D}}{\mathpzc{E}}{R}$$
be an adjunction of exact categories, and let $(\mathfrak{L}_{\mathpzc{D}},\mathfrak{R}_{\mathpzc{D}})$ be a $dg$-compatible cotorsion pair on $\mathpzc{D}$, and $(\mathfrak{L}_{\mathpzc{E}},\mathfrak{R}_{\mathpzc{E}})$ a $dg$-compatible cotorsion pair on $\mathpzc{E}$ such that $R$ sends $\mathfrak{R}_{\mathpzc{E}}$ to $\mathfrak{R}_{\mathpzc{D}}$. Then, when equipped with the model structures induced by $(\mathfrak{L}_{\mathpzc{D}},\mathfrak{R}_{\mathpzc{D}})$ and $(\mathfrak{L}_{\mathpzc{E}},\mathfrak{R}_{\mathpzc{E}})$ respectively, the adjunction
$$\adj{L}{\mathrm{Ch}(\mathpzc{D})}{\mathrm{Ch}(\mathpzc{E})}{R}$$
is Quillen. 
\end{lem}

\begin{proof}
By Corollary \ref{cor:adjcotors} $R$ preserves exact sequences
$$0\rightarrow A_{\bullet}\rightarrow B_{\bullet}\rightarrow C_{\bullet}\rightarrow 0$$
where each $A_{n}\in\mathfrak{R}_{\mathpzc{E}}$, and $L$ preserves exact sequences
$$0\rightarrow A_{\bullet}\rightarrow B_{\bullet}\rightarrow C_{\bullet}\rightarrow 0$$
with each $C_{n}\in\mathfrak{L}_{\mathpzc{D}}$. Moreover, since $R$ preserves exact sequences whose first term is in $\mathfrak{R}_{\mathpzc{E}}$ to exact sequences, it sends $\tilde{\mathfrak{R}_{\mathpzc{E}}}$ to $\tilde{\mathfrak{R}_{\mathpzc{D}}}$. Now let $Y\in\widetilde{dg\mathfrak{R}_{\mathpzc{E}}}$, and $X\in\tilde{\mathfrak{L}_{\mathpzc{D}}}$. Then $L(X)\in\tilde{\mathfrak{L}}_{\mathpzc{E}}$. We thus have
$$\mathbf{Hom}(X,R(Y))\cong\mathbf{Hom}(L(X),Y)$$
is acyclic. Thus $R(Y)\in\widetilde{dg\mathfrak{R}_{\mathpzc{D}}}$
\end{proof}

\subsubsection{The Dold-Kan Equivalences}\label{subsubsec:DoldKan}

As explained in \cite{kelly2016homotopy} Section 6.4.1 and following \cite{castiglioni2004cosimplicial}, for any weakly idempotent complete exact category $\mathpzc{E}$ there is a Dold-Kan equivalence
$$\adj{\Gamma}{\mathrm{Ch}_{\ge0}(\mathpzc{E})}{\mathrm{s}\mathpzc{E}}{N}$$

$\mathpzc{E}^{op}$ is also weakly idempotent complete, so we also get a Dold-Kan equivalence
$$\adj{\Gamma_{c}}{\mathrm{Ch}_{\le0}(\mathpzc{E})}{\mathrm{cs}\mathpzc{E}}{N_{c}}$$

If $\mathpzc{E}$ is locally presentable then there is an alternative Dold-Kan correspondence

$$\adj{Q}{\mathrm{Ch}_{\le0}(\mathpzc{E})}{\mathrm{cs}\mathpzc{E}}{H}$$

As explained in \cite{kelly2016homotopy} $Q$ is constructed as follows. For $n\ge0$ let $V^{n}$ denote the abelian group given by the kernel of the canonical map $\bigoplus_{i=0}^{n}\mathbb{Z}\rightarrow\mathbb{Z}$. Let $\{e_{i}:0\le i\le n\}$ be the standard basis of $\bigoplus_{i=0}^{n}\mathbb{Z}$ so that $\{v_{i}=e_{i}-e_{0}:0\le i\le n\}$ is a basis of $V^{n}$. $V^{\bullet}$ may be regarded as a cosimplicial abelian group, where for $\alpha:[n]\rightarrow [m]$ we define $\alpha(v_{i})=v_{\alpha(i)}-v_{\alpha(0)}$ and extend by linearity. Let $T(V)$ denote the tensor algebra in $\mathrm{cs}\mathrm{Ab}$ on $V^{\bullet}$, with multiplication denoted by $\mu$. Then $Q:\mathrm{Ch}_{\le0}(\mathpzc{E})\rightarrow\mathrm{cs}\mathpzc{E}$ is defined by sending $(A_{\bullet},d)$ to
$$(QA)_{\bullet}=\bigoplus_{i=0}^{\infty}(A_{-n}\otimes T^{i}(\mathbb{Z}^{n}))$$
where we use the natural tensoring of $\mathpzc{E}$ over $\mathrm{Ab}$. For any map $\alpha:[n]\rightarrow[m]$ of finite sets define $(QA)_{n}\rightarrow (QA)_{m}$ by
$$\mathrm{Id}_{A}\otimes\alpha+d\otimes\mu(v_{\alpha(0)\otimes\alpha})$$
This defines an object of $\mathrm{cs}\mathpzc{E}$. This functor clearly commutes with colimits, so has a right adjoint $H$.

\begin{rem}
 This right adjoint exists under more general circumstances, namely that $\mathpzc{E}$ is cocomplete and has a generating set. An identical proof to \cite{kelly2016homotopy} Proposition 4.4.87 works in this generality.
 \end{rem}

Exactly as in \cite{castiglioni2004cosimplicial} Theorem 4.2 i) there is a natural homotopy equivalence $\hat{p}:Q\rightarrow\Gamma_{c}$ $\hat{p}:Q\rightarrow\Gamma_{c}$

\begin{lem}
Let $\mathpzc{E}$ be a locally presentable exact category with a $dg$-compatible cotorsion pair $(\mathfrak{L},\mathfrak{R})$ such that $\mathfrak{L}$ is presentably deconstructible in itself. Equip $\mathrm{Ch}_{\le0}(\mathpzc{E})$ with the transferred model structure. The adjunction
$$\adj{Q}{\mathrm{Ch}_{\le0}(\mathpzc{E})}{\textit{cs}\mathpzc{E}}{H}$$
is a Quillen equivalence.
\end{lem}
\begin{proof}
This was shown for the projective model structure in \cite{kelly2016homotopy} Lemma 4.4.87. The proof here is very similar, but we repeat it for completeness. As in loc. cit. it suffices to prove that the adjunction is a Quillen adjunction. $\Gamma_{c}$ preserves all equivalences by definition, and $\hat{p}:Q\rightarrow\Gamma_{c}$ is a natural homotopy equivalence. Thus $Q$ preserves all equivalences.  It therefore suffices to prove that it preserves cofibrations. Now by Proposition \ref{prop:cofibrantgenerators} a set of cofibrant generators for $\mathrm{Ch}(\mathpzc{E})$ will be of the form
$$\widetilde{I}=\{0\rightarrow D^{n}(G)\}_{G\in\mathcal{G},n\in\mathbb{Z}}\cup\{S^{n-1}(G)\rightarrow D^{n}(G)\}_{G\in\mathcal{G},n\in\mathbb{Z}}\cup\{S^{n}(k_{i}):S^{n}(Y_{i})\rightarrow S^{n}(Z_{i})\}_{i\in\mathcal{I},n\in\mathbb{Z}}$$
for some sets $\mathcal{G}$ and $\mathcal{I}$. Thus a set of cofibrant generators for $\mathrm{Ch}_{\le0}(\mathpzc{E})$ will be of the form
$$\{0\rightarrow D^{n}(G)\}_{G\in\mathcal{G},n\in\mathbb{Z}_{\ge0}}\cup\{S^{n-1}(G)\rightarrow D^{n}(G)\}_{G\in\mathcal{G},n\in\mathbb{Z}_{\ge0}}\cup$$
$$\cup\{S^{n}(k_{i}):S^{n}(Y_{i})\rightarrow S^{n}(Z_{i})\}_{i\in\mathcal{I},n\in\mathbb{Z}_{\ge0}}\cup\{S^{0}(G)\rightarrow 0\}_{G\in\mathcal{G}}$$
By definition of the class of cofibrations for the model structure on $\mathbf{cs}\mathpzc{E}$ it suffices to show that $N_{c}Q(f)$ is a cofibration for any generating cofibration $f:X\rightarrow Y$. By the proof of Theorem 4.2 in \cite{castiglioni2004cosimplicial}, for $X$ an object of $\mathrm{Ch}_{\le0}(\mathpzc{E})$, $(N_{c}QX)_{-n}\cong\bigoplus_{-r=n}^{\infty}X_{r}\otimes\mathbb{Z}[Sur_{r,n}]$, where $Sur_{r,n}$ is the set of surjections from the set with $r$ elements to the set with $n$ elements. In particular, if $X$ is bounded below then so is $N_{c}QX$. First consider a cofibration of the form $S^{0}(G)\rightarrow 0$. $(N_{c}QS^{0}(G))_{m}=0$ for $m<0$, so $N_{c}QS^{0}(G)\cong S^{0}(G)$, and $N_{c}Q(S^{0}(G)\rightarrow 0)\cong S^{0}(G)\rightarrow 0$ is a cofibration. All other generating cofibrations are degrree-wise admissible monmorphisms $X\rightarrow Y$ of bounded below complexes with cokernel degree-wise in $\mathfrak{L}$. Both $N_{c}$ and $Q$ commute with all colimits, so $N_{c}Q(X)\rightarrow N_{c}Q(Y)$ is a degree-wise admissiblee monomorphism of bounded below complexes with  cokernel degree-wise in $\mathfrak{L}$. Thus it is a cofibration, as required.
\end{proof}

\section{Model Structures on Accessible Exact Categories}\label{ref:modelstructuresonacc}

 \subsection{Existence Results in Accessible Categories}
 
 In this section we apply the general existence result, Theorem \ref{thm:amendingmodel} to model structures on locally presentable exact categories.
 
 \subsubsection{Existence of Model Structures on Complexes}
 
 Let us begin with the main result of interest, namely existence theorems for model structures on complexes, which we will apply later to the flat model structure.
 
 \begin{thm}\label{thm:existencemodelpresentablecomplex}
Let $\mathpzc{E}$ be a purely $\lambda$-accessible exact category in which $\lambda$-pure monomorphisms are admissible. Let $\mathfrak{L}$ be a class of objects in $\mathpzc{E}$ which is strongly $\lambda$-pure subobject stable, is closed under transfinite extensions by admissible monomorphisms, and contains a generator. Suppose that $\mathpzc{E}$ is weakly $\mathbf{AdMon}_{\mathfrak{L}}$-elementary. Then
\begin{enumerate}
\item
$(\mathfrak{L},\mathfrak{R})$ is $dg_{\ge0}$-compatible.
\item
$(\mathfrak{L},\mathfrak{R})$ is $dg$-compatible whenever $\tilde{\mathfrak{L}}=\widetilde{dg\mathfrak{L}}\cap\mathfrak{W}$, in particular when $\mathfrak{L}$ is hereditary.
\item
$(\mathfrak{L},\mathfrak{R})$ is $dg_{\le0}$-compatible  whenever $\tilde{\mathfrak{L}}=\widetilde{dg\mathfrak{L}}\cap\mathfrak{W}$, in particular when $\mathfrak{L}$ is hereditary.

\end{enumerate}
\end{thm}

\begin{proof}
We use Corollary \ref{cor:modelexistencedec}
\begin{enumerate}
\item
The class $\widetilde{dg\mathfrak{L}}\cap\mathfrak{W}\cap\mathrm{Ch}_{\ge0}(\mathpzc{E})$ is precisely the class of acyclic complexes with components in $\mathfrak{L}$. Since $\mathfrak{L}$ is strongly $\lambda$-pure subobject stable, it follows that  $\widetilde{dg\mathfrak{L}}\cap\mathfrak{W}\cap\mathrm{Ch}_{\ge0}(\mathpzc{E})$  is strongly $\lambda$-pure subobject stable by Proposition \ref{prop:acyclicintersctstronglypure}, and therefore it is presentably deconstructible in itself.
\item
The class $\tilde{\mathfrak{L}}$ consists of acyclic complexes $X$ with each $Z_{n}X\in\mathfrak{L}$. A $\lambda$-pure subobject and quotient of such an object is acyclic. Moreover if $Y\rightarrow X$ is a $\lambda$-pure monomorphism in $\mathrm{Ch}(\mathpzc{E})$, then as in the proof of Lemma \ref{lem:Znsubexact}, $Z_{n}Y\rightarrow Z_{n}X$ is a $\lambda$-pure monomorphism. Thus each $Z_{n}Y$ and each $Z_{n}X\big\slash Z_{n}Y\cong Z_{n}(X\big\slash Y)$ is in $\mathfrak{L}$. Hence $\tilde{\mathfrak{L}}$ is $\lambda$-pure subobject stable. Now consider a fibre product diagram
\begin{displaymath}
\xymatrix{
N\ar[d]^{f}\ar[r]^{g} & A\ar[d]^{i}\\
M\ar[r]^{h} & B
}
\end{displaymath}
where $A,B\in\tilde{\mathfrak{L}}$, $i$ is a $\lambda$-pure monomorphism with cokernel in $\tilde{\mathfrak{L}}$, and $h$ is an admissible epimorphism, and $M\in\tilde{\mathfrak{L}}$. Note that $N$ is acyclic again by Proposition \ref{prop:acyclicintersctstronglypure}. Taking $Z_{n}$ gives another fibre product diagram
\begin{displaymath}
\xymatrix{
Z_{n}N\ar[d]^{f}\ar[r]^{g} & Z_{n}A\ar[d]^{i}\\
Z_{n}M\ar[r]^{h} & Z_{n}B
}
\end{displaymath}
Now $Z_{n}A\rightarrow Z_{n}B$ is a $\lambda$-pure monomorphism.  Moreover since $M$ and $B$ are acyclic, $Z_{n}M\rightarrow Z_{n}B$ is an admissible epimorphism. Thus each $Z_{n}N\in\mathfrak{L}$. Hence $\tilde{\mathfrak{L}}$ is   strongly $\lambda$-pure subobjet stable, and is therefore presentably deconstructible in itself. Note this also proves (3).
\end{enumerate}
\end{proof}
\begin{cor}
Let $\mathpzc{E}$ be a purely $\lambda$-accessible exact category with a generator which is weakly elementary. Then the injective cotorsion pair $(\mathpzc{E},\mathrm{Inj})$ is $dg_{\ge0}$-compatible, $dg$-compatible, and $dg_{\le0}$-compatible.
\end{cor}

\subsubsection{The Left Heart of a Locally Presentable Exact Category}


\begin{lem}
Let $\mathpzc{E}$ be an exact category with finite limits, and let $(\mathfrak{L},\mathfrak{R})$ be a cotorsion pair on $\mathpzc{E}$ which is both $dg$-compatible and $dg_{\ge0}$-compatible.
\begin{enumerate}
\item
The inclusion $i:\mathrm{Ch}_{\ge0}(\mathpzc{E})\rightarrow\mathrm{Ch}(\mathpzc{E})$ is a left Quillen coreflection.
\item
The essential image of the functor of homotopy categories $\mathrm{Ho}(\mathrm{Ch}_{\ge0}(\mathpzc{E}))\rightarrow\mathrm{Ho}(\mathrm{Ch}(\mathpzc{E}))$ is the category of complexes $X$ such that $\mathrm{LH}_{n}(X)\cong 0$ for $n<0$.
\end{enumerate}
\end{lem}

\begin{proof}
\begin{enumerate}
\item
Clearly the inclusion $i$ is left Quillen, and the unit
$$X\rightarrow\tau_{\ge0}X$$
is an equivalence. Thus we get a Quillen reflection
$$\adj{i}{\mathrm{Ch}_{\ge0}(\mathpzc{E})}{\mathrm{Ch}(\mathpzc{E})}{\tau_{\ge0}}$$
\item
The second claim follows since $\tau_{\ge0}$ is the truncation functor for the left $t$-structure on $\mathrm{Ho}(\mathrm{Ch}(\mathpzc{E}))$. 
\end{enumerate}
\end{proof}

\begin{lem}
Let $\mathpzc{E}$ be a purely locally $\lambda$-presentable exact category. Suppose further that filtered colimits are exact and commute with kernels in $\mathpzc{E}$. Then $\mathrm{LH}(\mathpzc{E})$ is Grothendieck abelian. In partiuclar it is locally presentable.
\end{lem}

\begin{proof}
Consider the left heart $\mathrm{LH}(\mathpzc{E})$. Let $G$ be a generator for $\mathpzc{E}$. By \cite{henrard2021left} Corollary 3.10  $G$ is also a generator for $\mathrm{LH}(\mathpzc{E})$. It remains to prove that filtered colimits in $\mathrm{LH}(\mathpzc{E})$ are exact. Equip $\mathrm{Ch}_{\ge0}(\mathpzc{E})$ with the injective model structure, and consider the $(\infty,1)$-category $\mathbf{Ch}_{\ge0}(\mathpzc{E})$ which it presents. $\mathrm{LH}(\mathpzc{E})$ is a reflective subcategory of $\mathbf{Ch}_{\ge0}(\mathpzc{E})$. The left adjoint to the inclusion is the functor $\mathrm{LH}_{0}$. In particular $\mathrm{LH}(\mathpzc{E})$ is cocomplete. Let $\mathcal{I}$ be a filtered category and consider both the category $\mathrm{Fun}(\mathcal{I},\mathrm{Ch}_{\ge0}(\mathrm{LH}(\mathpzc{E})))$, and the functor $\mathrm{colim}:\mathrm{Fun}(\mathcal{I},\mathrm{Ch}_{\ge0}(\mathrm{LH}(\mathpzc{E})))\rightarrow\mathrm{Ch}_{\ge0}(\mathrm{LH}(\mathpzc{E}))$. By taking functorial resolution of objects in $\mathrm{LH}(\mathpzc{E})$ by complexes in $\mathrm{Ch}_{\ge0}(\mathpzc{E})$, we get a functor $\mathbb{L}\colim:\mathrm{Ch}(\mathrm{LH}(\mathpzc{E}))\rightarrow\mathrm{Ch}_{\ge0}(\mathpzc{E})$. Now let $F:\mathcal{I}\rightarrow\mathrm{Ch}(\mathpzc{E})$ be a functor. Again by the assumption that filtered colimits are exact in $\mathpzc{E}$ and commute with kernels, we have $\mathrm{LH}_{n}(\colim F)\cong\colim\mathrm{LH}_{n}(F)$. This in turn implies that the derived colimit functor
$$\mathbb{L}\colim:\mathrm{Fun}(\mathcal{I},\mathrm{Ch}_{\ge0}(\mathrm{LH}(\mathpzc{E})))\rightarrow\mathrm{Ch}_{\ge0}\mathrm{LH}(\mathpzc{E})$$
has no higher derived functors, and in particular, that filtered colimits are exact in $\mathrm{LH}(\mathpzc{E})$.
\end{proof}

\begin{example}
Let $\mathpzc{E}$ be an elementary exact category, i.e. it has a set of tiny ($\aleph_{0}$-compact) projective generators. It is shown in \cite{kelly2021analytic} that the left heart of $\mathpzc{E}$ is equivalent to the free sfited cocompletion $\mathrm{P}_{\Sigma}(\mathcal{P})$ of a small subcategory $\mathcal{P}$ of compact projective generators of $\mathpzc{E}$. This is evidently a Grothendieck abelian category. Moreover the inclusion $\mathpzc{E}\rightarrow\mathrm{P}_{\Sigma}(\mathcal{P})$ commutes with filtered colimits. It follows that $\mathpzc{E}$ is locally finitely presentable. Moreover it has exact filtered colimits.
\end{example}

\subsubsection{Existence of Injective Cotorsion Pairs}

Let us now consider injective cotorsion pairs in categories of complexes.
%

\begin{example}
Let $\mathpzc{E}$ be a purely $\lambda$-accessible exact category with a generator which is weakly elementary. The injective model structure on $\mathrm{Ch}(\mathpzc{E})$ is injective, as one would hope.
\end{example}

As another example, we consider algebras over monads on $\mathrm{Ch}(\mathpzc{E})$ and $\mathrm{Ch}_{\ge0}(\mathpzc{E})$, generalising \cite{MR3166360} Proposition 3.11. Let $\mathpzc{E}$ be a purely locally $\lambda$-presentable exact category which is weakly $\mathbf{AdMon}$-elementary. Let $T:\mathrm{Ch}(\mathpzc{E})\rightarrow\mathrm{Ch}(\mathpzc{E})$ be an additive monad  which commutes with colimits. Then by \cite{kelly2016homotopy} there exists a exact structure on ${}_{T}\mathrm{Mod}(\mathpzc{E})$ in which a sequence
$$0\rightarrow X\rightarrow Y\rightarrow Z\rightarrow 0$$
is exact precisely if 
$$0\rightarrow |X|\rightarrow |Y|\rightarrow |Z|\rightarrow 0$$
is an exact sequence $\mathpzc{E}$, where $|-|:{}_{T}\mathrm{Mod}(\mathpzc{E})\rightarrow\mathpzc{E}$ is the forgetful functor. Note that if $\mathpzc{E}$ is an exact category, then this class of exact sequences on ${}_{T}\mathrm{Mod}(\mathpzc{E})$ also defines an exact structure.

\begin{prop}
Let $\mathpzc{E}$ be a purely $\lambda$-accessible exact category with a generator which is weakly elementary. Let $\mathpzc{M}$ be either $\mathrm{Ch}(\mathpzc{E})$ or $\mathrm{Ch}_{\ge0}(\mathpzc{E})$, and let $T$ be a cocontinuous monad in $\mathpzc{M}$. Then there is a combinatorial model structure on ${}_{T}\mathrm{Mod}(\mathpzc{M})$ such that
\begin{enumerate}
\item
the cofibrations are the admissible monomorphisms.
\item
the weak equivalences are the quasi-isomorphisms of underlying complexes.
\end{enumerate}
\end{prop}

\begin{cor}
Let $\mathpzc{E}$ be a closed symmetric monoidal purely $\lambda$-accessible exact category which is weakly elementary. Let $\mathpzc{M}$ be either $\mathrm{Ch}(\mathpzc{E})$ or $\mathrm{Ch}_{\ge0}(\mathpzc{E})$. These are both closed symmetric monoidal categories. Let $A\in\mathrm{Alg}_{\mathfrak{Ass}}(\mathpzc{M})$ be an associative $dg$-algebra. Then there is a model structure on the category ${}_{A}\mathrm{Mod}(\mathpzc{M})$ of left $A$-modules such that
\begin{enumerate}
\item
the cofibrations are the admissible monomorphisms.
\item
the weak equivalences are the quasi-isomorphisms of underlying complexes.
\end{enumerate}
\end{cor}

We can classify the trivially fibrant objects for model structures induced by injective cotorsion pairs exactly as in \cite{estrada2023k} Theorem 3.4.

\begin{prop}\label{prop:classifyingorth}
Let $\mathpzc{E}$ be a weakly idempotent complete exact category, and let $\mathfrak{C}$ be a class of objects in $\mathpzc{E}$ containing all objects of the form $D^{n}(G)$ for $n\in\mathbb{Z}$ and $G\in\mathpzc{E}$. Then $\mathfrak{C}^{\perp}$ consists of all complexes $X$ such that each $X_{n}$ is injective, and $E\rightarrow X$ is null-homotopic whenever $E\in\mathfrak{C}$. 
\end{prop}

\begin{proof}
Let $X\in\mathfrak{C}^{\perp}$. Then for every $G\in\mathpzc{E}$ and every $n\in\mathbb{Z}$ we have $0\cong\mathrm{Ext}^{1}(D^{n}(G),X)\cong\mathrm{Ext}(G,X_{n})$. Thus $X_{n}$ is injective. In particular for any $E\in\mathfrak{C}$ we have $\mathrm{Ext}^{1}_{dw}(E,X)\cong\mathrm{Ext}^{1}(E,X)\cong0$. Conversely suppose that $X$ is degree-wise injective and $E\rightarrow X$ is null-homotopic for any $E\in\mathfrak{C}$. We then have
$$\mathrm{Ext}^{1}(E,X)\cong\mathrm{Ext}_{dw}^{1}(E,X)\cong 0$$
\end{proof}

\subsubsection{Changing the Exact Structure for Complexes}

We have the following consequence of Corollary \ref{cor:changingtheunderlyingexactstructure} for chain complexes.

\begin{cor}\label{cor:underlyingexactstructurecomplexes}
Let $(\mathpzc{E},\mathpzc{Q})$ and $(\mathpzc{E},\mathpzc{Q}')$ be purely $\lambda$-accessible exact categories with the same underlying category $\mathpzc{E}$, and with $\mathpzc{Q}\subset\mathpzc{Q}'$. Let $(\mathfrak{C}',\mathfrak{W}',\mathfrak{F}')$ be a Hovey triple on $\mathrm{Ch}(\mathpzc{E},\mathpzc{Q}')$ with $\mathfrak{W}'$ the class of acyclic complexes in $\mathrm{Ch}(\mathpzc{E},\mathpzc{Q}')$. Suppose that 
\begin{enumerate}
\item
$\mathfrak{C}'$ is strongly $\lambda$-pure subobject stable in $(\mathpzc{E},\mathpzc{Q})$.
\item
 $(\mathrm{Ch}(\mathpzc{E}),\mathpzc{Q})$ is weakly $\mathbf{AdMon}_{\mathfrak{C}'}$-elementary, 
 \item
 $\mathfrak{W}'\cap\mathfrak{C}'$ and $\mathfrak{C}'$ are closed under transfinite extensions in $(\mathrm{Ch}(\mathpzc{E}),\mathpzc{Q})$
 \item
 $\mathfrak{C}'\cap\mathfrak{W}'$ contains a generator for $(\mathrm{Ch}(\mathpzc{E}),\mathpzc{Q})$,
 \item
 $\mathfrak{C}'$ contains a set of generators of $(\mathrm{Ch}(\mathpzc{E}),\mathpzc{Q}')$ of the form $\{S^{n}(G):n\in\mathbb{Z},G\in\mathcal{G}\}$ with $\mathcal{G}$ a set of generators of $(\mathpzc{E},\mathpzc{Q}')$. 
 \end{enumerate}
 Then 
 $$(\mathfrak{C}',\mathfrak{W}',(\mathfrak{C}\cap\mathfrak{W}')^{\perp_{\mathpzc{Q}}})$$
 is a Hovey triple on $\mathrm{Ch}(\mathpzc{E},\mathpzc{Q})$ such that the adjunction
 $$\adj{Id}{\mathrm{Ch}(\mathpzc{E},\mathpzc{Q})}{\mathrm{Ch}(\mathpzc{E},\mathpzc{Q}')}{Id}$$
 is a Quillen equivalence.
\end{cor}

\begin{proof}
By assumption $(\mathrm{Ch}(\mathpzc{E}),\mathpzc{Q}')$ is equipped with a left pseudo-compatible model structure determined by $(\mathfrak{C}',\mathfrak{W}')$. $\mathfrak{C}'$ is assumed strongly $\lambda$-pure subobject stable, and the class of $\mathfrak{C}'\cap\mathfrak{W}'$ is then autormatically is automatically strongly $\lambda$-pure subobject stable. By assumption $\mathfrak{W}'\cap\mathfrak{C}'$ and $\mathfrak{C}'$ are closed under transfinite extensions in $(\mathrm{Ch}(\mathpzc{E}),\mathpzc{Q})$, and  $(\mathrm{Ch}(\mathpzc{E}),\mathpzc{Q})$ is weakly $\mathbf{AdMon}_{\mathfrak{C}'}$-elementary. Now $\mathfrak{C}'$ contains a set of generators of the form $\{S^{n}(G):n\in\mathbb{Z},G\in\mathcal{G}\}$. Any map with the right lifting property against all maps of the form $0\rightarrow S^{n}(G)$ must be in $\mathfrak{W}'$. Let $\mathbf{AdMon}'$ denote the class of admissible monomorphisms in $\mathrm{Ch}(\mathpzc{E},\mathpzc{Q}')$. Let $\mathcal{W}'$ denote the class of weak equivalences in $\mathrm{Ch}(\mathpzc{E},\mathpzc{Q}')$. Since $\mathbf{AdMon}'_{\mathfrak{W}'}=\mathbf{AdMon}'\cap\mathcal{W}'$ we clearly have $\mathbf{AdMon}_{\mathfrak{W}'}=\mathbf{AdMon}\cap\mathcal{W}'$ in $\mathrm{Ch}(\mathpzc{E},\mathpzc{Q})$ as well. Thus there is a model structure on $(\mathrm{Ch}(\mathpzc{E}),\mathpzc{Q})$ whose cofibrant objects are $\mathfrak{C}'$ and whose weak equivalences are $\mathcal{W}'$. Finally, since $\mathfrak{C}'\cap\mathfrak{W}'$ contains a generator for $(\mathrm{Ch}(\mathpzc{E}),\mathpzc{Q})$ it is a compatible model structure. 
\end{proof}

We also have the following consequence of Corollary \ref{cor:goingup}.

\begin{cor}\label{cor:goingupforcomplexes}
Let $\mathpzc{E}$ be purely $\lambda$-accessible exact category, and $\mathpzc{D}$ a purely $\lambda$-accessible reflective thick exact subcategory, such that the inclusion commutes with transfinite compositions. Let $(\mathfrak{L},\mathfrak{R})$ be a cotorsion pair on $\mathpzc{D}$ such that $\mathfrak{L}$ is strongly $\lambda$-pure subobject stable, and $\mathfrak{L}$ contains a generator for $\mathpzc{E}$. Finally suppose that $\mathpzc{D}$ is weakly $\mathbf{AdMon}_{\mathfrak{L}}$-elementary, and $\mathfrak{L}$ is closed under transfinite extensions in $\mathpzc{E}$. Then 
\begin{enumerate}
\item
$(\mathfrak{L},\mathfrak{R})$ is a $dg$-compatible cotorsion pair on $\mathpzc{D}$
\item
$(\mathfrak{L},\mathfrak{R}^{\perp_{\mathpzc{E}}})$ is a $dg$-compatible cotorsion pair on $\mathpzc{E}$
\item
the adjunction
$$\adj{L}{\mathrm{Ch}(\mathpzc{E})}{\mathrm{Ch}(\mathpzc{D})}{i}$$
is a Quillen equivalence for the corresponding model structures. If $\mathpzc{E}$ (and hence $\mathpzc{D}$) has finite limits, then this Quillen equivalence is $t$-exact for the left exact structure. In particular there is an equivalence of categories
$$\mathrm{LH}(\mathpzc{E})\cong\mathrm{LH}(\mathpzc{D})$$
\end{enumerate}
\end{cor}

\begin{proof}
Everything except $t$-exactness is immediate from Corollary \ref{cor:goingup}. $t$-exactness is also straightforward. Indeed $i\cong\mathbb{R}i$ is clearly $t$-exact. To compute $\mathbb{L}L(E_{\bullet})$, we take a resolution $B_{\bullet}\rightarrow E_{\bullet}$ with $B_{\bullet}\in\widetilde{dg\mathfrak{L}}\subseteq\mathrm{Ch}(\mathpzc{D})$. In particular $\mathbb{L}L(E_{\bullet})\cong L(B_{\bullet})=B_{\bullet}$, and $\mathbb{L}L$ is clearly also $t$-exact. 
\end{proof}

\subsection{Recollements}

In \cite{MR3459032}, following work of \cite{MR3161097} and \cite{MR2681709}, Gillespie provides condition on a triplet of injective cotorsion pairs on an exact category, so that the corresponding triangulated homotopy categories fit into a recollement. In this section we establish some general results concering the existence of recollements of model structures on locally presentable exact categories using techniques from the theory of stable $(\infty,1)$-categories. 

\begin{lem}\label{lemrecollemodel}
Let $\mathpzc{M}$ and $\mathpzc{N}$ be stable, combinatorial, model categories and
$$\adj{L}{\mathpzc{M}}{\mathpzc{N}}{R}$$
a Quillen adjunction such that
\begin{enumerate}
\item
$L$ preserves all equivalences.
\item
$R$ is fully faithful
\item
$L$ commutes with products of fibrant objects.
\item
if $\{X_{i}\rightarrow Y_{i}\}_{i\in\mathcal{I}}$ is a collection of weak equivalences in $\mathpzc{N}$ where each $X_{i}$ is the image of a fibrant object in $\mathpzc{M}$ and each $Y_{i}$ is fibrant, then $\prod_{i\in\mathcal{I}}X_{i}\rightarrow\prod_{i\in\mathcal{I}}Y_{i}$ is an equivalence in $\mathpzc{N}$. 
\end{enumerate}
Then there is a recollement 
\begin{displaymath}
\begin{tikzcd}
\mathbf{K}\arrow[r,bend right=25,shift right=0.2ex]\arrow[r,bend left=25,shift left=0.2ex]  &\mathbf{M} \arrow[l,"\perp" {inner sep=0.3ex,rotate=180},
    "\perp"' {inner sep=0.3ex,rotate=180}] 
\arrow[l,"\perp" {inner sep=0.3ex,rotate=180},
    "\perp"' {inner sep=0.3ex,rotate=180}]\arrow[r,bend right=25,shift right=0.2ex]\arrow[r,bend left=25,shift left=0.2ex]  & 
    \mathbf{N} \arrow[l,"\perp" {inner sep=0.3ex,rotate=180},
    "\perp"' {inner sep=0.3ex,rotate=180}]& \\
    \end{tikzcd}
    \end{displaymath}

In particular if $\mathbf{M}$ is locally finitely presentable then $\mathbf{N}$ is compactly assembled  in the sense of \cite{SAG}  Definition 21.1.2.1/ Proposition D.7.3.1.
\end{lem}

\begin{proof}
The first assumption implies that the Quillen adjunction is in fact a Quillen reflection, so that the adjunction of $(\infty,1)$-categories it presents
$$\adj{\mathbf{L}}{\mathbf{M}}{\mathbf{N}}{\mathbf{R}}$$
realises $\mathbf{N}$ as a reflective subcategory of $\mathbf{M}$. Now $\mathbf{L}$ commutes with all colimits. By \cite{barwick2016note} Lemma 5 and \cite{nlab:adjointtriple} Proposition 2.3. it remains to prove that $\mathbf{L}$ commutes with limits. Since everything is stable, it suffices to prove that it commutes with products. The homotopy product of a collection $\{X_{i}\}$ in a general model category $\mathpzc{M}$ is computed as follows. Pick fibrant resolutions $X_{i}\rightarrow G_{i}$. Then $\prod_{i\in\mathcal{I}}G_{i}$ is the homotopy product. Now let $\{X_{i}\}_{i\in\mathcal{I}}$ be a collection of objects of $\mathpzc{M}$, and let $X_{i}\rightarrow G_{i}$ be a fibrant resolution for each $i\in\mathcal{I}$, so that $\prod_{i\in\mathcal{I}}G_{i}$ is the homotopy product. We have $L(\prod_{i\in\mathcal{I}}G_{i})\cong\prod_{i\in\mathcal{I}}L(G_{i})$. We need to show that $\prod_{i\in\mathcal{I}}L(G_{i})$ computes the homotopy product in $\mathpzc{N}$. Pick fibrant resolutions $L(G_{i})\rightarrow F_{i}$. Now $F_{i}\cong L\circ R(F_{i})$ and $R(F_{i})$ is fibrant in $\mathpzc{M}$. Thus by assumption $\prod_{i\in\mathcal{I}}L(G_{i})\rightarrow\prod_{i\in\mathcal{I}}F_{i}$ is an equivalence, as required.
\end{proof}

\begin{thm}\label{thm:recollementgenexact}
Let $\mathpzc{E}$ be a purely locally $\lambda$-presentable exact category with enough injectives. Let $\mathfrak{W}\subset\mathfrak{W}'$ be full subcategories of $\mathpzc{E}$ which are thick, generating, closed under transfinite extensions, contain all injectives of $\mathpzc{E}$,  and are presentably decosntructible in themselves. Suppose that $\mathpzc{E}$ is weakly $\mathbf{AdMon}_{\mathfrak{W}'}$-elementary. Then $(\mathfrak{W},\mathfrak{W}^{\perp})$ and $(\mathfrak{W}',(\mathfrak{W}')^{\perp})$ are injective cotorsion pairs. If the respective model category structures $\mathpzc{E}_{\mathfrak{W}}$ and $\mathpzc{E}_{\mathfrak{W}'}$ induced by the cotorsion pairs are stable, and homotopy products in $\mathpzc{E}_{\mathfrak{W}}$ and $\mathpzc{E}_{\mathfrak{W}'}$ coincide, then there is a recollement.
\begin{displaymath}
\begin{tikzcd}
\mathbf{K}\arrow[r,bend right=25,shift right=0.2ex]\arrow[r,bend left=25,shift left=0.2ex]  &\mathbf{E}_{\mathfrak{W}} \arrow[l,"\perp" {inner sep=0.3ex,rotate=180},
    "\perp"' {inner sep=0.3ex,rotate=180}] 
\arrow[l,"\perp" {inner sep=0.3ex,rotate=180},
    "\perp"' {inner sep=0.3ex,rotate=180}]\arrow[r,bend right=25,shift right=0.2ex]\arrow[r,bend left=25,shift left=0.2ex]  & 
\mathbf{E}_{\mathfrak{W}'} \arrow[l,"\perp" {inner sep=0.3ex,rotate=180},
    "\perp"' {inner sep=0.3ex,rotate=180}]& \\
    \end{tikzcd}
    \end{displaymath}
    where $\mathbf{E}_{\mathfrak{W}}$ (resp. $\mathbf{E}_{\mathfrak{W}'}$) is the $(\infty,1)$-category presented by $\mathpzc{E}_{\mathfrak{W}}$ (resp. $\mathpzc{E}_{\mathfrak{W}'}$).
\end{thm}

\begin{proof}
Consider the Quillen reflection
$$\adj{Id}{\mathpzc{E}_{\mathfrak{W}}}{\mathpzc{E}_{\mathfrak{W}'}}{Id}$$
which induces an adjucntion of $(\infty,1)$-categories
\begin{displaymath}
\xymatrix{
\adj{L}{\mathbf{E}_{\mathfrak{W}}}{\mathbf{E}_{\mathfrak{W}'}}{i}
}
\end{displaymath}
with $i$ being fully faithful. Now the $(\infty,1)$-categories $\mathbf{E}_{\mathfrak{W}}$ and $\mathbf{E}_{\mathfrak{W}'}$ are both locally presentable by Lemma \ref{cofibgen}. We can immediately apply Lemma \ref{lemrecollemodel}. 
\end{proof}

\begin{cor}
Let $\mathpzc{E}$ be a purely locally finitely presented exact category equipped with the pure exact structure, and consider the category $\mathrm{Ch}(\mathpzc{E})$. Suppose that $\tilde{\mathfrak{W}}$ is a class of complexes which 
%
\begin{enumerate}
\item
contains the pure acyclic complexes
\item
is deconstructible in itself
\item
$\tilde{\mathfrak{W}}$ is closed under products.
\end{enumerate}
Let $\mathpzc{M}$ be the model category whose underlying category is $\mathrm{Ch}(\mathpzc{E})$, equipped with the model structure induced by the injective cotorsion pair $(\tilde{\mathfrak{W}},\tilde{\mathfrak{W}}^{\perp})$. Then the $(\infty,1)$-category $\mathbf{M}$ is compactly assembled, provided it is stable.
\end{cor}

\begin{proof}
Consider the category $\mathrm{Ch}(\mathpzc{E})$ equipped with the pure injective model structure. $\tilde{\mathfrak{W}}$ contains all objects of the form $D^{n}(A)$, so it contains all injectives. Thus $(\tilde{\mathfrak{W}},\tilde{\mathfrak{W}}^{\perp})$ is an injective cotorsion pair. 
\end{proof}


\begin{cor}
Let $\mathpzc{E}$ be a purely locally finitely presentable category, and let $(\mathpzc{E},\mathpzc{Q})$ be an exact category in which filtered colimits are exact. Let $\mathbf{Ch}(\mathpzc{E},\mathpzc{Q})$ be the $(\infty,1)$-category presented by the injective model structure on $\mathrm{Ch}(\mathpzc{E},\mathpzc{Q})$. Then $\mathbf{Ch}(\mathpzc{E},\mathpzc{Q})$ is compactly assembled. 
\end{cor}

Now let $(\mathpzc{E},\mathpzc{Q})$ and $(\mathpzc{E},\mathpzc{Q}')$ be two different purely locally $\lambda$-presentable exact categories with the same underlying category $\mathpzc{E}$ and $\mathpzc{Q}\subset\mathpzc{Q}'$. Suppose both $(\mathpzc{E},\mathpzc{Q})$ and $(|mathpzc{E},\mathpzc{Q}')$ are weakly elementary. Let $\mathpzc{K}$ be a class of objects in $\mathrm{Ch}(\mathpzc{E})$ which is strongly $\lambda$-pure subobject stable in both $\mathrm{Ch}(\mathpzc{E},\mathpzc{Q})$ and  $\mathrm{Ch}(\mathpzc{E},\mathpzc{Q}')$, is closed under transfinite extensions in $\mathrm{Ch}(\mathpzc{E},\mathpzc{Q}')$, contains a generator of $\mathrm{Ch}(\mathpzc{E},\mathpzc{Q})$, and contains all injectives in $\mathrm{Ch}(\mathpzc{E},\mathpzc{Q})$. In particular $(\mathpzc{K},\mathpzc{K}^{\perp})$ is an injective cotorsion pair on $\mathrm{Ch}(\mathpzc{E},\mathpzc{Q})$ and $\mathrm{Ch}(\mathpzc{E},\mathpzc{Q}')$. We also have the injective Hovey triple $(\mathrm{All},\mathfrak{W},\mathrm{dgInj})$ where $\mathfrak{W}$ denotes the class of acyclic complexes in $\mathrm{Ch}(\mathpzc{E},\mathpzc{Q})$, and the injective Hovey triple $(\mathrm{All},\mathfrak{W}',\mathrm{dgInj}')$ where $\mathfrak{W}'$ denotes the class of acyclic objects in $\mathrm{Ch}(\mathpzc{E},\mathpzc{Q}')$, and $\mathrm{dgInj}'$ denotes the $dg$-injective complexes in $\mathrm{Ch}(\mathpzc{E},\mathpzc{Q}')$. Thanks to Corollary \ref{cor:changingtheunderlyingexactstructure} there is a Hovey triple $(\mathrm{All},\mathfrak{W}',(\mathfrak{W}')^{\perp})$ on $\mathrm{Ch}(\mathpzc{E},\mathpzc{Q})$ whose corresponding model structure is left Quillen equivalent, via the identity functor, to the injective model structure on $\mathrm{Ch}(\mathpzc{E},\mathpzc{Q}')$. Thus we have three injective Hovey triples
$$\mathpzc{M}_{1}=(\mathrm{All},\mathfrak{W},\mathrm{dgInj})$$
$$\mathpzc{M}_{2}=(\mathrm{All},\mathpzc{K},\mathpzc{K}^{\perp})$$
$$\mathpzc{M}_{3}=(\mathrm{All},\mathfrak{W}',(\mathfrak{W}')^{\perp})$$
on $\mathrm{Ch}(\mathpzc{E},\mathpzc{Q})$. By \cite{MR3459032} Theorem 4.6 we have the following

\begin{prop}
If $\mathfrak{W}'\cap\mathpzc{K}=\mathfrak{W}$ then there is a recollement of triangulated categories.
\begin{displaymath}
\begin{tikzcd}
\mathpzc{K}^{\perp}\big\slash\sim\arrow[r,bend right=25,shift right=0.2ex]\arrow[r,bend left=25,shift left=0.2ex]  &\mathrm{dgInj}\big\slash\sim\arrow[l,"\perp" {inner sep=0.3ex,rotate=180},
    "\perp"' {inner sep=0.3ex,rotate=180}] 
\arrow[l,"\perp" {inner sep=0.3ex,rotate=180},
    "\perp"' {inner sep=0.3ex,rotate=180}]\arrow[r,bend right=25,shift right=0.2ex]\arrow[r,bend left=25,shift left=0.2ex]  & 
   (\mathfrak{W}')^{\perp}\big\slash\sim \arrow[l,"\perp" {inner sep=0.3ex,rotate=180},
    "\perp"' {inner sep=0.3ex,rotate=180}]& \\
    \end{tikzcd}
    \end{displaymath}
    where $f\sim g$ if and only if $g-f$ factors through an injective object of $\mathrm{Ch}(\mathpzc{E})$.
\end{prop}

In fact let us show the following.

\begin{thm}
If $\mathfrak{W}'\cap\mathpzc{K}=\mathfrak{W}$ then there is a recollement of $(\infty,1)$-categories.
\begin{displaymath}
\begin{tikzcd}
\mathbf{K}(\mathpzc{E})\big\slash\mathpzc{K}\arrow[r,bend right=25,shift right=0.2ex]\arrow[r,bend left=25,shift left=0.2ex]  &\mathbf{Ch}(\mathpzc{E},\mathpzc{Q})\arrow[l,"\perp" {inner sep=0.3ex,rotate=180},
    "\perp"' {inner sep=0.3ex,rotate=180}] 
\arrow[l,"\perp" {inner sep=0.3ex,rotate=180},
    "\perp"' {inner sep=0.3ex,rotate=180}]\arrow[r,bend right=25,shift right=0.2ex]\arrow[r,bend left=25,shift left=0.2ex]  & 
 \mathbf{Ch}(\mathpzc{E},\mathpzc{Q}') \arrow[l,"\perp" {inner sep=0.3ex,rotate=180},
    "\perp"' {inner sep=0.3ex,rotate=180}]& \\
    \end{tikzcd}
    \end{displaymath}
\end{thm}

\begin{proof}
Since $\mathpzc{K}$ is strongly $\lambda$-pure subobject stable, closed under transfinite extensions in $\mathrm{Ch}(\mathpzc{E},\mathpzc{Q}')$, and contains a generator of $\mathrm{Ch}(\mathpzc{E},\mathpzc{Q}')$, and $\mathrm{Ch}(\mathpzc{E},\mathpzc{Q}')$ is weakly $\mathbf{AdMon}_{\mathcal{K}}$-elementary, there is a model structure on $\mathrm{Ch}(\mathpzc{E},\mathpzc{Q}')$ determined by the Hovey triple $(\mathpzc{K},\mathfrak{W}',(\mathpzc{K}\cap\mathfrak{W}')^{\perp}=\mathfrak{W}^{\perp})$. This, in particular, implies that products preserve quasi-isomorphisms in $\mathrm{Ch}(\mathpzc{E},\mathpzc{Q}')$ between objects in $(\mathpzc{K}\cap\mathfrak{W}')^{\perp}=\mathfrak{W}^{\perp}$. Since $\mathfrak{W}^{\perp}$ is the class of fibrant objects for the injective model structure on $\mathrm{Ch}(\mathpzc{E},\mathpzc{Q})$, this completes the proof.
\end{proof}

Here $\mathbf{K}(\mathpzc{E})\big\slash\mathcal{K}$ is just notation for the $(\infty,1)$-category presented by the model structure induced by the Hovey triple $(\mathpzc{K},\mathfrak{W}',\mathfrak{W}^{\perp})$. Its homotopy category is equivalent to the Verdier quotient of the homotopy category $K(\mathpzc{E})$ by the triangualted subcategory generated $\mathcal{K}$. 

\begin{cor}
Let $\mathpzc{E}$ be a purely locally finitely presentable category, and let $(\mathpzc{E},\mathpzc{Q})$ be an exact category in which filtered colimits are exact. Let $\mathbf{Ch}(\mathpzc{E},\mathpzc{Q})$ be the $(\infty,1)$-category presented by the injective model structure on $\mathrm{Ch}(\mathpzc{E},\mathpzc{Q})$. Then $\mathbf{Ch}(\mathpzc{E},\mathpzc{Q})$ is compactly assembled. 
\end{cor}

\section{The Flat and $K$-Flat Model Structures}\label{sec:flatKflat}

In this section we apply the results from the previous section to construct flat model structures on certain closed symmetric monoidal locally presentable exact categories. We also investgiate $K$-flat objects, and explain how recent results of \cite{estrada2023k} generalise immediately to exact categories.
We fix a purely $\lambda$-acessible closed symmetric monoidal exact category $\mathpzc{E}$.

\subsection{The Flat Model Structure}

\begin{thm}
Let $\mathpzc{E}$ be a purely $\lambda$-accessible closed symmetric monoidal exact category. Then the class $\mathcal{F}$ of flat objects of $\mathpzc{E}$ is $\lambda$-pure subobject stable. It is strongly $\lambda$-pure subobject stable if $\mathpzc{E}$ is weakly elementary.
\end{thm}

\begin{proof}
Let
$$0\rightarrow P\rightarrow F\rightarrow F\big\slash P\rightarrow0$$

$$0\rightarrow X\rightarrow Y\rightarrow Z\rightarrow 0$$ 

be exact sequences with the first being $\lambda$-pure exact and $F$ being flat. Consider the square

\begin{displaymath}
\xymatrix{
& 0\ar[d] & 0\ar[d] & 0\ar[d] & \\
0\ar[r] & X\otimes P\ar[r]\ar[d] & Y\otimes P\ar[r]\ar[d] & Z\otimes P\ar[r]\ar[d] & 0\\
0\ar[r] & X\otimes F\ar[r]\ar[d] & Y\otimes F\ar[r]\ar[d] & Z\otimes F\ar[r]\ar[d] & 0\\
0\ar[r] & X\otimes P\big\slash F\ar[r]\ar[d] & Y\otimes P\big\slash F\ar[d]\ar[r] & Z\otimes P\big\slash F\ar[d]\ar[r] & 0\\
& 0 &0 &0 &
}
\end{displaymath}
Since the sequence $0\rightarrow P\rightarrow F\rightarrow F\big\slash P\rightarrow0$ is $\lambda$-pure exact it is in particular $\otimes$-pure exact. Thus all columns are short exact sequences. The middle row is a short exact sequence since $F$ is flat. The map $X\otimes P\rightarrow Y\otimes P$ is then an admissible monomorphism by the obscure axiom. Thus the top row is exact, and by the $3\times 3$ Lemma the bottom row is also exact.

Now suppose that $\mathpzc{E}$ is weakly elementary. By Remark \ref{rem:pushpull} it suffies to check that transfinite colimits of flat objets are flat. Let 
$$F_{0}\rightarrow\ldots\rightarrow F_{\alpha}\rightarrow F_{\alpha+1}\rightarrow\ldots$$
be such a sequene, and let $W$ be an acycylic complex. Each term in the sequence
$$F_{0}\otimes W\rightarrow\ldots\rightarrow F_{\alpha}\otimes W\rightarrow F_{\alpha+1}\otimes W\rightarrow\ldots$$
is an acyclic complex. Since $\mathpzc{E}$ is weakly elementary, the colimit $\colim (F_{\alpha}\otimes W)\cong(\colim F_{\alpha})\otimes W$ is acyclic, so that $\colim F_{\alpha}$ is flat. 
\end{proof}

\begin{cor}
Let $\mathpzc{E}$ be a purely $\lambda$-accessible closed symmetric monoidal exact category which is weakly elementary and has a generator. Then every object of $\mathpzc{E}$ has a flat precover. If $\mathpzc{E}$ has enough flat objects then every object has a special flat precover and a special cotorsion envelope
\end{cor}

\begin{cor}\label{cor:flatmodelstructureexists}
Let $\mathpzc{E}$ be a purely $\lambda$-accessible closed symmetric monoidal exact category which is weakly elementary, has a generator and has enough flat objects. Then the flat cotorsion pair induces model structures on $\mathrm{Ch}_{\ge0}(\mathpzc{E}),\mathrm{Ch}(\mathpzc{E})$, and $\mathrm{Ch}_{\le0}(\mathpzc{E})$. In all cases the model structure is combinatorial. If the monoidal unit $k$ is flat then the model structures on $\mathrm{Ch}(\mathpzc{E})$ and $\mathrm{Ch}_{\ge0}(\mathpzc{E})$ are monoidal, and satisfy the monoid axiom. 
\end{cor}

\subsection{The $K$-Flat Model Structure}

As pointed out in \cite{estrada2023k}, $K$-flatness is a more natural notion for the purposes of homotopical/ homological algebra than flatness. In particular it is something which can be formulated in any model category.

\subsubsection{The $K$-Flat Cotorsion Pair}

Let us now explain how the work of \cite{estrada2023k} is easily generalised. Let $\mathpzc{E}$ be a purely $\lambda$-accessible closed symmetric monoidal model category and let $\mathfrak{W}$ and $\mathcal{S}$ be classes of objects in $\mathpzc{E}$. Suppose that $\mathfrak{W}$ is closed under transfinite extensions.

\begin{defn}
An object $X$ of $\mathpzc{E}$ is said to be $\mathcal{S}$-\textit{acyclic relative to }$\mathfrak{W}$ if $C\otimes X$ is in $\mathfrak{W}$ for any $C\in\mathcal{S}$. The class of all $\mathcal{S}$-acyclic objects is denoted ${}_{\mathcal{S}}\mathfrak{W}$.
\end{defn}

\begin{defn}
We define the class of $K$-\textit{flat objects} (relative to $\mathfrak{W}$), $K\mathcal{F}$, by $K\mathcal{F}\defeq{}_{\mathfrak{W}}\mathfrak{W}$. 
\end{defn}

\begin{prop}
If $\mathfrak{W}$ is $\lambda$-pure subobject stable then so is ${}_{\mathcal{S}}\mathfrak{W}$. If $\mathfrak{W}$ is strongly $\lambda$-pure subobject stable, the $\mathcal{S}$-pure exact structure is weakly elementary, and any short exact sequence with cokernel in $\mathfrak{W}$ is $\mathcal{S}$-pure then ${}_{\mathcal{S}}\mathfrak{W}$ is also strongly $\lambda$-pure subobject stable.
\end{prop}

\begin{proof}
This follows immediately from the fact that if $X\rightarrow Y$ is a $\lambda$-pure monomorphism then so is $C\otimes X\rightarrow C\otimes Y$ for any object $C$ of $\mathpzc{E}$. 

First consider a commutative diagram of exact sequences
\begin{displaymath}
\xymatrix{
K\ar[r]\ar[d] & N\ar[d]^{f}\ar[r]^{g} & A\ar[d]^{i}\\
K\ar[r] & M\ar[r]^{h} & B
}
\end{displaymath}
where the right-hand square is a pullback, $A,B\in{}_{\mathcal{S}}\mathfrak{W}$, $i$ is a $\lambda$-pure monomorphism, and $h$ is an admissible epimorphism. For any $S\in\mathcal{S}$
\begin{displaymath}
\xymatrix{
S\otimes K\ar[r]\ar[d] & S\otimes N\ar[d]^{S\otimes f}\ar[r]^{g} & S\otimes A\ar[d]^{S\otimes i}\\
S\otimes K\ar[r] & S\otimes M\ar[r]^{S\otimes h} & S\otimes B
}
\end{displaymath}
has the same properties. In particular for each $S\in\mathcal{S}$, $S\otimes f$ is almost $(\mathfrak{W},\lambda)$-pure. Now let
\begin{displaymath}
\xymatrix{
M_{0}\ar[r]\ar[d] & M_{\alpha}\ar[d]\ar[r] & M_{\alpha'}\ar[r]\ar[d] & \ldots\\
M\ar[d]\ar[r] & M\ar[r]\ar[d] & M\ar[r]\ar[d] & \ldots\\
M\big\slash M_{0}\ar[r] &M\big\slash M_{\alpha}\ar[r] & M\big\slash M_{\alpha'}\ar[r] &\ldots
}
\end{displaymath}
be  a $\Gamma$-indexed transfinite sequence where for each successor ordinal $\alpha+1\in \Gamma$, $M_{\alpha+1}\rightarrow M$ is an almost $({}_{\mathcal{S}}\mathfrak{W},\lambda)$-pure monomorphism, and $f$ is the colimit of the maps $M_{\alpha}\rightarrow M$ in $\mathrm{Mor}(\mathpzc{E})$. In particular each column is exact in the $\mathcal{S}$-pure exact structure. Since this exact structure is weakly elementary,
$$0\rightarrow\colim_{\alpha}M_{\alpha}\rightarrow M\rightarrow \colim_{\alpha}M\big\slash M_{\alpha}\rightarrow 0$$
is $\mathpzc{S}$-pure exact, Moreover for each $S\in\mathcal{S}$,
\begin{displaymath}
\xymatrix{
S\otimes M_{0}\ar[r]\ar[d] & S\otimes M_{\alpha}\ar[d]\ar[r] & S\otimes M_{\alpha'}\ar[r]\ar[d] & \ldots\\
S\otimes M\ar[d]\ar[r] & S\otimes M\ar[r]\ar[d] & S\otimes M\ar[r]\ar[d] & \ldots\\
S\otimes (M\big\slash M_{0})\ar[r] & S\otimes M\big\slash S\otimes M_{\alpha}\ar[r] & S\otimes(M\big\slash M_{\alpha'})\ar[r] &\ldots
}
\end{displaymath}
is alsoa $\Gamma$-indexed transfinite sequence where for each successor ordinal $\alpha+1\in \Gamma$, $S\otimes M_{\alpha+1}\rightarrow S\otimes M$ is an almost $(\mathpzc{W},\lambda)$-pure monomorphism. Thus
$$0\rightarrow\colim_{\alpha}S\otimes M_{\alpha}\rightarrow S\otimes M\rightarrow \colim_{\alpha}S\otimes (M\big\slash M_{\alpha})\rightarrow 0$$
is a short exact sequence with $ \colim_{\alpha}S\otimes (M\big\slash M_{\alpha})\cong S\otimes\colim_{\alpha}M\big\slash M_{\alpha}$ in $\mathcal{W}$, as required. 
\end{proof}

\begin{prop}\label{prop:sWext}
Suppose that
\begin{enumerate}
\item
 admissible monomorphisms with cokernel in $\mathfrak{W}$ are $\mathcal{S}$-pure
 \item
 $\mathpzc{E}$ is weakly $\mathbf{AdMon}_{\mathfrak{W}}$-elementary
 \item
 $\mathfrak{W}$ is closed under transfinite extensions by dmissible monomorphisms.
 \end{enumerate}
 Then ${}_{\mathcal{S}}\mathfrak{W}$ is closed under transfinite extensions by admissible monomorphisms.
\end{prop}

\begin{proof}
 Let
 $$X_{0}\rightarrow X_{1}\rightarrow\ldots\rightarrow X_{\alpha}\rightarrow\ldots$$
be a $\lambda$-indexed transfinite sequence of admissible monomorphisms with cokernel in ${}_{\mathcal{S}}\mathfrak{W}$. In particular they are $\mathcal{S}$-pure Let $S$ be an object of $\mathcal{S}$. Then
$$S\otimes X_{0}\rightarrow S\otimes X_{1}\rightarrow\ldots\rightarrow S\otimes X_{\alpha}\rightarrow\ldots $$
is a transfinite sequence with $\mathrm{coker}(S\otimes X_{\alpha}\rightarrow S\otimes X_{\alpha+1})\cong S\otimes\mathrm{coker}(X_{\alpha}\rightarrow X_{\alpha+1})\in\mathfrak{W}$. Thus $S\otimes X_{0}\rightarrow\colim_{\alpha}S\otimes X_{\alpha}\cong S\otimes\colim_{\alpha}X_{\alpha}$ is also an admissible monomorphism with cokernel $S\otimes\mathrm{coker}(X_{0}\rightarrow\colim_{\alpha}X_{\alpha})$ in $\mathfrak{W}$. 
\end{proof}

\begin{cor}
Suppose that
\begin{enumerate}
 \item
 $\mathfrak{W}$ is strongly $\lambda$-pure subobject stable.
\item
 admissible monomorphisms with cokernel in $\mathfrak{W}$ are $\mathcal{S}$-pure
 \item
 $\mathpzc{E}$ is weakly $\mathbf{AdMon}_{\mathfrak{W}}$-elementary
 \item
 $\mathfrak{W}$ is closed under transfinite extensions by admissible monomorphisms.
 \item
  $\mathpzc{E}$ contains enough injectives
  \item
  ${}_{\mathcal{S}}\mathfrak{W}$ contains a generator.
  \item
  ${}_{\mathcal{S}}\mathfrak{W}$ contains all injectives.
 \end{enumerate}
Then $({}_{\mathcal{S}}\mathfrak{W},{}_{\mathcal{S}}\mathfrak{W}^{\perp})$ is an injective cotorsion pair.
\end{cor}


Now we specialise to categories of complexes in purely $\lambda$-accessible exact categories. Let $(\mathpzc{E},\mathpzc{P})$ and $(\mathpzc{E},\mathpzc{Q})$ be exact structures on $\mathpzc{E}$, with $\mathpzc{P}\subseteq\mathpzc{Q}$, and such that both exact structures are purely $\lambda$-accessible. Let $\mathfrak{W}$ denote the class of trivial objects in $\mathrm{Ch}(\mathpzc{E},\mathpzc{Q})$.


\begin{prop}[c.f. \cite{estrada2023k} Theorem 3.4]
Let $\mathcal{S}\subseteq\mathpzc{E}$. Then
\begin{enumerate}
\item
If ${}_{\mathcal{S}}\mathfrak{W}$ is strongly $\lambda$-pure subobject stable, is closed under transfinite extensions, and $(\mathrm{Ch}(\mathpzc{E}),\mathcal{P})$ is weakly $\mathbf{AdMon}_{{}_{\mathcal{S}}\mathfrak{W}}$-elementary then
$({}_{\mathcal{S}}\mathfrak{W},{}_{\mathcal{S}}\mathfrak{W}^{\perp_{\mathcal{P}}})$ is an injective cotorsion pair on $\mathrm{Ch}(\mathpzc{E},\mathcal{P})$.
\item
${}_{\mathcal{S}}\mathfrak{W}^{\perp_{\mathcal{P}}}$ consists of those complexes $X$ such that each $X_{n}$ is $\mathcal{P}$-injective, and $E\rightarrow X$ is null-homotopic whenever $E\in{}_{\mathcal{S}}\mathfrak{W}$. 
\end{enumerate}
\end{prop}

\begin{proof}
Note that ${}_{\mathcal{S}}\mathfrak{W}$  contains all objects of the form $D^{n}(E)$ for $E$ an object of $\mathpzc{E}$, and so generates $\mathrm{Ch}(\mathpzc{E})$. In particular it contains all injectives. Thus $({}_{\mathcal{S}}\mathfrak{W},{}_{\mathcal{S}}\mathfrak{W}^{\perp,\mathcal{P}})$ is an injective cotorsion pair. The second claim is an immediate consequence of Proposition \ref{prop:classifyingorth}.
\end{proof}

\begin{example}
If $\mathpzc{E}$ is weakly elementary, and any short exact sequence with cokernel in ${}_{\mathcal{S}}\mathfrak{W}$ is $\mathcal{S}$-pure, then ${}_{\mathcal{S}}\mathfrak{W}$ is strongly $\lambda$-pure subobject stable, and is closed under transfinite extensions. Moreover $\mathpzc{E}$ is trivially weakly $\mathbf{AdMon}_{{}_{\mathcal{S}}\mathfrak{W}}$-elementary. Thus $({}_{\mathcal{S}}\mathfrak{W},{}_{\mathcal{S}}\mathfrak{W}^{\perp_{\mathcal{P}}})$ is an injective cotorsion pair on $\mathrm{Ch}(\mathpzc{E},\mathcal{P})$.
\end{example}

\begin{rem}
\begin{enumerate}
\item
If $(\mathpzc{E},\mathpzc{Q})$ is weakly $\mathbf{PureMon}_{\mathcal{T}_{\otimes}}$-elementary then the $\mathcal{T}$-pure exact category $(\mathpzc{E},\mathcal{P})$ is weakly $\mathbf{AdMon}$-elementary. The proof is similar to Proposition \ref{prop:sWext}. Indeed let
$$X_{0}\rightarrow X_{1}\rightarrow\ldots\rightarrow X_{\alpha}\rightarrow\ldots$$
be a $\lambda$-indexed transfinite sequence of $\mathcal{T}$-pure monomorphisms. Let $T$ be an object of $\mathcal{T}$. Then
$$T\otimes X_{0}\rightarrow T\otimes X_{1}\rightarrow\ldots\rightarrow T\otimes X_{\alpha}\rightarrow\ldots $$
is a transfinite sequences of admissible monomorphisms. Hence $T\otimes X_{0}\rightarrow\colim_{\alpha}T\otimes X_{\alpha}\cong T\otimes\colim_{\alpha}X_{\alpha}$ is also an admissible monomorphism. Hence $X_{0}\rightarrow\colim X_{\alpha}$ is $\mathcal{T}$-pure. 
\item
Similarly if $(\mathpzc{E},\mathpzc{Q})$ is weakly elementary then so is the $\mathcal{T}$-pure exact structure.
\end{enumerate}
\end{rem}

\begin{cor}\label{cor:pureKflat}
Fix a class of objects $\mathcal{T}$ in $(\mathpzc{E},\mathcal{Q})$ containing  $\mathcal{S}$, and let $\mathcal{P}$ denote the $\mathcal{T}$-pure exact structure. If $(\mathpzc{E},\mathcal{P})$ is weakly elementary then in particular the class ${}_{\mathcal{S}}\mathfrak{W}$ is closed under transfinite extensions and $({}_{\mathcal{S}}\mathfrak{W},{}_{\mathcal{S}}\mathfrak{W}^{\perp_{\mathcal{P}}})$ is an injective cotorsion pair on $\mathrm{Ch}(\mathpzc{E},\mathcal{P})$.
\end{cor}

Let us now fix $\mathpzc{Q}$, and set $\mathpzc{P}$ to be the $(\otimes,\mathcal{Q})$-pure exact structure. From now on we shall suppress $\mathpzc{Q}$. For $\mathcal{S}=\mathfrak{W}$, the class of acyclic complexes in $\mathrm{Ch}(\mathpzc{E})$ then we write $K\mathcal{F}\defeq{}_{\mathfrak{W}}\mathfrak{W}$ for the class of $K$-\textit{flat} complexes. We recover a generalisation of Theorem 4.2 from \cite{estrada2023k}.

\begin{prop}\label{prop:quotientKflat}
Let $\mathpzc{E}$ be a purely locally $\lambda$-presentable closed symmetric monoidal exact category which is wekly elementary. Then the cotorsion pair $(K\mathcal{F},(K\mathcal{F})^{\perp}_{\otimes})$ defines a stable model category structure on $\mathrm{Ch}(\mathpzc{E}_{\otimes})$ whose homotopy category is equivalent to the quotient category $K(\mathpzc{E})\big\slash K\mathcal{F}$.
\end{prop}

\begin{proof}
By Proposition \ref{cor:pureKflat} $(K\mathcal{F},(K\mathcal{F})^{\perp}_{\otimes})$ is an injective cotorsion pair. To prove that the homotopy category is $K(\mathpzc{E})\big\slash K\mathcal{F}$ it suffices to check that $K\mathcal{F}$ contains $\otimes$-pure acyclic complexes. The flat cotorsion pair on $\mathpzc{E}_{\otimes}$ coincides with the injective cotorsion pair, since all objects are flat. In particular the injective cotorsion pair is monoidally $dg$-compatible. Thus the tensor product of an $\otimes$-pure acyclic complex with any complex is $\otimes$-pure acyclic. Finally we prove that the model category is stable. It is the left Bousfield localisation of the ($\otimes$-pure) injective model structure on $\mathrm{Ch}(\mathpzc{E}_{\otimes})$ at $\otimes$-pure monomorphisms with $K$-flat cokernel. The collection of $K$-flat objects is closed under the shift functor, so the localisation is also stable by \cite{MR3137847} Proposition 3.6.
\end{proof}

\subsubsection{The $K$-Flat Model Structure and the Recollement}

Now suppose that $\mathpzc{E}$ has enough flat objects. The following can be proven exactly as in \cite{estrada2023k} Lemma 5.1.

\begin{lem}[\cite{estrada2023k} Lemma 5.1]
If $(\mathpzc{E},\mathpzc{Q})$ is a weakly idempotent complete closed symmetric monoidal exact category with enough flat objects then every complex in $K\mathcal{F}^{\perp_{\otimes}}$ must be acyclic in $(\mathpzc{E},\mathpzc{Q})$.
\end{lem}
%

We can now state the generalisation of \cite{estrada2023k} Proposition 5.2, Theorem 5.3. The proofs are essentially the same, with some small additional technicalities. 

\begin{thm}[\cite{estrada2023k} Proposition 5.2, Theorem 5.3]\label{thm:Kflat}
Let $\mathpzc{E}$ be a purely locally $\lambda$-presentable closed symmetric monoidal exact category which is weakly elementary, has enough flat objects, and has a flat tensor unit. 
\begin{enumerate}
\item
Let $\mathfrak{W}$ denote the class of quasi-isomorphisms in $\mathrm{Ch}(\mathpzc{E},\mathpzc{Q})$. Then $(K\mathcal{F},\mathfrak{W},(K\mathcal{F}\cap\mathfrak{W})^{\perp_{\otimes}})$ is a Hovey triple on $\mathrm{Ch}(\mathpzc{E}_{\otimes})$. The induced model structure on $\mathrm{Ch}(\mathpzc{E}_{\otimes})$ is monoidal and satisfies the monoid axiom.
\item
$(\widetilde{dg\mathrm{Flat}},\mathcal{W},\widetilde{\mathrm{Flat}}^{\perp})$ is a Hovey triple on $\mathrm{Ch}(\mathpzc{E})$. The induced model structure on $\mathrm{Ch}(\mathpzc{E})$ is monoidal and satisfies the monoid axiom. Moreover it is left Quillen equivalent to the one from Part i) through the identity functor.
\item
$X$ is acyclic and $K$-flat if and only if $X$ is $\otimes$-pure acyclic. In particular there is a recollement.
\begin{displaymath}
\begin{tikzcd}
\mathbf{K}(\mathpzc{E})\big\slash K\mathcal{F}\arrow[r,bend right=25,shift right=0.2ex]\arrow[r,bend left=25,shift left=0.2ex]  &\mathbf{Ch}(\mathpzc{E}_{\otimes})\arrow[l,"\perp" {inner sep=0.3ex,rotate=180},
    "\perp"' {inner sep=0.3ex,rotate=180}] 
\arrow[l,"\perp" {inner sep=0.3ex,rotate=180},
    "\perp"' {inner sep=0.3ex,rotate=180}]\arrow[r,bend right=25,shift right=0.2ex]\arrow[r,bend left=25,shift left=0.2ex]  & 
    \mathbf{Ch}(\mathpzc{E}) \arrow[l,"\perp" {inner sep=0.3ex,rotate=180},
    "\perp"' {inner sep=0.3ex,rotate=180}]& \\
    \end{tikzcd}
    \end{displaymath}
\end{enumerate}
\end{thm}

\begin{proof}
\begin{enumerate}
\item
Equip $\mathrm{Ch}(\mathpzc{E}_{\otimes})$ with the injective model structure. $K\mathcal{F}\cap\mathfrak{W}$ contains all objects of the form $D^{n}(G)$ so in particular it contains a generator. They are also both strongly $\lambda$-pure subobject stable and hence deconstructible in themselves. Finally $K\mathcal{F}$ is closed under transfinite extensions whenever $\mathpzc{E}$ is weakly $\mathbf{PureMon}_{\otimes}$-elementary. Indeed in this case if $K=\colim K_{i}$ is an $\otimes$-pure transfinite extension of $K$-flat objects, and $W$ is acyclic, then $K\otimes W\cong\colim K_{i}\otimes W$ is an $\otimes$-pure transfinite extension of acyclic complexes, which is acyclic by assumption. Now cofibrations are pure by assumption. Clearly the tensor product of two $K$-flat objects is $K$-flat. Moreover by definition the tensor product of an acyclic object with a $K$-flat object is acyclic. That the model structure is monoidal and satisfies the monoid axioms follows from Theorem \ref{exactmonoidal}.
\item
The existence of the model structure is Corollary \ref{cor:flatmodelstructureexists}. The claim regarding the Quillen adjunction follows from the fact that admissible monomorphisms with flat cokernel are $\otimes$-pure, and $dg\mathrm{Flat}$-complexes are $K$-flat.
\item
Suppose that $X$ is pure-acyclic. By the proof of Proposition \ref{prop:quotientKflat} any such complex is $K$-flat, and clearly acyclic. The converse proceeds exactly as in \cite{estrada2023k} Lemma 5.1. Now for the recollement just use Theorem \ref{thm:recollementgenexact}
\end{enumerate}
\end{proof}

\section{Homotopical Algebra}\label{sec:homotopicalalgebra}

Locally presentable exact categories with exact filtered colimits and enough flat objects provide very rich settings for homotopical algebra, as we explain in this section. In particular we show that, at least when enriched over $\mathbb{Q}$, they define HA contexts. Moreover we show that powerful Koszul duality theorems hold in remarkable generarility.

\subsection{Model Structures on Categories of Algebras Over Operads}

We begin by establishing the existence of model structures on algebras over operads. Much of this section conists of recollections from \cite{kelly2016homotopy} Section 6.3.1.

\subsubsection{Existence of Transferred Model Structures}

Let section $\mathpzc{M}$ will be a combinatorial model category, which is also a symmetric monoidal category and which satisfies the monoid axiom. Let $\mathfrak{P}$ be an operad in $\mathpzc{M}$ (either symmetric or non-symmetric), and consider the free-forgetful adjunction

$$\adj{\textrm{Free}_{\mathfrak{P}}(-)}{\mathpzc{M}}{\mathrm{Alg}_{\mathfrak{P}}(\mathpzc{M})}{|-|_{\mathfrak{P}}}$$

 Recall that the \textit{transferred model structure} on $\mathrm{Alg}_{\mathfrak{P}}(\mathpzc{M})$, if it exists, is the one for which weak equivalences (resp. fibrations) are maps $f:A\rightarrow B$ of algebras such that $|f|_{\mathfrak{P}}$ is a weak equivalence (resp. fibration).

\begin{defn}
An operad $\mathfrak{P}$ is said to be \textit{admissible} if the transferred model structure exists on $\mathrm{Alg}_{\mathfrak{P}}$. 
\end{defn}

By \cite{harper2010homotopy} Proposition 7.6 for $X$ a $\mathfrak{P}$-algebra, there is a $\Sigma$-module $\mathfrak{P}_{X}$ in $\mathpzc{M}$ such that for any $\Sigma$-module $Y$ in $\mathpzc{M}$, there is an isomorphism, natural in $X$ and $Y$,
 $$X\coprod(\mathfrak{P}\circ Y)\cong\mathfrak{P}_{X}\circ Y$$

  As in \cite{harper2010homotopy} Definition 7.31, for $s:A\rightarrow B$ a map in $\mathpzc{Gr}_{\mathbb{N}_{0}}(\mathpzc{C})$, we define $Q^{t}_{q}(s)$ for $t\ge 1$ and $0\le q\le t$ as follows. $Q_{0}^{t}(s)\defeq A^{\otimes t}$, $Q^{t}_{t}(s)\defeq B^{\otimes t}$ and for $0<q<t$, $Q^{t}_{q}(s)$ is defined by the pushout. 
 \begin{displaymath}
 \xymatrix{
 (X^{\otimes(t-q)}\otimes Q^{q}_{q-1}(s))^{\oplus\binom{t}{q}}\ar[d]\ar[r] & Q_{q-1}^{t}(s)\ar[d]\\
  (X^{\otimes(t-q)}\otimes B^{\otimes q})^{\oplus\binom{t}{q}}\ar[r] & Q^{t}_{q}(s)
 }
 \end{displaymath}
 where the top map is the obvious projection, and the left-hand map is induced by natural map $Q^{q}_{q-1}(s)\rightarrow B^{\otimes q}$. 
 \begin{prop}[\cite{harper2010homotopy}  Proposition 7.32]\label{prop:pushoutalg}
 Let $s:A\rightarrow B$ be a map in $\mathpzc{M}$, and let $X$ be a $\mathfrak{P}$-algebra. Consider a pushout diagram
 \begin{displaymath}
 \xymatrix{
 \mathfrak{P}(A)\ar[d]^{\mathfrak{P}(s)}\ar[r] & X\ar[d]\\
 \mathfrak{P}(B)\ar[r] & P
 }
 \end{displaymath}
 Then $P$ is naturally isomorphic to a filtered colimit
 $$P\cong\textrm{lim}_{\rightarrow_{n}}X_{n}$$
 where $X_{0}=X$, and for $n\ge 1$ the map $X_{n-1}\rightarrow X_{n}$ is given by the pushout diagram in $\mathpzc{M}$
 \begin{displaymath}
 \xymatrix{
 \mathfrak{P}_{X}(n)\otimes Q^{n}_{n-1}(s)\ar[d]\ar[r] & X_{n-1}\ar[d]\\
 \mathfrak{P}_{X}(n)\otimes B^{\otimes n}\ar[r] & X_{n}
 }
 \end{displaymath}
 \end{prop}
 
 \begin{prop}
 Let $\mathcal{S}$ be a class of maps closed under pushout--products. Any map of the form $Q^{t}_{p}(s)\rightarrow Q^{t}_{q}(s)$ for $s:X\rightarrow Y\in\mathcal{S}$ is a pushout of a map of the form $X\otimes s;$ for $X\in\mathpzc{M}$ and $s\in\mathcal{S}$.
 \end{prop}
 
 \begin{proof}
Using an inductive argument, the proof of \cite{harper2010homotopy} Proposition 7.23 shows that for any $t$ the map $Q^{t}_{t-1}(s)\rightarrow Y^{\otimes t}$ is in $\mathcal{S}$. Another inductive argument also proves this for $Q^{t}_{q-1}(s)\rightarrow Q^{t}_{q}(s)$ for $q<t$. 
 \end{proof}
 
 This immediately implies the following.
 
 \begin{cor}\label{cor:modelstruturealg}
If $\mathpzc{M}$ satisfies the pp-monoid axiom, then any operad is admissible. If $\mathpzc{M}$ is enriched over $\mathbb{Q}$, then any symmetric operad is also admissible. 
 \end{cor}
 
 \begin{proof}
Consider a pushout
\begin{displaymath}
 \xymatrix{
 \mathfrak{P}(A)\ar[d]^{\mathfrak{P}(s)}\ar[r] & X\ar[d]\\
 \mathfrak{P}(B)\ar[r] & P
 }
 \end{displaymath}
 with $s$ an acyclic cofibration. The map $X\rightarrow P$ is a transfinite composition of iterated pushout-products of acyclic cofibrations. Thus any transfinite composition of pushouts along maps of the form $\mathfrak{P}(s)$ with $s$ an acyclic cofibration is also a transfinite composition of iterated pushout-products of acyclic cofibrations. Thus it is an equivalence as required. If $\mathpzc{M}$ is enriched over $\mathbb{Q}$, then any operad $\mathfrak{P}$ is a retract of the free symmetric operad $\Sigma\otimes\mathfrak{P}_{ns}$ on the underlying non-symmetric operad $\mathfrak{P}_{ns}$ of $\mathfrak{P}$. $\Sigma\otimes\mathfrak{P}_{ns}$ is admissible, and it follows that $\mathfrak{P}$ is admissible. 
 \end{proof}
 

\subsection{HA Contexts}

In this subsection we show that locally $\lambda$-presentable closed symmetric monoidal exact categories enriched over $\mathbb{Q}$ in which transfinite compositions of $\lambda$-pure monomorphisms are admissible, and which have enough flat objects, naturally present HA contexts. HA contexts were introduced in \cite{toen2004homotopical} as convenient abstract frameworks for homotopical algebra and, ultimately, derived geometry. We recall the truncated definition of a HA context from \cite{kelly2016homotopy} Section 6.4.2. For the category $\mathsf{C}_{0}$ in \cite{toen2004homotopical} we always take $\mathsf{C}=\mathsf{C}_{0}$.

\begin{defn}\label{Defn:HA context}
Let $\mathpzc{M}$ be a combinatorial symmetric monoidal model category. We say that $\mathpzc{M}$ is an \textit{homotopical algebra context} (or HA context) if for any $A\in\mathrm{Alg}_{\mathfrak{Comm}}(\mathpzc{M})$.
\begin{enumerate}
\item
The model category $\mathpzc{M}$ is proper, pointed and for any two objects $X$ and $Y$ in $\mathpzc{M}$ the natural morphisms
$$QX\coprod QY\rightarrow X\coprod Y\rightarrow RX\times RY$$
are equivalences.
\item
$Ho(\mathpzc{M})$ is an additive category.
\item
With the transferred model structure and monoidal structure $-\otimes_{A}-$, the category ${}_{A}\mathrm{Mod}$ is a combinatorial, proper, symmetric monoidal model category.
\item
For any cofibrant object $M\in{}_{A}\mathrm{Mod}$ the functor
$$-\otimes_{A}M:{}_{A}\mathrm{Mod}\rightarrow{}_{A}\mathrm{Mod}$$
preserves equivalences.
\item
With the transferred model structures $\mathrm{Alg}_{\mathfrak{Comm}}({}_{A}\mathrm{Mod})$ and $\mathrm{Alg}_{\mathfrak{Comm}_{nu}}({}_{A}\mathrm{Mod})$ are combinatorial proper model categories.
\item
If $B$ is cofibrant in $\mathrm{Alg}_{\mathfrak{Comm}}({}_{A}\mathrm{Mod})$ then the functor
$$B\otimes_{A}-:{}_{A}\mathrm{Mod}\rightarrow{}_{B}\mathrm{Mod}$$
preserves equivalences.
\end{enumerate}
\end{defn}

\subsubsection{Left Properness for Commutative Monoids}

Often the most difficult axiom to establish is left properness of $\mathrm{Alg}_{\mathfrak{Comm}}({}_{A}\mathrm{Mod})$. 

\begin{thm}[\cite{white2017model} Theorem 4.17]\label{thm:Commadm}
Let $\mathpzc{M}$ be a monoidal model category. Suppose
\begin{enumerate}
\item
cofibrations are $\mathpzc{M}$-monoidally left proper
\item
$\mathpzc{M}$ satisfies strong commutative monoid axiom and the monoid axiom
\item
cofibrant objects are $K$-flat
\item
if $f:X\rightarrow Y$ is a generating cofibration then $X$ is cofibrant.
\end{enumerate}
Then $\mathfrak{Comm}$ is admissible and the category $\mathrm{Alg}_{\mathfrak{Comm}}(\mathpzc{C})$ is left proper.
\end{thm}

A more general theorem which is a rather straightforward generalisation of the above one is stated as theorem 6.3.31 in \cite{kelly2016homotopy}.

\subsubsection{HA Contexts from (Possibly) Non-Commutative Algebras}

Before proceeding to establish when monoidal model categories arising from exact categories are in fact HA contexts, let us prove a general result which will produce several examples. Fix a monoidal model category $\mathpzc{M}$. Let $H\in\mathrm{Alg}_{\mathfrak{Ass}}(\mathpzc{M})$ be a unital associative algebra. We suppose that there exists a closed monoidal structure $\tilde{\otimes}$ on ${}_{H}\mathrm{Mod}(\mathpzc{M})$, and a strong closed monoidal structure on the forgetful functor
$${}_{H}\mathrm{Mod}(\mathpzc{M})\rightarrow\mathpzc{M}$$
For example, $H$ could be a cocommutative bialgebra. Later we will consider modules over algebras of differential operators. Note that if $A\in\mathrm{Alg}_{\mathfrak{Comm}}({}_{H}\mathrm{Mod}(\mathpzc{M}))$, then there is a natural way of producing a closed monoidal structure on ${}_{A}\mathrm{Mod}({}_{H}\mathrm{Mod}(\mathpzc{M}))$, and a strong closed monoidal structure on the forgetful functor
$${}_{A}\mathrm{Mod}({}_{H}\mathrm{Mod}(\mathpzc{M}))\rightarrow{}_{A}\mathrm{Mod}(\mathpzc{M})$$
%

\begin{prop}
Let $\mathpzc{M}$ be a monoidal model category, and let $A\in\mathrm{Alg}_{\mathfrak{Ass}}(\mathpzc{M})$ be an associative monoid. Suppose the transferred model structure exists on ${}_{A}\mathrm{Mod}(\mathpzc{M})$. If $\mathpzc{M}$
\begin{enumerate}
\item
is pointed
\item
is such that for any two objects $X$ and $Y$ of $\mathpzc{M}$ the natural maps
$$QX\coprod QY\rightarrow X\coprod Y\rightarrow RX\times RY$$
are all equivalences, where $Q$ denotes cofibrant replacement and $R$ denotes fibrant replacent, 
\item
and is such that $\mathrm{Ho}(\mathpzc{M})$ is additive
\end{enumerate}
 then ${}_{A}\mathrm{Mod}(\mathpzc{M})$ also satisfies all of these properties.
\end{prop}

\begin{proof}
The zero object of $\mathpzc{M}$ has a unique $A$-module structure making it a zero object of ${}_{A}\mathrm{Mod}(\mathpzc{M})$. 

Let $X$ and $Y$ be left $A$-modules. First assume they are cofibrant, and let $Q|X|$ and $Q|Y|$ be their cofibrant replacements in $\mathpzc{M}$. If $RX$ and $RY$ are their fibrant replacements in ${}_{A}\mathrm{Mod}(\mathpzc{M})$, then they are also fibrant replacements in $\mathpzc{M}$. Consider the diagram in $\mathpzc{M}$
$$Q|X|\coprod Q|Y|\rightarrow |X\coprod Y|\cong|X|\coprod|Y|\rightarrow |RX\times RY|\cong|RX|\times|RY|$$
(where $|-|$ denotes the forgetful functor). Since $|-|$ commutes with limits and colimits, all maps are equivalences. Thus $X\coprod Y\rightarrow RX\times RY$ is an equivalence by the two-out-of-three property. Now let $\overline{Q}X$ and $\overline{Q}Y$ we the cofibrant replacements in ${}_{A}\mathrm{Mod}(\mathpzc{M})$. We have a commutative diagram
\begin{displaymath}
\xymatrix{
\overline{Q}X\coprod\overline{Q}Y\ar[d]\ar[r] & R\overline{Q}X\times R\overline{Q}Y\ar[d]\\
X\coprod Y\ar[r] & RX\times RY
}
\end{displaymath}
we have shown that the top and bottom horizontal maps are equivalences. The right-hand vertical map is an equivalence, as products preserve equivalences between fibrant objects. Thus by the two-out-of-three property the left-hand vertical map is an equivalence.

For the final claim, we know that at least ${}_{A}\mathrm{Mod}(\mathpzc{M})$ has all finite biproducts. We need to show that $\mathrm{Hom}_{\mathrm{Ho}({}_{A}\mathrm{Mod}(\mathpzc{M}))}(X,Y)$ is an abelian group. In particular it suffices to prove that it has additive inverses. We may assume that $X$ and $Y$ are cofibrant. We have
$$\mathrm{Hom}_{\mathrm{Ho}({}_{A}\mathrm{Mod}(\mathpzc{M}))}(A\otimes X,Y)\cong \mathrm{Hom}_{\mathrm{Ho}({}_{A}\mathrm{Mod}(\mathpzc{M}))}(A\otimes^{\mathbb{L}}X,Y)\cong\mathrm{Hom}_{\mathrm{Ho}(\mathpzc{M})}(X,Y)$$
is an abelian group by assumption. There is a natural map $i:X\rightarrow A\otimes X$ which is a section of the multiplication map $\mu_{X}:A\otimes X\rightarrow X$. Let $f\in\mathrm{Hom}_{\mathrm{Ho}({}_{A}\mathrm{Mod}(\mathpzc{M}))}(X,Y)$. Consider $g=\mu_{N}\circ (Id\otimes f):A\otimes Y\rightarrow Y$. Define $\tilde{g}\defeq -g\circ i$. Then $\tilde{g}+f=0$ as required.
\end{proof}

\begin{prop}\label{prop:monoidaxiomH}
If $H$ is cofibrant as an object of $\mathpzc{M}$ then, when equipped with the transferred model structure on the monoidal structure defined above, ${}_{H}\mathrm{Mod}(\mathpzc{M})$ is a monoidal model category. If in addition $\mathpzc{M}$ satisfies the monoid axiom then so does ${}_{H}\mathrm{Mod}(\mathpzc{M})$, and if it satisfies the commutative monoid axiom then so does  ${}_{H}\mathrm{Mod}(\mathpzc{M})$
\end{prop}

\begin{proof}
Since $H$ is cofibrant, the underlying map in $\mathpzc{M}$ of a cofibration (resp. an acyclic cofibration) in ${}_{H}\mathrm{Mod}(\mathpzc{M})$ is a cofibration (resp. an acyclic cofibration) in $\mathpzc{M}$. Thus the monoidal model property/ the monoid axiom/ the commutative monoid axiom for ${}_{H}\mathrm{Mod}(\mathpzc{M})$ follows from the same properties for $\mathpzc{M}$.
\end{proof}

\begin{prop}
Suppose that for any $A\in\mathrm{Alg}_{\mathfrak{Ass}}(\mathpzc{M})$ the transferred model structure exists on ${}_{A}\mathrm{Mod}(\mathpzc{M})$. Then for any $A\in\mathrm{Alg}_{\mathfrak{Ass}}({}_{H}\mathrm{Mod}(\mathpzc{M}))$ the transferred model structure exists on ${}_{A}\mathrm{Mod}({}_{H}\mathrm{Mod}(\mathpzc{M}))$.
\end{prop}

\begin{proof}
We need to show that any transfinite composition of pushouts of maps of the form $A\otimes H\otimes f$ where $f$ is a generating acyclic cofibration of $\mathpzc{M}$ is an equivalence. But $H\otimes f$ is an acyclic cofibration in $\mathpzc{M}$. Moreover the forgetful functor ${}_{A}\mathrm{Mod}({}_{H}\mathrm{Mod}(\mathpzc{M}))\rightarrow {}_{A}\mathrm{Mod}(\mathpzc{M})$ commutes with colimits. Since the transferred model structure exists on ${}_{A}\mathrm{Mod}(\mathpzc{M})$ the claim follows.
\end{proof}

\begin{prop}\label{prop:tensorequivH}
Suppose that for any $A\in\mathrm{Alg}_{\mathfrak{Ass}}(\mathpzc{M})$ the transferred model structure exists on ${}_{A}\mathrm{Mod}(\mathpzc{M})$, and that for any such $A$ and any cofibrant $M\in{}_{A}\mathrm{Mod}(\mathpzc{M})$ the functor 
$$M\otimes_{A}(-):{}_{A}\mathrm{Mod}(\mathpzc{M})\rightarrow{}_{A}\mathrm{Mod}(\mathpzc{M})$$ preserves equivalences. Then for any $B\in\mathrm{Alg}_{\mathfrak{Ass}}({}_{H}\mathrm{Mod}(\mathpzc{M}))$, and any cofibrant $N\in{}_{B}\mathrm{Mod}({}_{H}\mathrm{Mod}(\mathpzc{E}))$, the functor
$$N\otimes_{B}(-):{}_{B}\mathrm{Mod}({}_{H}\mathrm{Mod}(\mathpzc{M}))\rightarrow{}_{B}\mathrm{Mod}({}_{H}\mathrm{Mod}(\mathpzc{M}))$$
preserves equivalences.
\end{prop}

\begin{proof}
This follows from the fact that if $N$ is cofibrant as an object of ${}_{B}\mathrm{Mod}({}_{H}\mathrm{Mod}(\mathpzc{E}))$, then it is cofibrant as an object of ${}_{B}\mathrm{Mod}(\mathpzc{M})$
\end{proof}

\begin{prop}\label{prop:AmodH}
Let $H$ be a cofibrant as an object of $\mathpzc{M}$. Let $A\in\mathrm{Alg}_{\mathfrak{Comm}}({}_{H}\mathrm{Mod}(\mathpzc{M}))$. Then when equipped with the transferred model structure ${}_{A}\mathrm{Mod}({}_{H}\mathrm{Mod}(\mathpzc{M}))$ is a combinatorial symmetric monoidal model, where the tensor product is given by $\otimes_{A}$. It is proper if ${}_{A}\mathrm{Mod}(\mathpzc{M})$ is, and it satisfies the monoid/ commutative monoid axiom whenever ${}_{A}\mathrm{Mod}(\mathpzc{M})$ does. 
\end{prop}

\begin{proof}
The fact that it is a combinatorial symmetric monoidal model category satisfying the monoid axiom follows by general nonsense from the fact that ${}_{H}\mathrm{Mod}(\mathpzc{M})$ is. It is also right proper by general nonsense. It remains to prove that it is left proper. The underlying map in $\mathpzc{M}$ of a cofibration in ${}_{A}\mathrm{Mod}({}_{H}\mathrm{Mod}(\mathpzc{M}))$ is a cofibration in ${}_{A}\mathrm{Mod}(\mathpzc{M})$. Since the forgetful functor preserves and reflects equivalences, and commutes with both limits and colimits, left properness of ${}_{A}\mathrm{Mod}({}_{H}\mathrm{Mod}(\mathpzc{M}))$ follows from left properness of ${}_{A}\mathrm{Mod}(\mathpzc{M})$. The claim regaarding the monoid and commutative monoid axioms are similar to Proposition \ref{prop:monoidaxiomH}
\end{proof}

\begin{prop}
Suppose that for any $A\in\mathrm{Alg}_{\mathfrak{Comm}}(\mathpzc{M})$ the transferred model structure exists on ${}_{A}\mathrm{Mod}(\mathpzc{M})$ and on both $\mathrm{Alg}_{\mathfrak{Comm}}({}_{A}\mathrm{Mod}(\mathpzc{M}))$ and $\mathrm{Alg}_{\mathfrak{Comm}}^{nu}({}_{A}\mathrm{Mod}(\mathpzc{M}))$, and that with the transferred model structures these are combinatorial proper model categories. Then the same is true for any $B\in\mathrm{Alg}_{\mathfrak{Comm}}(\mathrm{Mod}({}_{H}\mathrm{Mod}(\mathpzc{M})))$.

Further suppose that whenever $B$ is cofibrant in $\mathrm{Alg}_{\mathfrak{Comm}}({}_{A}\mathrm{Mod}(\mathpzc{M}))$, the functor
$$B\otimes_{A}(-):{}_{A}\mathrm{Mod}(\mathpzc{M})\rightarrow{}_{A}\mathrm{Mod}(\mathpzc{M})$$
preserves equivalences. Then whenever $D$ is cofibrant in $\mathrm{Alg}_{\mathfrak{Comm}}({}_{C}\mathrm{Mod}({}_{H}\mathrm{Mod}(\mathpzc{M})))$, the functor
$$D\otimes_{C}(-):{}_{C}\mathrm{Mod}({}_{H}\mathrm{Mod}(\mathpzc{M}))\rightarrow {}_{D}\mathrm{Mod}({}_{H}\mathrm{Mod}(\mathpzc{M}))$$
preserves equivalences.
\end{prop}

\begin{proof}
This is similar to Proposition \ref{prop:AmodH} and Proposition \ref{prop:tensorequivH}.
\end{proof}

Altogether, this proves the following theorem.

\begin{thm}\label{thm:HAHopf}
With the assumptions on $H$ as in the start of this subsection, ${}_{H}\mathrm{Mod}(\mathpzc{M})$ is a HA context.
\end{thm}

\subsubsection{HA Contexts in Exact Categories}

We now have all the ingredients necessary to prove the following.

\begin{thm}\label{thm:dgHA}
Let $\mathpzc{E}$ be a locally presentable closed symmetric monoidal exact category enriched over $\mathbb{Q}$. Let $(\mathfrak{L},\mathfrak{R})$ be a strongly monoidally $dg$-compatible hereditary cotorsion pair on $\mathpzc{E}$. Then when equipped with the model structure induced by $(\mathfrak{L},\mathfrak{R})$, $\mathrm{Ch}_{\ge0}(\mathpzc{E})$ and $\mathrm{Ch}(\mathpzc{E})$ are HA contexts.
\end{thm}

\begin{proof}
$\mathrm{Ch}(\mathpzc{E})$ and $\mathrm{Ch}_{\ge0}(\mathpzc{E})$ are clearly pointed. They are proper by \cite{kelly2016homotopy} Corollary 4.2.47. 

Direct sums of equivalences are equivalences, so it is clear that for any two objects $X$ and $Y$ all of the maps
$$QX\oplus QY\rightarrow X\oplus Y\rightarrow RX\oplus RY$$
are equivalences. The homotopy category of $\mathrm{Ch}(\mathpzc{E})$ is the derived category, which is of course additive. The homotopy category of $\mathrm{Ch}_{\ge0}(\mathpzc{E})$ is a full subcategry of $\mathrm{Ho}(\mathrm{Ch}(\mathpzc{E}))$ and hence is also additive.

 Since $\mathrm{Ch}(\mathpzc{E})$ (resp. $\mathrm{Ch}_{\ge0}(\mathpzc{E}))$ is a monoidal model category satisfying the monoid axiom, it follows immediately that, with the transferred model structure ${}_{A}\mathrm{Mod}(\mathrm{Ch}(\mathpzc{E}))$ (resp. ${}_{A}\mathrm{Mod}(\mathrm{Ch}_{\ge0}(\mathpzc{E}))$) is a combinatorial, proper, symmetric monoidal model category. 
 
 For any $A\in\mathrm{Alg}_{\mathfrak{Comm}}(\mathrm{Ch}(\mathpzc{E}))$ (resp. any $A\in\mathrm{Alg}_{\mathfrak{Comm}}(\mathrm{Ch}_{\ge0}(\mathpzc{E}))$). By Proposition \ref{prop:gens} and Proposition \label{prop:cofibrantgenerators}, a cofibrant object of ${}_{A}\mathrm{Mod}(\mathrm{Ch}(\mathpzc{E}))$ is a transfinite extension of objects of the form $A\otimes S^{n}(G)$ for $G$ a flat object of $\mathpzc{E}$. Clearly $A\otimes S^{n}(G)$ is flat as an object of ${}_{A}\mathrm{Mod}(\mathrm{Ch}(\mathpzc{E})$. This also works for $\mathrm{Ch}_{\ge0}(\mathpzc{E})$. 
 
 Since $\mathpzc{E}$ is is a monoidal model category enriched over $\mathbb{Q}$ and satisfies the monoid axiom, it also satisfies the strong commutative monoid axiom.  ${}_{A}\mathrm{Mod}(\mathrm{Ch}(\mathpzc{E}))$ and ${}_{A}\mathrm{Mod}(\mathrm{Ch}_{\ge0}(\mathpzc{E}))$ both have sets of generating cofibrations consisting of maps of the form $A\otimes X\rightarrow A\otimes Y$ with $X$ and $Y$ in $\widetilde{dg\mathfrak{L}}$, by Proposition \ref{prop:gens}, Lemma \ref{cofibgen}, and the fact that $\mathfrak{L}$ generates $\mathpzc{E}$. Moreover cofibrant objects are $K$-flat. Since admissible monomorphisms are left proper, and cofibrations are $\otimes$-pure, cofibrations are monoidally ${}_{A}\mathrm{Mod}(\mathrm{Ch}(\mathpzc{E}))$-left proper (resp. monoidally ${}_{A}\mathrm{Mod}(\mathrm{Ch}_{\ge0}(\mathpzc{E}))$-left proper). By Theorem \ref{thm:Commadm} $\mathfrak{Comm}$ is admissible, and the transferred model structures are left proper. Note that since $\mathfrak{L}$ is hereditary, we may choose generating cofibrations with cofibrant domain They are combinatorial and right proper by general nonsense. 

 Finally, by \cite{kelly2019koszul} Proposition 2.20, the underlying $A$-module of a cofibrant object of $\mathrm{Alg}_{\mathfrak{Comm}}({}_{A}\mathrm{Mod}(\mathrm{Ch}(\mathpzc{E}))$ (resp. $\mathrm{Alg}_{\mathfrak{Comm}}({}_{A}\mathrm{Mod}(\mathrm{Ch}_{\ge0}(\mathpzc{E}))$) is cofibrant, and hence $K$-flat.
\end{proof}

\begin{cor}\label{cor:flatHA}
Let $\mathpzc{E}$ be a purely locally $\lambda$-presentable closed symmetric monoidal exact category enriched over $\mathbb{Q}$, which is weakly elementary and has enough flat objects. When equipped with the flat model structures $\mathrm{Ch}(\mathpzc{E})$ and $\mathrm{Ch}_{\ge0}(\mathpzc{E})$ are HA contexts. 
\end{cor}

\begin{example}
Let $\mathcal{X}$ be a geometric stack over a ring $R$ (as in \cite{estrada2014derived}). Then $\mathrm{QCoh}(\mathcal{X})$ has enough flats by \cite{Gross2010VectorBA} Theorem 3.5.5. By \cite{estrada2014derived} Corollary 4.5 $\mathrm{QCoh}(\mathcal{X})$ is a Grothendieck abelian category. Thus by Corollary \ref{cor:flatmodelstructureexists} the flat model structure exists on $\mathrm{QCoh}(\mathcal{X})$. This recovers \cite{estrada2014derived} Theorem 8.1. In fact our proof of the existence of the flat model structure is essentially a generalisation of Estrada's proof in the case of $\mathrm{QCoh}(\mathcal{X})$ (as well as Gillespie's proof for modules over a sheaf of rings on a topological spaces). Moreover Corollary \ref{cor:flatHA} in fact shows that $\mathrm{Ch}_{\ge0}(\mathrm{QCoh}(\mathcal{X}))$ and $\mathrm{Ch}(\mathrm{QCoh}(\mathcal{X}))$ are HA contexts whenever $\mathbb{Q}\subset R$. 
\end{example}

\begin{example}\label{example:locallyprojartin}
Let us consider another related example. As in \cite{estrada2014derived} Theorem 8.2. Let $\mathcal{X}$ be an algebraic stack with pointwise affine stabilizer group that satisfies the quotient property, such as a quotient stack. Let $\mathrm{QCoh}(\mathcal{X})$ denote the category of quasicoherent sheaves on $\mathcal{X}$, and let $\mathfrak{L}$ be the class of locally projective quasi-coherent sheaves (i.e. vector bundles) on $\mathcal{X}$. Then  \cite{estrada2014derived} Theorem 8.2 essentially says that $(\mathfrak{L},\mathfrak{L}^{\perp})$ is a monoidally $dg$-compatible cotorsion pair. The class of locally projective is evidently hereditary. Thus when equipped with the model structure determined by  $(\mathfrak{L},\mathfrak{L}^{\perp})$, $\mathrm{Ch}(\mathrm{QCoh}(\mathcal{X}))$ and $\mathrm{Ch}_{\ge0}(\mathrm{QCoh}(\mathcal{X}))$ are HA contexts by Theorem \ref{thm:dgHA}.
\end{example}

\begin{example}
Let $X$ be a smooth algebraic variety defined over a field $K$ of characteristic $0$. Consider the category $\mathrm{QCoh}(X)$ of quasicoherent sheaves on $X$, and equip it with the locally projective model structure of Example \ref{example:locallyprojartin}. Let $\mathcal{D}_{X}$ denote the sheaf of differential operators on $X$. By \cite{MR2357361} Proposition 1.2.9 there is a closed symmetric monoidal structure on ${}_{\mathcal{D}_{X}}\mathrm{Mod}(\mathrm{QCoh}(X))$, and a strong closed symmetric monoidal structure on the forgetful functor 
$${}_{\mathcal{D}_{X}}\mathrm{Mod}(\mathrm{QCoh}(X))\rightarrow\mathrm{QCoh}(X)$$
Moreover $\mathcal{D}_{X}$ is locally projective. Thus by Theorem \ref{thm:HAHopf}, ${}_{\mathcal{D}_{X}}\mathrm{Mod}(\mathrm{Ch}(\mathrm{QCoh}(X)))$ and ${}_{\mathcal{D}_{X}}\mathrm{Mod}(\mathrm{Ch}_{\ge0}(\mathrm{QCoh}(X)))$ are HA contexts. This generalises the main result of \cite{MR3913977}, which proves that modules over $\mathcal{D}_{X}$ form a HA context when $X$ is smooth affine.
\end{example}

\begin{rem}
Let $\mathpzc{E}$ be a purely locally $\lambda$-presentable closed symmetric monoidal exact category satisfying the monoid axiom, in which transfinite compositions of admissible monomorphisms are admissible monomorphisms. Consider the monoidal model category structure on $\mathrm{Ch}(\mathpzc{E}_{\otimes})$ determined by the Hovey triple $(K\mathcal{F},\mathfrak{W},\mathfrak{W}_{\otimes}^{\perp_{\otimes}})$. It is not at all clear that this is a HA context. The problem is that it is unclear if there are generators of the $\otimes$-pure model structure which are $K$-flat, so that properness of $\mathrm{Alg}_{\mathfrak{Comm}}(\mathpzc{E})$ and $\mathrm{Alg}_{\mathfrak{Comm}}^{nu}(\mathpzc{E})$ is not automatic. However all of the other axioms do hold.
\end{rem}

\subsection{Koszul Duality}

In \cite{kelly2019koszul} we introduced the concept of a Koszul category, and showed that a version of Koszul duality between algebras over operads and their Koszul dual co-operads holds. 

\begin{defn}[\cite{kelly2019koszul} Definition 3.23]
A \textit{Koszul category} is a weakly monoidal model category $\mathpzc{M}$ of the form ${}_{R}\mathrm{Mod}(\mathrm{Ch}(\mathpzc{E}))$ where:
\begin{enumerate}
\item
$\mathpzc{E}$ is a complete and cocomplete symmetric monoidal exact category, and the monoidal structure on $\mathrm{Ch}(\mathpzc{E})$ is the one induced from $\mathpzc{E}$.
\item
$R$ is a commutative monoid in $\mathrm{Ch}(\mathpzc{E})$.
\item
$\mathrm{Ch}(\mathpzc{E})$ is equipped with a combinatorial model structure satisfying the monoid axiom.
\item
weak equivalences in the model structure on $\mathrm{Ch}(\mathpzc{E})$ are thick: if
\begin{displaymath}
\xymatrix{
0\ar[r] & X\ar[d]\ar[r] & Y\ar[d]\ar[r] & Z\ar[d]\ar[r] & 0\\
0\ar[r] & U\ar[r] & V\ar[r] & W\ar[r] & 0
}
\end{displaymath}
is a diagram in $\mathrm{Ch}(\mathpzc{E})r$ in which the top and bottom rows are exact, and any two of the vertical morphisms are weak equivalences, then the third map is a weak equivalence.
\item
$\mathpzc{M}$ is equipped with the transferred model structure (which we assume to exist), and this model structure satisfies the monoid axiom.
\item
$\mathpzc{M}$ satisfies the weak pushout-product axiom.
\end{enumerate}
A Koszul category is said to be \textit{closed} if it is a closed monoidal category.
\end{defn}

\begin{cor}
Let $\mathpzc{E}$ be a purely locally $\lambda$-presentable exact category which is weakly elementary and has enough flat objects. Let $R$ be a commutative monoid in $\mathrm{Ch}(\mathpzc{E})$. Then when $\mathrm{Ch}(\mathpzc{E})$ is equipped with the flat model structure on $\mathrm{Ch}(\mathpzc{E})$ with the transferred model structure, ${}_{R}\mathrm{Mod}(\mathrm{Ch}(\mathpzc{E}))$ is a Koszul category. Moreover it is hereditary in the sense of \cite{kelly2019koszul} Definition 3.39.
\end{cor}
Let $\mathfrak{C}$ be a divided powers co-operad (\cite{kelly2019koszul} Section 1.1.3) in $\mathpzc{M}$, and $\mathfrak{P}$ an admissible operad in $\mathpzc{M}$. A degree $-1$ morphism $\mathfrak{C}\rightarrow\mathfrak{P}$ gives rise to a twisted differential $d^{r}_{\alpha}$ on $\mathfrak{C}\circ\mathfrak{B}$. $\alpha$ is said to be a \textit{twisting morphism} if $(d^{r}_{\alpha})^{2}=0$. We then define $\mathfrak{C}\circ_{\alpha}\mathfrak{P}$ to be the $\Sigma$-module $\mathfrak{C}\circ\mathfrak{P}$ equipped with this differential. A twisting morphism gives rise to the so-called \textit{bar-cobar adjunction}
$$\adj{\Omega_{\alpha}}{\mathpzc{coAlg}_{\mathfrak{C}}^{conil}(\mathpzc{M})}{\mathrm{Alg}_{\mathfrak{P}}(\mathpzc{M})}{B_{\alpha}}$$
between conilpotent coalgebras over $\mathfrak{C}$, and algebras over $\mathfrak{P}$. We assume the existence of certain filtrations on $\mathfrak{C}$ and $\mathfrak{P}$ (\cite{kelly2019koszul} Definitions 4.1 and 4.2, and Asumptions 4.9). Since $\mathfrak{C}$ is a filtered cooperad, we can consider filered coalgebras over $\mathfrak{C}$. Let $\mathrm{coAlg}^{|K|}_{\mathfrak{C}}$ denote the category of filtered $\mathfrak{C}$-coalgebras $C$, whose underlying (unfiltered) coalgebra is conilpotent, and such that $\mathrm{gr}_{n}(C)$ is $K$-flat for each $n\in\mathbb{Z}_{\ge0}$. Finally, let $I$ denote the full subcategory of (unfiltered) $\mathfrak{C}$-coalgebras $D$, for which there exists \textit{some filtration} on $D$, which turns it into an object of $\mathrm{coAlg}^{|K|}_{\mathfrak{C}}$, and let $\mathrm{Alg}_{\mathfrak{P}}^{|K|}$ the full subcategory of $\mathfrak{P}$-algebras whose underlying object in $\mathpzc{M}$ is $K$-flat. The bar-cobar adjunction restricts to an adjunction
$$\adj{\Omega_{\alpha}}{I}{\mathrm{Alg}_{\mathfrak{P}}^{|K|}}{B_{\alpha}}$$
(this is stated for $\mathpzc{M}$ being strong Koszul in the sense of \cite{kelly2019koszul} Definition 3.32, but this is not necessary).
Equip $I$ with the relative structure, where a map $f:C\rightarrow D$ is an equivalence precisely if $\Omega_{\alpha}(f)$ is an equivalence in $\mathrm{Alg}_{\mathfrak{P}}^{|K|}$. By hammock localisation, this presents an $(\infty,1)$-category $\mathrm{L^{H}}(I)$.

\begin{thm}[\cite{kelly2019koszul} Theorem 4.28]
Let $\mathpzc{M}$ be a strong $K$-monoidal Koszul category. The bar-cobar adjunction induces an adjoint equivalence of $(\infty,1)$-categories
$$\adj{\Omega_{\alpha}}{\mathrm{L^{H}}(\mathpzc{I}_{top})}{\textbf{Alg}_{\mathfrak{P}}}{\textbf{B}_{\alpha}}$$
\end{thm}

\begin{rem}
\cite{kelly2019koszul} Theorem 4.28 is stated more generally, and does not require $\mathpzc{M}$ to be $K$-monoidal.
\end{rem}

\section{Extended Example: Sheaves on Spaces Valued in Exact Categories}\label{sec:sheaves}

In this section we come to our main motivating example, the category of sheaves valued in a monoidal elementary exact category. We fix a locally finitely presentable exact category $\mathpzc{E}$ in which filtered colimits are exact and commute with kernels.

\subsection{The Stalk-Wise Exact Structure}
Let $(\mathcal{C},\tau)$ be a small Grothendieck site. Denote by $\mathrm{PreShv}(\mathcal{C},\mathpzc{E})$ the category of presheaves on $\mathcal{C}$ valued in $\mathpzc{E}$, i.e. the category of functors $F:\mathcal{C}^{op}\rightarrow\mathpzc{E}$. This has a natural exact structure where we declare
$$0\rightarrow\mathcal{F}\rightarrow\mathcal{G}\rightarrow\mathcal{H}\rightarrow 0$$
to be exact if 
$$0\rightarrow\mathcal{F}(U)\rightarrow\mathcal{G}(U)\rightarrow\mathcal{H}(U)\rightarrow 0$$
is exact for any $U\in\mathcal{C}$. 

Consider the category $\mathrm{Shv}(\mathcal{C},\tau,\mathpzc{E})$ of sheaves on $(\mathcal{C},\tau)$ valued in $\mathpzc{E}$.  Since filtered colimits commute with kernels in $\mathcal{C}$, the inclusion $\mathrm{Shv}(\mathcal{C},\tau,\mathpzc{E})\rightarrow\mathrm{PrShv}(\mathcal{C},\mathpzc{E})$ admits a left adjoint $L$. If $\kappa$ is such that any cover $\{C_{i}\rightarrow C\}$ in $\tau$ is refinable by a cover of size at most $\kappa$ then $L$ commutes with $\kappa$-filtered colimits. Thus $\mathrm{Shv}(\mathcal{C},\tau,\mathpzc{E})$ is a presentable catgory. 

Let us now specialise to the case that $(\mathcal{C},\tau)$ is the site associated to a topological space $X$. In this instance we will just right the category of sheaves as $\mathrm{Shv}(X,\mathpzc{E})$.

Declare a sequence 
$$0\rightarrow \mathcal{X}\rightarrow\mathcal{Y}\rightarrow\mathcal{Z}\rightarrow 0$$
to be \textit{stalk-wise} exact in $\mathrm{Shv}(X,\mathpzc{E})$ if for each $x\in X$ the sequence on stalks
$$0\rightarrow\mathcal{X}_{x}\rightarrow\mathcal{Y}_{x}\rightarrow\mathcal{Z}_{x}\rightarrow 0$$
is exact in $\mathpzc{E}$. 

Note that given a map of spaces $f:X\rightarrow Y$ we get an adjunction
$$\adj{f^{-1}}{\mathrm{Shv}(Y;\mathpzc{E})}{\mathrm{Shv}(X,\mathpzc{E})}{f_{*}}$$
and $(f^{-1}\mathcal{F})_{x}\cong\mathcal{F}_{f(x)}$.

\begin{lem}
The collection of stalk-wise exact sequences defines an exact structure on $\mathrm{Shv}(X,\mathpzc{E})$ in which filtered colimits are exact. 
\end{lem}

\begin{proof}
Clearly it contains split exact sequences and isomorphisms, and both admissible monomorphisms and admissible epimorphisms are closed under composition. The stalk of a pushout of sheaves is the pushout of the stalks. Since filtered colimits commute with kernels, the stalk of a pullback of sheaves is also the pullback of the stalks. Thus stalk-wise admissible monomorphisms are closed under pushouts, and stalk-wise admissible epimorphisms are closed under pullbacks. 
\end{proof}

Almost tautologically we have the following.

\begin{prop}
Both $\mathrm{PreShv}(X,\mathpzc{E})$ and $\mathrm{Shv}(X,\mathpzc{E})$ are locally finitely presented exact categories.
\end{prop}

Since filtered colimits in $\mathpzc{E}$ are exact, we have the following.

\begin{prop}
The functor $L:\mathrm{PreShv}(X,\mathpzc{E})\rightarrow\mathrm{Shv}(X,\mathpzc{E})$ is exact. 
\end{prop}

Let $U\subset X$ be open, and let $\mathcal{F}\in\mathrm{Shv}(X,\mathpzc{E})$. Denote by $j^{pre}_{U!}\mathcal{F}$ the presheaf with  $j_{U!}\mathcal{F}(V)=\mathcal{F}(V)$ if $V\subset U$ and $j_{U!}\mathcal{F}(V)=\emptyset$ otherwise. If $\mathcal{F}$ is a sheaf then so is $j_{U!}\mathcal{F}$. 
$j^{pre}_{U!}:\mathrm{PreShv}(U,\mathpzc{E})\rightarrow\mathrm{PreShv}(X,\mathpzc{E})$ is left adjoint to $j^{-1}:\mathrm{PreShv}(X,\mathpzc{E})\rightarrow \mathrm{PreShv}(U,\mathpzc{E})$. We define $f_{!}$ for arbitary maps later.

For $P$ an object of $\mathpzc{E}$ and $U\subset X$ open, denote by $P^{pre}_{U}$ the constant presheaf with value $P$, and by $P_{U}$ its sheafification. 

%

\begin{lem}
Let $G$ be a generator for $\mathpzc{E}$. Then $\{j^{pre}_{U!}G^{pre}_{U}:U\subset U\subset X\textrm{ open }\}$ generates $\mathrm{PreShv}(X,\mathpzc{E})$, and $\{j_{U!}G_{U}:U\subset U\subset X\textrm{ open }\}$ generates $\mathrm{Shv}(X,\mathpzc{E})$
\end{lem}

\begin{proof}
The following is standard for $\mathpzc{E}=\mathpzc{Ab}$ (see e.g. the section in \cite{stacks-project} on locally free sheaves). It suffices to prove the pre-sheaf claim, since the sheafification functor is exact. Let $\mathcal{F}$ be a pre-sheaf, and consider $\bigoplus_{U\subset X,s\in\mathrm{Hom}(G,\mathcal{F}(U))}j^{pre}_{U!}G$. For each $\alpha\in\mathrm{Hom}(G,\mathcal{F}(U))$ there is a tautological map $j_{U!}G_{U}\rightarrow\mathcal{F}$. This induces the map $\bigoplus_{U\subset X,s\in\mathrm{Hom}(G,\mathcal{F}(U))}j^{pre}_{U!}G\rightarrow\mathcal{F}$. We claim that this map is an admissible epimorphism. For this it suffices to prove that for each open $U\subset X$ the map $\bigoplus_{U\subset X,s\in\mathrm{Hom}(G,\mathcal{F}(U))}j^{pre}_{U!}G(U)\rightarrow\mathcal{F}(U)$ is an epimorphism. Now 
$$\bigoplus_{U\subset X,s\in\mathrm{Hom}(G,\mathcal{F}(U))}j^{pre}_{U!}G(U)=\bigoplus_{s\in\mathrm{Hom}(G,\mathcal{F}(U))}G(U)$$ There is some admissible epimorphism $\bigoplus_{i\in\mathcal{I}}G\rightarrow \mathcal{F}(U)$. The restriction of this map to the copy of $G$ indexed by $i\in\mathcal{I}$ provides a map $G\rightarrow\mathcal{F}(U)$. Thus the map $\bigoplus_{i\in\mathcal{I}}G\rightarrow \mathcal{F}(U)$ factors through $\bigoplus_{s\in\mathrm{Hom}(G,\mathcal{F}(U))}G(U)\rightarrow\mathcal{F}(U)$.
\end{proof}

\begin{cor}
Let $\mathcal{S}$ be a set of objects in $\mathpzc{E}$ such that $\mathrm{Filt}(\mathcal{S})$ contains a generator. Let $\mathcal{S}^{\mathrm{Loc}_{X},pre}$ denote the set of objects $\{j^{pre}_{U!}S^{pre}_{U}:U\subset X\textrm{ open },S\in\mathcal{S}\}$ in $\mathrm{PreShv}(X,\mathpzc{E})$, and let $\mathcal{S}^{\mathrm{Loc}_{X}}$ denote the set of objects $\{j_{U!}S_{U}:U\subset X\textrm{ open },S\in\mathcal{S}\}$ in $\mathrm{PShv}(X,\mathpzc{E})$. $(\mathrm{Filt}(\mathcal{S}^{\mathrm{Loc}_{X},pre}),(\mathcal{S}^{\mathrm{Loc}_{X},pre})^{\perp})$ and $(\mathrm{Filt}(\mathcal{S}^{\mathrm{Loc}_{X}}),(\mathcal{S}^{\mathrm{Loc}_{X}})^{\perp})$ are functorially complete cotorsion pairs on $\mathrm{PreShv}(X,\mathpzc{E})$ and $\mathrm{Shv}(X,\mathpzc{E})$ respectively.
\end{cor}


\subsubsection{Locally and Stalkwise Deconstructible Classes}

Let $\mathpzc{A}$ be a class of objects in $\mathpzc{E}$ which is $\gamma$-pure subobject stable for any regular $\gamma\ge\lambda$. Let $X$ be a space, and denote by $\mathpzc{A}^{loc}$ the full subcategory of $\mathrm{Shv}(X,\mathpzc{E})$ consisting of those sheaves $\mathcal{F}$ for which there exists a cover $\mathcal{V}$ of $X$, such that for any $U\in\mathcal{V}$, $\mathcal{F}(U)$ is in $\mathpzc{A}$. Denote by $\mathpzc{A}^{stalk}$ the full subcategory of $\mathrm{Shv}(X,\mathpzc{E})$ consisting of those sheaves $\mathcal{F}$ such that $\mathcal{F}_{x}\in\mathpzc{A}$ for any $x\in X$. Since $j_{U!}^{pre}(G))_{x}$ is either $0$ or $G$, we have the following.

\begin{prop}
If $\mathpzc{A}$ generates $\mathpzc{E}$ then $\mathpzc{A}^{stalk}$ generates $\mathrm{Shv}(X,\mathpzc{E})$
\end{prop}

\begin{lem}
If $\mathpzc{A}$ is $(\aleph_{0}$-)pure subobject stable in $\mathpzc{E}$ then
\begin{enumerate}
\item
 $\mathpzc{A}^{stalk}$ is $(\aleph_{0}$-)pure subobject stable in $\mathrm{Shv}(X,\mathpzc{E})$.
 \item
  there are arbitrarily large cardinals $\kappa$ such that
$\mathpzc{A}^{loc}$ is $\kappa$-pure subobect stable.
\end{enumerate}

\end{lem}

\begin{proof}
\begin{enumerate}
\item
This is clear for $\mathpzc{A}^{stalk}$. 
\item
For $\mathpzc{A}^{loc}$, we let $\kappa$ be sufficiently large so that 
\begin{enumerate}
\item
the inclusion $\mathrm{Shv}(X,\mathpzc{E})\rightarrow\mathrm{PreShv}(X,\mathpzc{E})$ commutes with $\kappa$-filtered colimits
\item
for each $U\subset X$ open the functor $\Gamma(U,-):\mathrm{Shv}(U,\mathpzc{E})\rightarrow\mathpzc{E}$ commutes with $\kappa$-filtered colimits.
\end{enumerate}
 and let $\mathcal{F}\in\mathpzc{A}^{loc}$. Then if $\mathcal{K}\rightarrow\mathcal{F}$ is a $\kappa$-pure monomorphism in $\mathrm{Shv}(X,\mathpzc{E})$, it is also a $\kappa$-pure monomorphism in $\mathrm{PreShv}(X,\mathpzc{E})$. Thus for any open $U$ in $\mathcal{X}$ we have that $\mathcal{K}(U)\rightarrow\mathcal{F}(U)$ is a $\kappa$-pure monomorphism. Hence $\mathcal{K}\in\mathpzc{A}^{loc}$. 
 \end{enumerate}
\end{proof}

\begin{cor}
Let $(\mathfrak{L},\mathfrak{R})$ be a $dg$-compatible cotorsion pair on $\mathpzc{E}$ such that $\mathfrak{L}$ is $\gamma$-pure subobject stable for any sufficiently large regular $\gamma$. Then $(\mathfrak{L}^{stalk},(\mathfrak{L}^{stalk})^{\perp})$ is a functorially complete cotorsion pair, and it is $dg_{*}$-compatible for $*\in\{\ge0,\le0,\emptyset\}$. 
\end{cor}

\subsubsection{Monoidal Structures on Categories of Sheaves}

Let $\mathpzc{C}$ be a closed symmetric monoidal locally finitely presented category and $X$ a topological space. Let $k$ denote the unit of the monoidal structure. There is a closed symmetric monoidal structure on $\mathrm{PrShv}(X,\mathpzc{C})$. The tensor product is defined by 
$$(\mathcal{F}\otimes\mathcal{G})(U)\defeq\mathcal{F}(U)\otimes\mathcal{G}(U)$$
The internal hom may be constructed as follows (c.f. \cite{qacs} 2.2.13). For $U\subset X$ open define $\underline{\mathpzc{Hom}}(\mathcal{F},\mathcal{G})(U)$ to be the equaliser of the two natural maps
$$\prod_{V\in\mathrm{Open}(U)}\mathpzc{Hom}(\mathcal{F}(V),\mathcal{G}(V))\rightarrow\prod_{W\subset V\in\mathrm{Op}(U)}\mathpzc{Hom}(\mathcal{F}(V),\mathcal{G}(W))$$

The following is straightforward. 

\begin{prop}
Let $\mathcal{F}\in\mathrm{PrShv}(X,\mathpzc{E})$ and $\mathcal{G}\in\mathrm{Shv}(X,\mathpzc{E})$. Then $\underline{\mathpzc{Hom}}(\mathcal{F},\mathcal{G})\in\mathrm{Shv}(X,\mathpzc{E})$. In particular the category of sheaves inherits a symmetric monoidal structure, given by sheafifying the presheaf tensor product.
\end{prop}

Equip $\mathrm{Shv}(X,\mathpzc{E})$ with the stalkwise exact structure, and the closed symmetric monoidal structure as constructed above.  

\begin{prop}
An object $\mathcal{F}$ of $\mathrm{Shv}(X,\mathpzc{E})$ is (strongly) flat if and only if each $\mathcal{F}_{x}$ is (strongly) flat. In particular if $\mathpzc{E}$ has enough (strong) flats then so does $\mathrm{Shv}(X,\mathpzc{E})$.
\end{prop}

\begin{proof}
For flats this follows immediately from the fact thay by construction, the sheaf tensor product commutes with taking stalks. For strong flats it follows from the fact that taking kernels also commutes with stalks. 
\end{proof}

\subsubsection{Model Structures for Sheaves}

In this section let $\mathpzc{E}$ be a locally finitely presentable exact category in which filtered colimits are exact and commute with finite limits. In particular $\mathrm{LH}(\mathpzc{E})$ is Grothendieck abelian. Equip $\mathrm{Ch}(\mathpzc{E})$ with the injective model structure, and denote by $\mathbf{Ch}(\mathpzc{E})$ the corresponding $(\infty,1)$-category. Let $X$ be a topological space, and consider the $(\infty,1)$-category $\mathbf{PrShv}(X,\mathbf{Ch}(\mathpzc{E}))=\mathbf{Fun}(\mathrm{N}(\mathrm{Open}(X)^{op}),\mathbf{Ch}(\mathpzc{E}))$. This is presented by a model category which is Quillen equivalent to the injective model structure on $\mathrm{Ch}(\mathrm{PrShv}(X,\mathpzc{E}))$. 

Let $R:\mathpzc{D}\rightarrow\mathpzc{E}$ be a right adjoint exact functor between locally presentable exact categories in which filtered colimits are exact and commute with finite limits. Further assume that $R$ commutes with filtered colimits. Consider the induced functor 
$$\mathcal{R}:\mathrm{PrShv}(X,\mathpzc{D})\rightarrow\mathrm{PrShv}(X,\mathpzc{E})$$
This is also a right adjoint exact functor. Let $\mathcal{D}\in\mathrm{Shv}(X,\mathpzc{D})$. Since $\mathcal{R}$ commutes with limits, we have $\mathcal{R}(\mathcal{D})\in\mathrm{Shv}(X,\mathpzc{E})$. Since $R$ commutes with filtered colimits, $\mathcal{R}$ is exact for the stalk-wise exact structures. Moreover $\mathcal{R}:\mathrm{Shv}(X,\mathpzc{D})\rightarrow\mathrm{Shv}(X,\mathpzc{E})$ admits a left adjoint $\overline{\mathcal{L}}$. Indeed $R$ admits a left adjoint $L$ by assumption, and then $\mathcal{R}:\mathrm{PrShv}(X,\mathpzc{D})\rightarrow\mathrm{PrShv}(X,\mathpzc{E})$ admits a left adjoint $\mathcal{L}$ given by applying $L$ object-wise. $\mathcal{L}$ is the sheafification of the restriction of $\mathcal{L}:\mathrm{PrShv}(X,\mathpzc{E})\rightarrow\mathrm{PrShv}(X,\mathpzc{D})$ to $\mathrm{Shv}(X,\mathpzc{E})$.

Suppose now that $\mathpzc{D}$ is a thick, reflective, generating subcategory of $\mathpzc{E}$. We claim that $\mathrm{Shv}(X,\mathpzc{D})$ is a thick, reflective, generating subcategory of $\mathrm{Shv}(X,\mathpzc{E})$. It is clearly reflective. Indeed let $\mathcal{D}\in\mathrm{Shv}(X,\mathpzc{D})$. Then $\mathcal{L}\circ\mathcal{R}(\mathcal{D})\cong\mathcal{D}$ is already a sheaf. Now let
$$0\rightarrow\mathcal{E}\rightarrow\mathcal{F}\rightarrow\mathcal{G}\rightarrow 0$$
be a stalk-wise exact sequence in $\mathrm{Shv}(X,\mathpzc{E})$ with $\mathcal{E},\mathcal{G}\in\mathrm{Shv}(X,\mathpzc{D})$. Then for each $x\in X$ we have that $\mathcal{F}_{x}\in\mathpzc{D}$. Consider the natural map
$$\mathcal{F}\rightarrow\mathcal{R}\circ\overline{\mathcal{L}}(\mathcal{F})$$
Since $\overline{\mathcal{L}}(\mathcal{F})_{x}\cong\mathcal{L}(\mathcal{F})_{x}\cong L(\mathcal{F}_{x})$ and $\mathcal{R}(\mathcal{H})_{x}\cong R(\mathcal{H}_{x})$ for any $\mathcal{F}\in\mathrm{Shv}(X,\mathpzc{E})$ and any $\mathcal{H}\in\mathrm{Shv}(X,\mathpzc{D})$, the map $\mathcal{F}\rightarrow\mathcal{R}\circ\overline{\mathcal{L}}(\mathcal{F})$ is stalk-wise an isomorphism, and thus is an isomorphism of sheaves. Hence $\mathcal{F}\in\mathrm{Shv}(X,\mathpzc{D})$. Finally, to see that $\mathrm{Shv}(X,\mathpzc{D})$ is generating it suffices to observe that $\mathcal{R}(j_{U!}G)\cong j_{U!}(R(G))$ for any $G\in\mathpzc{D}$. By Corollary \ref{cor:goingupforcomplexes} we have the following lemma.

\begin{lem}
Let $\mathpzc{E}$ be a locally presentable exact category in which filtered colimits are exact and commute with finite limits, and $\mathpzc{D}$ a locally presentable thick exact generating subcategory such that the inclusion $\mathpzc{D}\rightarrow\mathpzc{E}$ is reflective and commutes with filtered colimits. Then there is an adjoint equivalence of $(\infty,1)$-categories
$$\adj{\overline{\mathbf{L}}}{\mathbf{Ch}(\mathrm{Shv}(X,\mathpzc{E})}{\mathbf{Ch}(\mathrm{Shv}(X,\mathpzc{D}))}{\mathbf{R}}$$
which is $t$-exact for the left $t$-structure.
\end{lem}

\begin{cor}
Let $\mathpzc{E}$ be a locally presentable exact category in which filtered colimits are exact and commute with finite limits. Then there is a natural equivalence of categories
$$\mathrm{LH}(\mathrm{Shv}(X,\mathpzc{E}))\cong\mathrm{Shv}(X,\mathrm{LH}(\mathpzc{E}))$$
\end{cor}

\subsection{The $(\infty,1)$-Category of Rigid Sheaves}

We now compare the $(\infty,1)$-category $\mathbf{Ch}(\mathrm{Shv}(X,\mathpzc{E}))$ of `rigid' sheaves with the category of $(\infty,1)$-sheaves $\mathbf{Shv}(X,\mathbf{Ch}(\mathpzc{E}))$. From now on we suppose that $\mathpzc{E}$ is \textit{elementary}, that is, it has a generating set of $\aleph_{0}$-compact projectives. For an object $F_{\bullet}\in\mathrm{Ch}(\mathrm{PrShv}(X,\mathpzc{E}))$ denote by $\mathpzc{LH}_{n}(F_{\bullet})\in\mathrm{LH}(\mathpzc{E})$ the sheafification of the assignment $U\mapsto\mathrm{LH}_{n}(F_{\bullet}(U))$.


\begin{prop}
 The following categories are equivalent.
\begin{enumerate}
\item
$\mathbf{Ch}(\mathrm{Shv}(X,\mathpzc{E}))$
\item
The localisation of $\mathbf{Ch}(\mathrm{PreShv}(X,\mathpzc{E}))$ at maps $f:X\rightarrow Y$ such that $\mathpzc{LH}_{n}(X)\rightarrow\mathpzc{LH}_{n}(Y)$ is an isomorphism of sheaves for all $n\in\mathbb{Z}$.
\end{enumerate}
If $\mathpzc{E}$ is an elementary exact category, then the following are equivalent
\begin{enumerate}
\item
$\mathbf{Ch}(\mathrm{Shv}(X,\mathpzc{E}))$
\item
The localisation of $\mathbf{Ch}(\mathrm{PreShv}(X,\mathpzc{E}))$ at hypercovers.
\end{enumerate}
\end{prop}

\begin{proof}
For the first claim equip both $\mathrm{Ch}(\mathrm{PreShv}(X,\mathpzc{E}))$ and $\mathrm{Ch}(\mathrm{Shv}(X,\mathpzc{E}))$ with the injective model structures. There is a Quillen adjunction
$$\adj{L}{\mathrm{Ch}(\mathrm{PreShv}(X,\mathpzc{E}))}{\mathrm{Ch}(\mathrm{Shv}(X,\mathpzc{E}))}{i}$$
where $L$ denotes sheafification, and $i$ denotes the inclusion. Since $L$ is exact this is in fact a Quillen reflection. Thus it presents a localisation of $(\infty,1)$-categories, where we localise $\mathrm{Ch}(\mathrm{PreShv}(X,\mathpzc{E}))$ at those maps $f:A\rightarrow B$ such that $L(f):L(A)\rightarrow L(B)$ is an equivalence. But $L(f)$ is an equivalence precsiely if each map of stalks $L(A)_{x}\cong A_{x}\rightarrow B_{x}\cong L(B)_{x}$ is an equivalence. Since filtered colimits are exact for the left $t$-structure, this is equivalent to 
$$\mathrm{LH}_{n}(A_{x})\cong\mathpzc{LH}_{n}(A)_{x}\rightarrow\mathpzc{LH}_{n}(B)_{x}\cong\mathrm{LH}_{n}(B_{x})$$
being an isomorphism for all $n\in\mathbb{Z}$. However this in turn is equivalent to $\mathpzc{LH}_{n}(X)\rightarrow\mathpzc{LH}_{n}(Y)$ being an isomorphism of sheaves for all $n\in\mathbb{Z}$. 

For the second claim, observe that for $\mathpzc{E}$ elementary, a presheaf $\mathcal{F}$ is in $\mathbf{Ch}(\mathrm{Shv}(X,\mathpzc{E}))$ (resp. in the localisation of $\mathbf{Ch}(\mathrm{PreShv}(X,\mathpzc{E}))$ at hypercovers) if and only if $\mathrm{Hom}(P,\mathcal{F})$ is in $\mathbf{Ch}(\mathrm{Shv}(X,\mathpzc{E}))$ (resp. in the localisation of $\mathbf{Ch}(\mathrm{PreShv}(X,\mathpzc{E}))$ at hypercovers) for all compact projective generators $P$. Thus the claim immediately follows since it is true for the category of abelian groups by \cite{scholze2022six} Proposition 7.1.
\end{proof}

The next result also follows from the fact that it is true when $\mathpzc{E}$ is the category of abelian groups.

\begin{prop}
Let $\mathpzc{E}$ be an elementary exact category. The inclusion
$$\mathbf{Ch}_{\le n}(\mathrm{Shv}(X,\mathpzc{E}))\rightarrow\mathbf{Shv}_{\le n}(X;\mathbf{Ch}(\mathpzc{E}))$$
is an equivalence for all $n\in\mathbb{Z}$. In particular inclusion
$$\mathbf{Ch}_{-}(\mathrm{Shv}(X,\mathpzc{E}))\rightarrow\mathbf{Shv}_{-}(X;\mathbf{Ch}(\mathpzc{E}))$$
is an equivalence.
\end{prop}
%
%
%

\begin{defn}
A space $X$ is said to be $\mathpzc{E}$-\textit{hypercomplete} if the localisation map
$$L:\mathbf{Shv}(X,\mathbf{Ch}(\mathpzc{E}))\rightarrow\mathbf{Ch}(\mathrm{Shv}(X,\mathpzc{E}))$$
is an equivalence of $(\infty,1)$-categories.
\end{defn}

\begin{prop}
Let $X$ be paracompact of finite covering dimension, and suppose that $\mathpzc{E}$ is an elementary exact category. Then $X$ is $\mathpzc{E}$-hypercomplete.
\end{prop}

\begin{proof}
Let $f:\mathcal{F}\rightarrow\mathcal{G}$ be a map of presheaves. It is a \v{C}ech-local equivalence (resp. a hyper-local equiavalence) if and only if $\mathrm{Hom}(P,f)$ is a \v{C}ech-local equivalence (resp. a hyper-local equivalence) of presheaves of abelian groups for each $\aleph_{0}$-compact projective $P$. The claim then follows from the corresponding result for abelian groups by\cite{lurie2006higher} Theorem 7.2.3.6, Proposition 7.2.1.10, and Corollary 7.2.1.12..
\end{proof}

\begin{defn}
A map $f:X\rightarrow Y$ of spaces is said to be \textit{fibrewise } $\mathpzc{E}${-complete} if for any $y\in Y$ the space $f^{-1}(y)$ is $\mathpzc{E}$-complete.
\end{defn}

\begin{example}
Any injection of spaces is fibrewise hypercomplete.
\end{example}

\subsection{The Three- and Six-Functor Formalisms}

We conclude by using the model structures we have developed, as well as recent formulations of three- and six- functor formalisms due to \cite{mann2022p} (see \cite{scholze2022six} for a good exposition), to compare with \cite{spaltenstein} and generalise results therein. Let $\mathrm{C}$ be an $(\infty,1)$-category with finite limits, and $\mathrm{E}\subset\mathrm{C}$ be a class of morphisms containing equivalences and is stable by both pullback and composition. Denote by $\mathrm{Corr}(\mathrm{C},\mathrm{E})$ the category of of correspondences. Objects are the same as objects of $\mathrm{C}$. A morphism $c\rightarrow d$ is a span
\begin{displaymath}
\xymatrix{
& x\ar[dl]^{f}\ar[dr]^{g}& \\
c & & d
}
\end{displaymath}
in $\mathrm{C}$ where $g\in\mathrm{E}$. Composition is given by forming a pullback square, and looking at the `long legs':

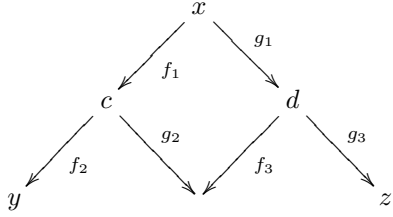
\begin{displaymath}
\xymatrix{
& & x\ar[dl]^{f_{1}}\ar[dr]^{g_{1}} & & \\
& c\ar[dl]^{f_{2}}\ar[dr]^{g_{2}} & & d\ar[dl]^{f_{3}}\ar[dr]^{g_{3}} & \\
y & & & & z
}
\end{displaymath}
This category is a symmetric monoidal  with the monoidal functor given by the Cartesian product. 

\begin{defn}[\cite{mann2022p} Definition A.5.6]
A $3$\textit{-functor formalism} (or a \textit{pre-}$6$\textit{ functor formalism}) is a lax symmetric monoidal functor $\mathbf{D}:\mathrm{Corr}(\mathrm{C},\mathrm{E})\rightarrow\mathbf{Cat}_{\infty}$.
\end{defn}

\subsubsection{$6$-Functor Formalims for Rigid Sheaves}

Let $\mathbf{E}$ be a monoidal stable $(\infty,1)$-category. Let $\mathrm{LCHaus}$ denote the category of locally compact Hausdorff spaces. Consider the geometric setup $(\mathrm{LCHaus},\mathrm{LCHaus})$. Volpe \cite{volpe2021six} explains how to associate to any map $f:X\rightarrow Y$ in $\mathrm{LCHaus}$ the functors $f_{\infty}^{-1},(f_{*})_{\infty},(f_{!})_{\infty},f_{\infty}^{!}$, and shows that the assignment $\mathrm{Corr}(\mathrm{LCHaus},\mathrm{LCHaus})^{op}\rightarrow\mathbf{Cat}_{\infty}$ sending $X$ to $\mathbf{Shv}(X,\mathbf{E})$, and a span
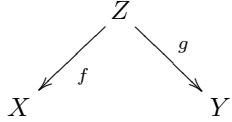
\begin{displaymath}
\xymatrix{
&Z\ar[dl]^{f}\ar[dr]^{g}&\\
X & & Y
}
\end{displaymath}
to $(f_{!})_{\infty}g_{\infty}^{-1}$
 defines a six-functor formalism. (The decoration of the functors with the subscript $\infty$ here is to distinguish them from the functors between categories of the form $\mathbf{Ch}(\mathrm{Shv}(X,\mathpzc{E}))$ later). Now fix $\mathbf{E}=\mathbf{Ch}(\mathpzc{E})$ for a monoidal elementary exact category $\mathpzc{E}$. In this case $(f_{!})_{\infty}$ can be constructed as follows. 
 
\begin{defn}
Let $f:X\rightarrow Y$ be a map of locally compact Hasudorff spaces. A subset $Q$ of $X$ is said to be $f$-\textit{proper} if the restriction of $f$ to $Q$ is proper.
\end{defn}

For a closed subset $Q$ of $X$ and $\mathcal{F}\in\mathbf{Shv}(X,\mathbf{C})$ define
$$\mathbb{R}\mathcal{F}(Q)\defeq\mathrm{Fib}(\mathcal{F}(X)\rightarrow\mathcal{F}(X\setminus Q))$$

\begin{defn}
Let $f:X\rightarrow Y$ be a map of $\mathpzc{E}$-ringed locally compact Hasudorff spaces. For $\mathcal{F}\in\mathcal{O}_{X}\mathbf{Shv}(X,\mathpzc{E})$ and $U\subset Y$ open define
$$(f_{!})_{\infty}\mathcal{F}(U)\defeq\colim_{Q\subset f^{-1}(U), Q\; f-\textrm{proper}}\mathbb{R}f_{*}\mathcal{F}(Q)$$
\end{defn}

\begin{rem}
If $\mathcal{F}\in\mathbf{Ch}(\mathrm{Shv}(X,\mathpzc{E}))$ then $(f_{!})_{\infty}\mathcal{F}\in\mathbf{Ch}(\mathrm{Shv}(Y,\mathpzc{E}))$. We denote the resulting functor by $\mathbb{R}f_{!}$.
\end{rem}

We wish to use the results of Volpe to establish a six functor formalism for rigid sheaves, i.e. the assignment $X\mapsto\mathbf{Ch}(\mathrm{Shv}(X,\mathpzc{E}))$ sending
$$X\mapsto\mathbf{Ch}(\mathrm{Shv}(X,\mathpzc{E}))$$
and a span
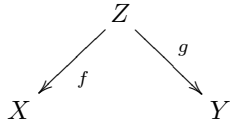
\begin{displaymath}
\xymatrix{
&Z\ar[dl]^{f}\ar[dr]^{g}&\\
X & & Y
}
\end{displaymath}
with $f$ in some class $\mathrm{E}$ to $f_{!}g^{-1}$ defines a three-functor formalism. This amounts to proving the base-change formula, and for this we must restrict the maps in our category of correspondences. Let
\begin{displaymath}
\xymatrix{
X'\ar[d]^{f'}\ar[r]^{g'} & X\ar[d]^{f}\\
Y'\ar[r]^{g} & Y
}
\end{displaymath}
be a fibre-product diagram of locally compact Hausdorff spaces. There is a natural transformation.
 $$g^{-1}\mathbb{R}f_{!}\rightarrow \mathbb{R}f_{!}'(g')^{-1}$$
 This can be seen by standard methods as for sheaves valued in abelian groups, or as follows. We have a natural isomrphism
 $$g_{\infty}^{-1}(f_{!})_{\infty}\cong (f'_{!})_{\infty}(g')^{-1}_{\infty}$$
 Now for any space $U$ let $L_{U}:\mathbf{Shv}(U,\mathbf{Ch}(\mathpzc{E}))\rightarrow\mathbf{Ch}(\mathrm{Shv}(U,\mathpzc{E}))$ denote the localisation functor. For any map $h:U\rightarrow V$, $\mathcal{F}\in\mathbf{Ch}(\mathrm{Shv}(V,\mathpzc{E}))$, and $\mathcal{G}\in\mathbf{Ch}(\mathrm{Shv}(U,\mathpzc{E}))$ we have $L_{U}(h^{-1}_{\infty}\mathcal{F})\cong h^{-1}\mathcal{F}$. We also have $(h_{!})_{\infty}\mathcal{G}\cong\mathbb{R}h_{!}\mathcal{G}$. Now 
 $$g^{-1}\mathbb{R}f_{!}\mathcal{F}\cong  L_{Y'}(g_{\infty}^{-1}(\mathbb{R}f_{!}\mathcal{F}))\cong L_{Y'} (g_{\infty}^{-1}(f_{!})_{\infty}\mathcal{F})\cong L_{Y'}((f'_{!})_{\infty}(g')^{-1}_{\infty}\mathcal{F})\rightarrow L_{Y'}((f'_{!})_{\infty}L_{X'}((g')^{-1}\mathcal{F}))\cong \mathbb{R}f_{!}'(g')^{-1}(\mathcal{F})$$
 
Say that $f:X\rightarrow Y$ is in $\tilde{\mathrm{E}}$ if for any $:Z\rightarrow Y$ the natural transformation $g^{-1}f_{!}\rightarrow f_{!}'(g')^{-1}$ is an equivalence. Say that $f$ is in $\mathrm{E}$ if any pullback of $f$ along any morphism is in $\tilde{\mathrm{E}}$. It is clear that $\mathrm{E}$ contains isomorphisms, and is closed under fibre-product and composition. The following is tautological.
 
 \begin{prop}
The assignment $\mathbf{Ch}(\mathrm{Shv}(-,\mathpzc{E})):\mathrm{Corr}(\mathrm{LCHaus},\mathrm{E})^{op}\rightarrow\mathbf{Cat}_{\infty}$ as defined above is a three-functor formalism.
 \end{prop}

Let us now investigate when we get a six functor formalism. The functor $f^{-1}$ has a right adjoint $\mathbb{R}f_{*}$, and the functor $\otimes^{\mathbb{L}}$ has a right adjoint $\mathbb{R}\underline{\mathrm{Hom}}(-,-)$, it suffices to determine when $\mathbb{R}f_{!}$ has a right adjoint.

\begin{defn}
A map $f:X\rightarrow Y$ is said to be \textit{universally }$!$-\textit{adjointable} if for any map $g:X'\rightarrow X$, the projection $f':X'\times_{X}Y\rightarrow X'$ is such that $f'_{!}$ commutes with colimits. 
\end{defn}

The class of universally $!$-adjointable maps is denoted $\mathrm{A}_{!}$. Clearly $(\mathrm{LCHaus},\mathrm{A}_{!})$ is a geometric setup.  Completely tautologically, we have the following.

\begin{prop}
 The assignment $\mathbf{Ch}(\mathrm{Shv}(-,\mathpzc{E})):\mathrm{Corr}(\mathrm{LCHaus},\mathrm{E}\cap\mathrm{A}_{!})^{op}\rightarrow\mathbf{Cat}_{\infty}$ is a six-functor formalism.
\end{prop}

 Let $\overline{\mathrm{H}}$ denote the class of fibrewise $\mathpzc{E}$-complete maps in $\mathrm{LCHaus}$. It is clear that $\overline{\mathrm{H}}$ is closed under fibre-products and contains isomorphisms. Let $\mathrm{H}$ denote the class of maps obtained as compositions of maps in $\overline{\mathrm{H}}$. Then $(\mathrm{LCHaus},\mathrm{H})$ is a geometric setup. We claim that $\mathrm{H}\subset\mathrm{E}\cap\mathrm{A}_{!}$. It suffices to prove that $\overline{\mathrm{H}}\subseteq\mathrm{E}\cap\mathrm{A}_{!}$.
 
\begin{prop}
Let $f\in\overline{\mathrm{H}}$. Then $f\in\mathrm{E}$.
\end{prop} 
 
 \begin{proof}
Consider a fibre-product diagram

\begin{displaymath}
\xymatrix{
X'\ar[d]^{f'}\ar[r]^{g'} & X\ar[d]^{f}\\
Y'\ar[r]^{g} & Y
}
\end{displaymath}
with  $f\in\overline{\mathrm{H}}$. 

By passing to stalks, we may assume that $Y'=\{y\}$ where $y\in Y$, and so $X'=f^{-1}(y)$. Base change then amounts to proving
$$\mathbb{R}(f_{!}\mathcal{F})_{y}\cong\mathbb{R}\Gamma_{c}(f^{-1}(y),\mathcal{F}|_{f^{-1}(y)})$$
for each $\mathcal{F}\in\mathbf{Ch}(\mathrm{Shv}(X,\mathpzc{E}))$. However we have 
$$(\mathbb{R}f_{!}\mathcal{F})_{y}\cong((f_{!})_{\infty}\mathcal{F})_{y}\cong\mathbb{R}\Gamma_{c}(f^{-1}(y),(i_{f^{-1}(y)}^{-1})_{\infty}\mathcal{F})\cong\mathbb{R}\Gamma_{c}(f^{-1}(y),i_{f^{-1}(y)}^{-1})\mathcal{F})$$
where $i_{f^{-1}(y)}:f^{-1}(y)\rightarrow X$ is the inclusion, and we have used that $f^{-1}(y)$ is $\mathpzc{E}$-hypercomplete. This suffices to prove the claim. 
\end{proof}

\begin{prop}
Let $f:X\rightarrow Y$ be a map in $\mathrm{H}$. Then $f_{!}$ has a right adjoint. 
\end{prop}

\begin{proof}
Again by passing to stalks we may assume that $Y=\{y\}$ is a point. But then we are just replacing $f$ by the map $f^{-1}(y)\rightarrow \{y\}$, and this follows from the claim for $\mathbf{Shv}(f^{-1}(y),\mathbf{Ch}(\mathpzc{E}))$ since $X$ is $\mathpzc{E}$-complete. 
\end{proof}

This proves the following.

\begin{thm}
The assignment $\mathbf{Ch}(\mathrm{Shv}(-,\mathpzc{E})):\mathrm{Corr}(\mathrm{LCHaus},\mathrm{H})^{op}\rightarrow\mathbf{Cat}_{\infty}$ as defined above is a six-functor formalism.
\end{thm}


\subsubsection{Ringed Spaces}

Finally we generalise this to ringed spaces

\begin{defn}
Let $X$ be a locally compact Hausdorff space, and let $\mathcal{O}_{X}\in\mathrm{Alg}_{\mathfrak{Comm}}(\mathrm{Ch}(\mathrm{Shv}(X,\mathpzc{E})))$. 
\begin{enumerate}
\item
A pair $(X,\mathcal{O}_{X})$ is called a \textit{dg-}$\mathpzc{E}$-\textit{ringed space}.
\item
A $dg-\mathpzc{E}$ ringed space $(X,\mathcal{O}_{X})$ is said to be a $\mathpzc{E}$-\textit{ringed space} if $\mathcal{O}_{X}$ is concentrated in degree $0$.
\end{enumerate}
\end{defn}

Let $f:X\rightarrow Y$ be a map of locally compact Hausdorff spaces, and $\mathcal{R}$ an object of $\mathrm{Alg}_{\mathfrak{Comm}}(\mathrm{Ch}(\mathrm{Shv}(Y;\mathpzc{E})))$. Then $f^{-1}\mathcal{R}$ has a natural structure as an object of $\mathrm{Alg}_{\mathfrak{Comm}}(\mathrm{Ch}(\mathrm{Shv}(X,\mathpzc{E})))$. 

\begin{defn}
A morphism $f=(\overline{f},f^{\#}):(X,\mathcal{O}_{X})\rightarrow (Y,\mathcal{O}_{Y})$ of $dg-\mathpzc{E}$ locally ringed spaces is a pair $(\overline{f},f^{\#})$ where $\overline{f}:X\rightarrow Y$ is a map of spaces, and $f^{\#}:f^{-1}\mathcal{O}_{Y}\rightarrow\mathcal{O}_{X}$ is a map in $\mathrm{Alg}_{\mathfrak{Comm}}(\mathrm{Ch}(\mathrm{Shv}(X,\mathpzc{E})))$.
\end{defn}

A map $f:(X,\mathcal{O}_{X})\rightarrow (Y,\mathcal{O}_{Y})$ gives rise to an adjunction
$$\adj{f^{*}}{{}_{\mathcal{O}_{Y}}\mathrm{Mod}(\mathrm{Ch}(\mathrm{Shv}(Y;\mathpzc{E})))}{{}_{\mathcal{O}_{X}}\mathrm{Mod}(\mathrm{Ch}(\mathrm{Shv}(X,\mathpzc{E})))}{f_{*}}$$
which is in fact a Quillen adjunction for the flat model structures. By \cite{scholze2022six} Remark 3.13, the functor
$$f_{!}:\mathbf{Ch}(\mathrm{Shv}(X,\mathpzc{E}))\rightarrow\mathbf{Ch}(\mathrm{Shv}(Y;\mathpzc{E}))$$
is $\mathbf{Ch}(\mathrm{Shv}(Y;\mathpzc{E}))$-linear, where $\mathbf{Ch}(\mathrm{Shv}(Y;\mathpzc{E}))$ acts on $\mathbf{Ch}(\mathrm{Shv}(X,\mathpzc{E}))$ by pullback. Together with the projection formula, this implies that given a map $f:(X,\mathcal{O}_{X})\rightarrow (Y,\mathcal{O}_{Y})$ of ringed spaces, $f_{!}$ naturally induces a functor
$$f_{!}:\mathcal{O}_{X}\mathrm{Mod}(\mathbf{Ch}(\mathrm{Shv}(X,\mathpzc{E})))\rightarrow\mathcal{O}_{Y}\mathrm{Mod}(\mathbf{Ch}(\mathrm{Shv}(Y;\mathpzc{E})))$$
(whose underlying functor $\mathbf{Ch}(\mathrm{Shv}(X,\mathpzc{E}))\rightarrow\mathbf{Ch}(\mathrm{Shv}(Y;\mathpzc{E}))$ coincides with the shriek push-forward along $f$ defined above). Let

\begin{displaymath}
\xymatrix{
(X',\mathcal{O}_{X'})\ar[d]^{f'}\ar[r]^{g'} & (X,\mathcal{O}_{X})\ar[d]^{f}\\
(Y',\mathcal{O}_{Y'})\ar[r]^{g} & (Y,\mathcal{O}_{Y})
}
\end{displaymath}
be a fibre-product diagram of $dg-\mathpzc{E}$-ringed spaces. Again there is a natural transformation
 $$g^{*}f_{!}\rightarrow f_{!}'(g')^{*}$$
Indeed there is a natural map
$$g^{-1}f_{!}\rightarrow f_{!}'(g')^{*}$$
which is $g^{-1}(\mathcal{O}_{Y})$-linear. Thus this induces a natural map
  $$g^{*}f_{!}\rightarrow f_{!}'(g')^{*}$$

%

\begin{defn}
We say that $f:(X\,\mathcal{O}_{X})\rightarrow(Y,\mathcal{O}_{Y})$ is \textit{flat} if $\mathcal{O}_{X,x}$ is flat as a $\mathcal{O}_{Y,f(x)}$-module for each $x\in X$. We say that $f$ is \textit{strongly flat} if $\mathcal{O}_{X,x}$ is a filtered colimit of free $\mathcal{O}_{Y,f(x)}$-modules on projective objects of $\mathpzc{E}$. 
\end{defn}

\begin{defn}
A map $f:(X,\mathcal{O}_{X})\rightarrow (Y,\mathcal{O}_{Y})$ is said to be \textit{a stalk-wise projective base-change map} if for any $\mathcal{F}\in{}_{\mathcal{O}_{X}}\mathrm{Mod}$, any projective $P\in\mathpzc{E}$, and any $y\in Y$, we have that the natural maps
$$(f_{!}\mathcal{F})_{y}\rightarrow\Gamma_{c}(f^{-1}(y),\mathcal{F}|_{f^{-1}(y)})$$
$$P\otimes\mathbb{R}\Gamma_{c}(f^{-1}(y),\mathcal{F}|_{f{-1}(y)})\rightarrow\mathbb{R}\Gamma_{c}(f^{-1}(y),P\otimes \mathcal{F}|_{f{-1}(y)})$$
are equivalences.
\end{defn}

\begin{example}
If $f=(\overline{f},f^{\#}):(X,\mathcal{O}_{X})\rightarrow (Y,\mathcal{O}_{Y})$ is such that the underlying map $\overline{f}:X\rightarrow Y$ is universally base-change, then $f$ is a stalk-wise projective base-change map. Indeed using the projection formula we have
$$P\otimes\mathbb{R}\Gamma_{c}(f^{-1}(y),\mathcal{F}|_{f^{-1}(y)})\cong P\otimes \mathbb{R}pt_{!}(\mathcal{F}|_{f^{-1}(g(y'))})\cong \mathbb{R}pt_{!}(pt^{-1}(P)\otimes\mathcal{F}|_{f^{-1}(g(y'))})\cong\mathbb{R}\Gamma_{c}(f^{-1}(y),P\otimes \mathcal{F}|_{f{-1}(y)})$$
\end{example}

Exactly as in \cite{spaltenstein} Proposition 6.20 we have the following.

\begin{lem}
Let $f:(X,\mathcal{O}_{X})\rightarrow (Y,\mathcal{O}_{Y})$, $g:(Y',\mathcal{O}_{Y'})\rightarrow (Y,\mathcal{O}_{Y})$ be a a stalk-wise projective base-change map of $\mathpzc{E}$-ringed spaces which is universally left-adjointable. If $g$ is strongly flat, then the map
  $$g^{*}f_{!}\rightarrow f_{!}'(g')^{*}$$
  is an equivalence. 
\end{lem}

\begin{proof}
By passing to stalks, we reduce to showing that 
$$\mathcal{O}_{Y',y'}\otimes_{\mathcal{O}_{Y,g(y')}}^{\mathbb{L}}\mathbb{R}\Gamma_{c}(f^{-1}(y),\mathcal{F}|_{f^{-1}(g(y'))})\cong\mathbb{R}\Gamma_{c}(f^{-1}(g(y')),\mathcal{O}_{Y',y'}\otimes_{\mathcal{O}_{Y,g(y')}}^{\mathbb{L}}\mathcal{F}|_{f^{-1}(y)})$$
By assumption $\overline{f}$ is universally $!$-adjointable, so $\mathbb{R}\Gamma_{c}$ commutes with colimits, and we may assume that $\mathcal{O}_{Y',y'}\cong \mathcal{O}_{Y,g(y')}\otimes P$. Now this follows from the assumption that $f$ is a stalk-wise projective base-change map. 
\end{proof}

Note that this does not quite gives us a three-functor formalism as in a span
\begin{displaymath}
\xymatrix{
(X',\mathcal{O}_{X'})\ar[d]^{f'}\ar[r]^{g'} & (X,\mathcal{O}_{X})\ar[d]^{f}\\
(Y',\mathcal{O}_{Y'})\ar[r]^{g} & (Y,\mathcal{O}_{Y})
}
\end{displaymath}
we have restrictions on both $f$ and $g$. However all the usual coherences for $\mathbb{R}f_{*},\mathbb{L}f^{*},\mathbb{R}f_{!},f^{!},\otimes^{\mathbb{L}},\mathbb{R}\underline{\mathpzc{Hom}}$ can also be deduced. 

%

Let us finally make a precise connection with condition 6.14 (2) of \cite{spaltenstein}

\begin{prop}
Let $I_{\bullet}$ be an injective object of ${}_{\mathcal{O}_{X}}\mathrm{Mod}(\mathrm{Ch}(\mathrm{Shv}(X,\mathpzc{E})))$. Then $I_{\bullet}$ is soft, and in particular $c$-soft.
\end{prop}

\begin{proof}
Since injective complexes of $\mathcal{O}_{X}$-modules are in particular injective complexes of sheaves, it suffices to prove this in the case that $\mathcal{O}_{X}=k_{X}$. In this case, it follows from the fact that $I_{\bullet}$ has the right-lifting property against $P_{U}\rightarrow P_{V}$ for any object $P$, and any $U\subset V$ with $U,V$ open. 
\end{proof}

Now the following can be proven exactly as in \cite{spaltenstein} Section 6.

\begin{prop}
Let $f=(\overline{f},f^{\#}):(X,\mathcal{O}_{X})\rightarrow (Y,\mathcal{O}_{Y})$ be such that for each $y\in Y$, $(f^{-1}(y),\mathcal{O}_{X}|_{f^{-1}(y)})$ satisfy the following: for every acyclic complex of $c$-soft sheaves $\mathcal{L}_{\bullet}\in{}_{\mathcal{O}_{f^{-1}(y)}}\mathrm{Mod}$, and each $d_{i}:\mathcal{L}_{i}\rightarrow\mathcal{L}_{i-1}$, $\mathrm{Ker} d_{i}$ is $c$-soft. Then $f$ is universally $!$-adjointable and a stalk-wise projective base-change map.
\end{prop}

\bibliographystyle{amsalpha}
\bibliography{FlatModelv3.bib}

\providecommand{\bysame}{\leavevmode\hbox to3em{\hrulefill}\thinspace}
\providecommand{\MR}{\relax\ifhmode\unskip\space\fi MR }
\providecommand{\MRhref}[2]{%
  \href{http://www.ams.org/mathscinet-getitem?mr=#1}{#2}
}
\providecommand{\href}[2]{#2}
\begin{thebibliography}{HKvRW21}

\bibitem[AR94]{adamek1994locally}
Ji\v{r}\'{\i} Ad\'{a}mek and Ji\v{r}\'{\i} Rosick\'{y}, \emph{Locally
  presentable and accessible categories}, London Mathematical Society Lecture
  Note Series, vol. 189, Cambridge University Press, Cambridge, 1994.
  \MR{1294136}

\bibitem[B\"10]{Buehler}
Theo B\"{u}hler, \emph{Exact categories}, Expo. Math. \textbf{28} (2010),
  no.~1, 1--69. \MR{2606234}

\bibitem[BC13]{bazzoni2013one}
Silvana Bazzoni and Septimiu Crivei, \emph{One-sided exact categories}, J. Pure
  Appl. Algebra \textbf{217} (2013), no.~2, 377--391. \MR{2969259}

\bibitem[Bec14]{MR3161097}
Hanno Becker, \emph{Models for singularity categories}, Adv. Math. \textbf{254}
  (2014), 187--232. \MR{3161097}

\bibitem[BG16]{barwick2016note}
Clark Barwick and Saul Glasman, \emph{A note on stable recollements}, arXiv
  preprint arXiv:1607.02064 (2016).

\bibitem[BR14]{MR3137847}
David Barnes and Constanze Roitzheim, \emph{Stable left and right {B}ousfield
  localisations}, Glasg. Math. J. \textbf{56} (2014), no.~1, 13--42.
  \MR{3137847}

\bibitem[CCn04]{castiglioni2004cosimplicial}
Jos\'{e}~Luis Castiglioni and Guillermo Corti\~{n}as, \emph{Cosimplicial versus
  {DG}-rings: a version of the {D}old-{K}an correspondence}, J. Pure Appl.
  Algebra \textbf{191} (2004), no.~1-2, 119--142. \MR{2048310}

\bibitem[DBPP19]{MR3913977}
Gennaro Di~Brino, Damjan Pi\v{s}talo, and Norbert Poncin, \emph{Homotopical
  algebraic context over differential operators}, J. Homotopy Relat. Struct.
  \textbf{14} (2019), no.~1, 293--347. \MR{3913977}

\bibitem[EE05]{MR2139915}
Edgar Enochs and Sergio Estrada, \emph{Relative homological algebra in the
  category of quasi-coherent sheaves}, Adv. Math. \textbf{194} (2005), no.~2,
  284--295. \MR{2139915}

\bibitem[EGO17]{estrada2017pure}
Sergio Estrada, James Gillespie, and Sinem Odaba\c{s}i, \emph{Pure exact
  structures and the pure derived category of a scheme}, Math. Proc. Cambridge
  Philos. Soc. \textbf{163} (2017), no.~2, 251--264. \MR{3682629}

\bibitem[EGO23]{estrada2023k}
Sergio Estrada, James Gillespie, and Sinem Odaba{\c{s}}{\i}, \emph{K-flatness
  in grothendieck categories: Application to quasi-coherent sheaves}, arXiv
  preprint arXiv:2306.04816 (2023).

\bibitem[Est15]{estrada2014derived}
Sergio Estrada, \emph{The derived category of quasi-coherent modules on an
  {A}rtin stack via model structures}, Int. Math. Res. Not. IMRN (2015),
  no.~15, 6411--6432. \MR{3384483}

\bibitem[Gil]{Gillespiebook}
James Gillespie, \emph{Abelian model category theory}.

\bibitem[Gil04]{Gillespie2}
\bysame, \emph{The flat model structure on {${\rm Ch}(R)$}}, Trans. Amer. Math.
  Soc. \textbf{356} (2004), no.~8, 3369--3390. \MR{2052954}

\bibitem[Gil06]{gillespie2006flat}
\bysame, \emph{The flat model structure on complexes of sheaves}, Transactions
  of the American Mathematical Society \textbf{358} (2006), no.~7, 2855--2874.

\bibitem[Gil11]{gillespie}
\bysame, \emph{Model structures on exact categories}, J. Pure Appl. Algebra
  \textbf{215} (2011), no.~12, 2892--2902. \MR{2811572}

\bibitem[Gil16a]{gillespie2016derived}
\bysame, \emph{The derived category with respect to a generator}, Ann. Mat.
  Pura Appl. (4) \textbf{195} (2016), no.~2, 371--402. \MR{3476679}

\bibitem[Gil16b]{gillespie2016exact}
\bysame, \emph{Exact model structures and recollements}, J. Algebra
  \textbf{458} (2016), 265--306. \MR{3500779}

\bibitem[Gil16c]{MR3459032}
\bysame, \emph{Gorenstein complexes and recollements from cotorsion pairs},
  Adv. Math. \textbf{291} (2016), 859--911. \MR{3459032}

\bibitem[Gro10]{Gross2010VectorBA}
Philipp Gross, \emph{Vector bundles as generators on schemes and stacks}, 2010.

\bibitem[Har10]{harper2010homotopy}
John~E. Harper, \emph{Homotopy theory of modules over operads and
  non-{$\Sigma$} operads in monoidal model categories}, J. Pure Appl. Algebra
  \textbf{214} (2010), no.~8, 1407--1434. \MR{2593672}

\bibitem[HKvRW21]{henrard2021left}
Ruben Henrard, Sondre Kvamme, Adam-Christiaan van Roosmalen, and Sven-Ake
  Wegner, \emph{The left heart and exact hull of an additive regular category},
  arXiv preprint arXiv:2105.11483 (2021).

\bibitem[Hov02]{hovey}
Mark Hovey, \emph{Cotorsion pairs, model category structures, and
  representation theory}, Math. Z. \textbf{241} (2002), no.~3, 553--592.
  \MR{1938704}

\bibitem[HS14]{MR3166360}
Kathryn Hess and Brooke Shipley, \emph{The homotopy theory of coalgebras over a
  comonad}, Proc. Lond. Math. Soc. (3) \textbf{108} (2014), no.~2, 484--516.
  \MR{3166360}

\bibitem[HTT08]{MR2357361}
Ryoshi Hotta, Kiyoshi Takeuchi, and Toshiyuki Tanisaki, \emph{{$D$}-modules,
  perverse sheaves, and representation theory}, japanese ed., Progress in
  Mathematics, vol. 236, Birkh\"{a}user Boston, Inc., Boston, MA, 2008.
  \MR{2357361}

\bibitem[Kel16]{kelly2016homotopy}
Jack Kelly, \emph{Homotopy in exact categories}, To appear in Memoirs of the
  AMS. arXiv preprint arXiv:1603.06557 (2016).

\bibitem[Kel19]{kelly2019koszul}
\bysame, \emph{Koszul duality in exact categories}, arXiv preprint
  arXiv:1905.10102 (2019).

\bibitem[KKM21]{kelly2021analytic}
Jack Kelly, Kobi Kremnizer, and Devarshi Mukherjee, \emph{Analytic
  hochschild-kostant-rosenberg theorem}, arXiv preprint arXiv:2111.03502
  (2021).

\bibitem[Kra10]{MR2681709}
Henning Krause, \emph{Localization theory for triangulated categories},
  Triangulated categories, London Math. Soc. Lecture Note Ser., vol. 375,
  Cambridge Univ. Press, Cambridge, 2010, pp.~161--235. \MR{2681709}

\bibitem[Kra12]{krause2012approximations}
\bysame, \emph{Approximations and adjoints in homotopy categories}, Math. Ann.
  \textbf{353} (2012), no.~3, 765--781. \MR{2923949}

\bibitem[Lur]{SAG}
Jacob Lurie, \emph{Spectral algebraic geometry}.

\bibitem[Lur09]{lurie2006higher}
\bysame, \emph{Higher topos theory}, Annals of Mathematics Studies, vol. 170,
  Princeton University Press, Princeton, NJ, 2009. \MR{2522659}

\bibitem[Man22]{mann2022p}
Lucas Mann, \emph{A $ p $-adic 6-functor formalism in rigid-analytic geometry},
  arXiv preprint arXiv:2206.02022 (2022).

\bibitem[{nLa}23]{nlab:adjointtriple}
{nLab authors}, \emph{adjoint triple},
  https://ncatlab.org/nlab/show/adjoint+triple, November 2023,
  \href{https://ncatlab.org/nlab/revision/adjoint+triple/43}{Revision 43}.

\bibitem[Pos23]{positselski2023locally}
Leonid Positselski, \emph{Locally coherent exact categories}, arXiv preprint
  arXiv:2311.02418 (2023).

\bibitem[PS00]{reconstruction}
Fabienne Prosmans and Jean-Pierre Schneiders, \emph{A topological
  reconstruction theorem for {$ D^\infty$}-modules}, Duke Math. J. \textbf{102}
  (2000), no.~1, 39--86. \MR{1741777}

\bibitem[Rie14]{Riehl}
Emily Riehl, \emph{Categorical homotopy theory}, New Mathematical Monographs,
  vol.~24, Cambridge University Press, Cambridge, 2014. \MR{3221774}

\bibitem[Sch99]{qacs}
Jean-Pierre Schneiders, \emph{Quasi-abelian categories and sheaves}, M\'{e}m.
  Soc. Math. Fr. (N.S.) (1999), no.~76, vi+134. \MR{1779315}

\bibitem[Sch22]{scholze2022six}
Peter Scholze, \emph{Six-functor formalisms}, URl: https://people. mpim-bonn.
  mpg. de/scholze/SixFunctors. pdf (2022).

\bibitem[Spa88]{spaltenstein}
N.~Spaltenstein, \emph{Resolutions of unbounded complexes}, Compositio Math.
  \textbf{65} (1988), no.~2, 121--154. \MR{932640}

\bibitem[{Sta}23]{stacks-project}
The {Stacks project authors}, \emph{The stacks project},
  https://stacks.math.columbia.edu, 2023.

\bibitem[Sv11]{saorin2011exact}
Manuel Saor\'{\i}n and Jan \v{S}t'ov\'{\i}\v{c}ek, \emph{On exact categories
  and applications to triangulated adjoints and model structures}, Adv. Math.
  \textbf{228} (2011), no.~2, 968--1007. \MR{2822215}

\bibitem[TV08]{toen2004homotopical}
Bertrand To\"{e}n and Gabriele Vezzosi, \emph{Homotopical algebraic geometry.
  {II}. {G}eometric stacks and applications}, Mem. Amer. Math. Soc.
  \textbf{193} (2008), no.~902, x+224. \MR{2394633}

\bibitem[Vol21]{volpe2021six}
Marco Volpe, \emph{Six functor formalism for sheaves with non-presentable
  coefficients}, arXiv preprint arXiv:2110.10212 (2021).

\bibitem[\v{S}13]{vst2012exact}
Jan \v{S}t'ov\'{\i}\v{c}ek, \emph{Exact model categories, approximation theory,
  and cohomology of quasi-coherent sheaves}, Advances in representation theory
  of algebras, EMS Ser. Congr. Rep., Eur. Math. Soc., Z\"{u}rich, 2013,
  pp.~297--367. \MR{3220541}

\bibitem[Whi17]{white2017model}
David White, \emph{Model structures on commutative monoids in general model
  categories}, J. Pure Appl. Algebra \textbf{221} (2017), no.~12, 3124--3168.
  \MR{3666740}

\bibitem[YD15]{yang2014question}
Xiaoyan Yang and Nanqing Ding, \emph{On a question of {G}illespie}, Forum Math.
  \textbf{27} (2015), no.~6, 3205--3231. \MR{3420339}

\end{thebibliography}

\end{document}